\documentclass[a4paper]{article}
\usepackage{geometry}
\usepackage{amsmath}
\usepackage{amssymb}
\usepackage{latexsym}
\usepackage{amsfonts}
\usepackage{indentfirst}
\usepackage{graphicx}
\usepackage[colorlinks=true]{hyperref}
\usepackage{mathrsfs}
\usepackage{chngcntr}
\usepackage{color}

\newtheorem{Thm}{Theorem}[section]
\newtheorem{Cor}{Corollary}[section]
\newtheorem{Lem}{Lemma}[section]
\newtheorem{Pro}{Proposition}[section]

\newtheorem{Rek}{Remark}[section]
\newtheorem{Def}{Definition}[section]

\newcommand{\R}{\mathbb{R}}

\numberwithin{equation}{section}

\newenvironment{proof}{\medskip\par\noindent{\bf Proof\/}:\quad}{\qquad
	\raisebox{-0.5mm}{\rule{1.5mm}{1mm}}\vspace{6pt}}
\providecommand{\qed}{\hfill\ensuremath{\square}}
\begin{document}
	
	\title{Normalized solutions for a class of fractional Choquard equations with mixed nonlinearities}
	
	\author{Shaoxiong Chen \quad Zhipeng Yang\thanks{Corresponding author: yangzhipeng326@163.com} \quad Xi Zhang\\[2pt]
		\small Department of Mathematics, Yunnan Normal University, Kunming, China\\
	}
	
	\date{}
	\maketitle
	
	\begin{abstract}
		In this paper we study the following fractional Choquard equation with mixed nonlinearities:
		\[
		\left\{
		\begin{array}{l}
			(-\Delta)^s u = \lambda u + \alpha \left( I_\mu * |u|^q \right) |u|^{q-2} u + \left( I_\mu * |u|^p \right) |u|^{p-2} u, \quad x \in \mathbb{R}^N, \\[4pt]
			\displaystyle \int_{\mathbb{R}^N} |u|^2 \,\mathrm{d}x = c^2 > 0.
		\end{array}
		\right.
		\]
		Here $N > 2s$, $s \in (0,1)$, $\mu \in (0, N)$, and the exponents satisfy
		\[
		\frac{2N - \mu}{N} < q < p < \frac{2N - \mu}{N - 2s},
		\]
		while $\alpha > 0$ is a sufficiently small parameter, $\lambda \in \mathbb{R}$ is the Lagrange multiplier associated with the mass constraint, and $I_\mu$ denotes the Riesz potential. We establish existence and multiplicity results for normalized solutions and, in addition, prove the existence of ground state normalized solutions for $\alpha$ in a suitable range.

		
	\end{abstract}

%\tableofcontents

\section{Introduction and main results}

In this paper, we aim to study the existence of multiple normalized solutions for the nonlinear fractional Choquard equation
\begin{equation}\label{eq1.1}
	\left\{
	\begin{array}{l}
		(-\Delta)^{s} u = \lambda u + \alpha\bigl(I_\mu * |u|^q\bigr)|u|^{q-2} u + \bigl(I_\mu * |u|^p\bigr)|u|^{p-2} u \quad \text{in } \mathbb{R}^N,\\[1mm]
		\displaystyle\int_{\mathbb{R}^N} |u|^2\,dx = c^2,
	\end{array}
	\right.
\end{equation}
where \(s\in(0,1)\), \(N>2s\), \(0<\mu<N\), \(c>0\), and
\[
\frac{2N-\mu}{N}<q<p< 2_{\mu,s}^*:=\frac{2N-\mu}{N-2s}.
\]
Here \(\alpha>0\) is a suitably small real parameter, \(\lambda\in\mathbb{R}\) is the Lagrange multiplier associated with the mass constraint, and \(I_\mu\) is the Riesz potential. More precisely, for each \(x\in\mathbb{R}^N\setminus\{0\}\),
\[
I_\mu(x) = \frac{A_{N,\mu}}{|x|^\mu}, 
\qquad 
A_{N,\mu} = \frac{\Gamma\bigl(\frac{\mu}{2}\bigr)}{2^{\,N-\mu}\pi^{N/2}\Gamma\bigl(\frac{N-\mu}{2}\bigr)},
\]
and
\[
(I_\mu * |u|^t)(x)
=\int_{\mathbb{R}^N}\frac{|u(y)|^t}{|x-y|^\mu}\,dy,
\qquad t\in\{p,q\}.
\]
Alternatively, the fractional Laplacian can be written as
\[
\begin{aligned}
	(-\Delta)^s u(x) 
	&= C_{N,s}\,\mathrm{P.V.}\int_{\mathbb{R}^N}\frac{u(x)-u(y)}{|x-y|^{N+2s}}\,dy \\
	&= -\frac{C_{N,s}}{2}\int_{\mathbb{R}^N}\frac{u(x+y)+u(x-y)-2u(x)}{|y|^{N+2s}}\,dy,
	\qquad u\in\mathcal{S}(\mathbb{R}^N).
\end{aligned}
\]
where \(\mathcal{S}(\mathbb{R}^N)\) denotes the Schwartz space of rapidly decaying smooth functions, P.V. stands for the principal value, and \(C_{N,s}>0\) is a normalization constant.

As a nonlocal counterpart of the classical Laplacian in the framework of nonlinear Schrödinger equations, the operator \((-\Delta)^s\) with \(s\in(0,1)\) appearing in \eqref{eq1.1} was introduced by Laskin \cite{Laskin2000phys} in the context of fractional quantum mechanics, where Brownian trajectories are replaced by Lévy flights in Feynman's path integral formalism. The fractional Laplacian arises naturally in several theoretical and applied contexts, including biology, chemistry, and finance; see, for instance, \cite{2011changADV,2004contchapman,2013Kaymathphys,2000metzlerphys,2007Slivestreappl} and the references therein.

From a physical point of view, normalized solutions, namely solutions with prescribed \(L^2\)-norm, play a central role in nonlinear dispersive models. In the last two decades, normalized solutions of nonlinear elliptic and Schrödinger-type equations have attracted considerable attention, mainly because the \(L^2\)-norm is conserved along the associated evolution flow and because variational characterizations of such solutions are closely related to their orbital stability or instability. A systematic study of normalized solutions was initiated by Jeanjean in \cite{1997jeananal}, where he considered semilinear elliptic equations under the mass constraint
\[
S_c = \Bigl\{u\in H^1(\mathbb{R}^N): \int_{\mathbb{R}^N}|u|^2\,dx = c^2\Bigr\}.
\]
More precisely, Jeanjean studied the equation
\begin{equation}\label{eq1.2}
	\left\{
	\begin{array}{l}
		-\Delta u = \lambda u + |u|^{p-2}u \quad \text{in } \mathbb{R}^N,\quad u\in H^1(\mathbb{R}^N),\\[1mm]
		\displaystyle\int_{\mathbb{R}^N} |u|^2\,dx = c^2,
	\end{array}
	\right.
\end{equation}
where \(\lambda\in\mathbb{R}\) appears as a Lagrange multiplier. His approach is based on a suitable Pohozaev-type manifold and on the construction of bounded Palais--Smale sequences, leading to existence results for normalized solutions.

Later, Soave \cite{2020soavejde} investigated the combined effect of \(L^2\)-subcritical, \(L^2\)-critical, and \(L^2\)-supercritical power nonlinearities, which drastically affects the geometry of the energy functional. He considered, in particular, the problem
\begin{equation}\label{eq1.3}
	-\Delta u = \lambda u + |u|^{p-2}u + \alpha |u|^{q-2}u \quad \text{in } \mathbb{R}^N,
	\qquad \int_{\mathbb{R}^N} |u|^2\,dx = c^2,
\end{equation}
where \(2 < q \leq 2 + \frac{4}{N} \leq p < 2^* = \frac{2N}{N-2}\). Here \(q\) is \(L^2\)-subcritical or \(L^2\)-critical, while \(p\) is subcritical in the Sobolev sense. Among other results, Soave proved the existence of a ground state solution when \(2 < q < 2 + \frac{4}{N}\) and \(2 + \frac{4}{N} < p < 2^*\). In the same paper, the case \(2 < q < 2^* = p\) was also addressed: if \(q\in(2,2+\frac{4}{N})\), a ground state with negative energy was obtained, while for \(q\in(2+\frac{4}{N},2^*)\) a mountain-pass type solution with positive energy was constructed, together with conditions for the existence and nonexistence of normalized solutions when \(\lambda<0\). Subsequent extensions of \eqref{eq1.3} were obtained by Jeanjean--Jendrej--Le--Visciglia \cite{2022jeanmathappl} and Jeanjean--Le \cite{2022jeanmathann}, where several open questions raised in \cite{2020soavejde} were answered.

Equation \eqref{eq1.1} is of Choquard type, due to the presence of the nonlocal convolution terms \((I_\mu * |u|^q)|u|^{q-2}u\) and \((I_\mu * |u|^p)|u|^{p-2}u\).
In the fractional setting, Luo and Zhang \cite{2020luopde} studied the following fractional Schrödinger equation with combined local nonlinearities:
\begin{equation}\label{eq1.4}
	\left\{
	\begin{array}{l}
		(-\Delta)^s u = \lambda u + \mu |u|^{q-2}u + |u|^{p-2}u \quad \text{in } \mathbb{R}^N,\\[1mm]
		\displaystyle\int_{\mathbb{R}^N} |u|^2\,dx = a^2,\quad u\in H^s(\mathbb{R}^N),
	\end{array}
	\right.
\end{equation}
where \(s\in(0,1)\), \(2<q<p<2_s^*:=\frac{2N}{N-2s}\), and \(\mu>0\). They obtained existence and nonexistence results for normalized solutions of \eqref{eq1.4} in the case of combined subcritical nonlinearities. Later, Li and Zou \cite{2022zouadv} and Zhen and Zhang \cite{2022zhenRev} considered the critical case \(p=2_s^*\) and proved the existence and multiplicity of normalized solutions. For further results on normalized solutions of fractional Schrödinger equations we refer, for instance, to \cite{2021AppoJDE,2021CingolaniNonlinearity} and the references therein. Related results for fractional Schrödinger systems can be found in \cite{2022zuoAnalmath,2022zuomediterr,2022liuiMinimax}.

Yang \cite{2020YangPhys} considered the mixed local-nonlocal problem
\begin{equation}\label{eq1.5}
	\left\{
	\begin{array}{l}
		(-\Delta)^\sigma u = \lambda u + |u|^{q-2}u + \mu\bigl(I_\alpha * |u|^p\bigr)|u|^{p-2}u \quad \text{in } \mathbb{R}^N,\\[1mm]
		\displaystyle\int_{\mathbb{R}^N} |u|^2\,dx = a^2,
	\end{array}
	\right.
\end{equation}
where \(N\geq 2\), \(\sigma\in(0,1)\), \(\alpha\in(0,N)\), \(q\in\bigl(2+\frac{4\sigma}{N},\frac{2N}{N-2\sigma}\bigr]\), \(p\in\bigl[1+\frac{2\sigma+\alpha}{N},\frac{N+\alpha}{N-2\sigma}\bigr)\), \(a,\mu>0\). By a refined min--max scheme, it was shown that for suitable choices of the parameters the problem admits a mountain-pass type normalized solution \(\hat u_\mu\) associated with some \(\hat\lambda<0\). Moreover, \(\hat u_\mu\) is a ground state whenever \(p\leq \frac{q}{2}+\frac{\alpha}{N}\).

The HLS upper critical situation $p=2^*_{\alpha,s}$ has also attracted considerable attention. Lan, He and Meng \cite{LanHeMeng2023} investigated a critical fractional Choquard equation perturbed by a nonlocal term and established the existence of normalized solutions by combining sharp HLS inequalities with concentration-compactness arguments.
Yu et al.\ \cite{2023DCDSYU} investigated
\begin{equation}\label{eq1.6}
	\left\{
	\begin{array}{l}
		(-\Delta)^s u = \lambda u + \gamma (I_{\alpha}*|u|^{1+\frac{\alpha}{N}})|u|^{\frac{\alpha}{N}-1}u + \mu |u|^{q-2}u \quad \text{in } \mathbb{R}^N,\\[1mm]
		\displaystyle\int_{\mathbb{R}^N}|u|^2\,dx = a^2,
	\end{array}
	\right.
\end{equation}
where \(N\geq 3\), \(s\in(0,1)\), \(\alpha\in(0,N)\), \(a,\gamma,\mu>0\), and \(2<q\leq 2_s^*:=\frac{2N}{N-2s}\). They established nonexistence and existence results, as well as symmetry properties for normalized ground states. In the \(L^2\)-subcritical regime \(2<q<2+\frac{4s}{N}\), the existence of radially symmetric normalized ground states was proved without additional constraints. In the \(L^2\)-supercritical regime \(2+\frac{4s}{N}<q<2_s^*\), the authors constructed a homotopy-stable family of subsets to obtain a Palais--Smale sequence whose compactness yields normalized ground states. In the critical case \(q=2_s^*\), a subcritical approximation combined with detailed asymptotic analysis leads again to the existence of normalized ground states.

More recently, Chen et al. \cite{2025chenNoDEA} considered the fractional Choquard equation with external potential
\begin{equation}\label{eq1.7}
	\left\{
	\begin{array}{l}
		(-\Delta)^{s} u + V(\varepsilon x) u = \lambda u + \bigl(I_\alpha * |u|^q\bigr)|u|^{q-2}u + \bigl(I_\alpha * |u|^p\bigr)|u|^{p-2}u \quad \text{in } \mathbb{R}^N,\\[1mm]
		\displaystyle\int_{\mathbb{R}^N} |u|^2\,dx = a^2,
	\end{array}
	\right.
\end{equation}
and, by means of Lusternik--Schnirelmann category theory, proved the existence of normalized solutions and showed that the number of such solutions is related to the topology of the set where the potential \(V(x)\) attains its minimum. Later, they also \cite{ChenYangZhangPotentials2025} studied more general weighted Hartree nonlinearities of the form
	\[
	(-\Delta)^s u + V(x)u = \lambda u
	+ f(x)\bigl(I_\alpha*(f|u|^q)\bigr)|u|^{q-2}u
	+ g(x)\bigl(I_\alpha*(g|u|^p)\bigr)|u|^{p-2}u,
	\]
and established existence results for normalized solutions on the mass constraint by combining refined compactness and a careful use of the HLS inequality.

Motivated by the preceding developments and building mainly on the works
\cite{2025chenNoDEA,ChenYangZhangPotentials2025,2020soavejde}, we now turn to problem \eqref{eq1.1} and
address the existence of multiple normalized solutions. A key tool in our
analysis is the Gagliardo–Nirenberg inequality, and the exponent
\[
2+\frac{2s-\mu}{N}
\]
plays the role of the $L^2$–critical threshold for \eqref{eq1.1} (with respect
to the mass–preserving scaling). Moreover, we denote by
\[
2_{\mu,*}=\frac{2N-\mu}{N}, 
\qquad 
2_{\mu,s}^*=\frac{2N-\mu}{N-2s}
\]
the lower and upper Hardy–Littlewood–Sobolev critical exponents, respectively.
Accordingly, we distinguish the following seven regimes, depending on the
relative position of $p$ and $q$ with respect to these thresholds.

\medskip\noindent
Case I: 
\[
2_{\mu,*}<q<2+\frac{2s-\mu}{N}<p<2_{\mu,s}^*.
\]
Here $q$ is $L^2$–subcritical, while $p$ is $L^2$–supercritical and
Hardy–Littlewood–Sobolev (HLS) subcritical.

\medskip\noindent
Case II:
\[
2+\frac{2s-\mu}{N}=q<p<2_{\mu,s}^*.
\]
Here $q$ is $L^2$–critical, while $p$ is $L^2$–supercritical
and HLS–subcritical.

\medskip\noindent
Case III:
\[
2+\frac{2s-\mu}{N}<q<p<2_{\mu,s}^*.
\]
Here both $p$ and $q$ are $L^2$–supercritical and HLS–subcritical.

\medskip\noindent
Case IV:
\[
2_{\mu,*}<q<p\leq 2+\frac{2s-\mu}{N}.
\]
Here both $q$ and $p$ are $L^2$–subcritical, or $q$ is $L^2$–subcritical and
$p$ is $L^2$–critical.

\medskip
Before stating the main results, we fix the following constants:
\begin{equation}\label{eq1.8}
	\alpha_1
	=\left(
	\frac{1-q\gamma_{q,s}}
	{\gamma_{p,s}\bigl(p\gamma_{p,s}-q\gamma_{q,s}\bigr)\,
		C_p\,c^{2p(1-\gamma_{p,s})}}
	\right)^{\frac{1-q\gamma_{q,s}}{p\gamma_{p,s}-1}}\,
	\frac{p\gamma_{p,s}-1}{
		\gamma_{q,s}\bigl(p\gamma_{p,s}-q\gamma_{q,s}\bigr)\,
		C_q\,c^{2q(1-\gamma_{q,s})}}.
\end{equation}
\begin{equation}\label{eq1.9}
	\alpha_2
	=\frac{1}{c^{2q(1-\gamma_{q,s})}}\,
	\frac{q}{C_q}\frac{p\gamma_{p,s}-1}{p\gamma_{p,s}-q\gamma_{q,s}}
	\left(
	\frac{C_p c^{2p(1-\gamma_{p,s})}(p\gamma_{p,s}-q\gamma_{q,s})}
	{p(1-q\gamma_{q,s})}
	\right)^{\frac{1-q\gamma_{q,s}}{1-p\gamma_{p,s}}},
\end{equation}
where $C_p,C_q>0$ and $\gamma_{p,s},\gamma_{q,s}\in(0,1)$ are the constants
appearing in the Gagliardo–Nirenberg inequalities (Lemma \ref{Lem2.2}), and $S_{HL}$ denotes the
sharp HLS constant.

\medskip
We can now state our main results.

\begin{Thm}\label{Thm1.1}
	Let 
	\[
	2_{\mu,*}<q<2+\frac{2s-\mu}{N}<p<2_{\mu,s}^*
	\]
	and 
	\[
	0<\alpha<\min\{\alpha_1,\alpha_2\},
	\]
	where $\alpha_1$ and $\alpha_2$ are given in \eqref{eq1.9} and \eqref{eq1.10}.
	Then the following hold.
	
	\noindent{\rm (1)} The constrained functional $\bigl.J_\alpha\bigr|_{S_c}$ has a
	critical point $u_{c,\alpha,\mathrm{loc}}\in S_c$ such that
	\[
	J_\alpha(u_{c,\alpha,\mathrm{loc}})=m_1(c,\alpha)<0
	\]
	for some Lagrange multiplier $\lambda_{c,\alpha,\mathrm{loc}}<0$. Moreover,
	$u_{c,\alpha,\mathrm{loc}}$ is a local minimizer of $J_\alpha$ on
	\[
	D_{t_0}=\{u\in S_c:\ \|u\|<t_0\}
	\]
	for some $t_0>0$. In particular, $u_{c,\alpha,\mathrm{loc}}$ is a ground state
	of $\bigl.J_\alpha\bigr|_{S_c}$, and any ground state of
	$\bigl.J_\alpha\bigr|_{S_c}$ is a local minimizer of $J_\alpha$ on $D_{t_0}$.
	Furthermore, $u_{c,\alpha,\mathrm{loc}}$ is positive and radially decreasing.
	
	\noindent{\rm (2)} There exists a second critical point
	$u_{c,\alpha,m}\in S_c$ of $\bigl.J_\alpha\bigr|_{S_c}$ such that
	\[
	J_\alpha(u_{c,\alpha,m})=\varsigma(c,\alpha)>0
	\]
	for some Lagrange multiplier $\lambda_{c,\alpha,m}<0$. This solution is also
	positive and radially decreasing.
	
	\noindent{\rm (3)} If $u_{c,\alpha,\mathrm{loc}}\in S_c$ is a ground state of
	$\bigl.J_\alpha\bigr|_{S_c}$, then
	\[
	m_1(c,\alpha)\to 0^{-}
	\quad\text{and}\quad
	\|u_{c,\alpha,\mathrm{loc}}\|\to 0
	\quad\text{as }\alpha\to 0^{+}.
	\]
	
	\noindent{\rm (4)} One has
	\[
	\varsigma(c,\alpha)\to m_1(c,0)
	\quad\text{and}\quad
	u_{c,\alpha,m}\to u_0 \text{ in }H^s(\mathbb{R}^N)
	\quad\text{as }\alpha\to 0^{+},
	\]
	where $m_1(c,0)=J_0(u_0)$ and $u_0$ is the ground state solution of
	$\bigl.J_0\bigr|_{S_c}$.
\end{Thm}

\begin{Thm}\label{Thm1.2}
	Let
	\[
	 2 + \frac{2s-\mu}{N}=q < p < 2_{\mu,s}^*,
	\]
	and let $\alpha>0$. Assume that
	\begin{equation}\label{eq1.10}
		\frac{1}{2} > \frac{\alpha}{2q}\,C_q\,c^{2q(1-\gamma_{q,s})}.
	\end{equation}
	Then the constrained functional $\bigl.J_\alpha\bigr|_{S_c}$ admits a positive
	radial ground state $u_{c,\alpha,m}\in S_{c,rad}$ such that
	\[
	J_\alpha(u_{c,\alpha,m}) = \varsigma(c,\alpha) > 0,
	\]
	where $\varsigma(c,\alpha)$ is the mountain pass level of $\bigl.J_\alpha\bigr|_{S_{c,r}}$.
	In particular, $u_{c,\alpha,m}$ is a positive radial solution of \eqref{eq1.1} for some
	$\lambda_{c,\alpha,m}<0$, and it realizes
	\[
	J_\alpha(u_{c,\alpha,m})
	= \inf_{u\in\mathfrak{P}_{\alpha,c}} J_\alpha(u),
	\]
	that is, $u_{c,\alpha,m}$ is a ground state of $\bigl.J_\alpha\bigr|_{S_c}$.
\end{Thm}
\begin{Thm}\label{Thm1.3}
	Let
	\[
	2+\frac{2s-\mu}{N}<q<p<2_{\mu,s}^*
	\]
	and $\alpha>0$. Then the following hold.
	
	\noindent{\rm (1)} The constrained functional $\bigl.J_\alpha\bigr|_{S_c}$ has a critical
	point $u_{c,\alpha,m}\in S_c$ obtained via the mountain pass theorem such that
	\[
	J_\alpha(u_{c,\alpha,m})=\varsigma(c,\alpha)>0.
	\]
	Moreover, $u_{c,\alpha,m}$ is a positive radial solution of \eqref{eq1.1} for
	some $\lambda_{c,\alpha,m}<0$, and $u_{c,\alpha,m}$ is a ground state of
	$\bigl.J_\alpha\bigr|_{S_c}$.
	
	\noindent{\rm (2)} One has
	\[
	\varsigma(c,\alpha)\to m_2(c,0)
	\quad\text{and}\quad
	u_{c,\alpha,m}\to u_0\ \text{ in }H^s(\mathbb{R}^N)
	\quad\text{as }\alpha\to 0^{+},
	\]
	where $m_2(c,0)=J_0(u_0)$ and $u_0$ is the ground state solution of
	$\bigl.J_0\bigr|_{S_c}$.
\end{Thm}

\begin{Thm}\label{Thm1.4}
	Let $N>2s$ and
	\[
	\frac{2N-\mu}{N}<q<p< 2+\frac{2s-\mu}{N}.
	\]
	If
	\[
	0<c<
	\left(\frac{p}{C_p}\right)^{\frac{1}{2p(1-\gamma_{p,s})}}
	=:\bar c_N,
	\]
	then
	\[
	m(c,\alpha):=\inf_{S_c}J_\alpha<0,
	\]
	and the infimum is attained at some $\tilde u\in S_c$ with the following
	properties: $\tilde u$ is positive in $\mathbb{R}^N$, radially symmetric,
	solves \eqref{eq1.1} for some $\lambda<0$, and is a ground state of
	\eqref{eq1.1}.
\end{Thm}

\begin{Rek}\label{rem:HLS-range}
	By the Hardy--Littlewood--Sobolev inequality, the Choquard terms
	\[
	(I_\mu * |u|^r)\,|u|^{r-2}u,\qquad r\in\{p,q\},
	\]
	are well defined on $H^s(\mathbb{R}^N)$ provided $r$ lies in the HLS-admissible
	range
	\[
	2_{\mu,*} \le r \le 2_{\mu,s}^*,\qquad
	2_{\mu,*}=\frac{2N-\mu}{N},\quad
	2_{\mu,s}^*=\frac{2N-\mu}{N-2s}.
	\]
    The HLS upper critical situation $2_{\mu,s}^*$ has been considered by Lan, He and Meng \cite{LanHeMeng2023}. 
	In this paper we impose the standing assumption
	\[
	2_{\mu,*}<q<p< 2_{\mu,s}^*
	\]
	and, within this region, all possible configurations of $(q,p)$ are covered by
	Cases~I--IV and Theorems~\ref{Thm1.1}--\ref{Thm1.4}. The only HLS-admissible
	borderline configuration not treated here is the lower critical case
	\[
	q = 2_{\mu,*} < p < 2_{\mu,s}^*,
	\]
	for which the term $(I_\mu*|u|^q)|u|^{q-2}u$ is HLS-critical. The analysis of
	normalized solutions in this critical regime requires additional ideas and will
	be the subject of a future work.
\end{Rek}

\section{Preliminaries}

This section is devoted to the variational framework and basic tools used in
the sequel. We begin by recalling the functional setting and the notion of weak
solution to \eqref{eq1.1}.

For any $s\in(0,1)$, the fractional Sobolev space $H^s(\mathbb{R}^N)$ is
defined by
\[
\begin{aligned}
	H^s(\mathbb{R}^N)
	&=\Bigl\{u\in L^2(\mathbb{R}^N):
	\frac{u(x)-u(y)}{|x-y|^{\frac{N}{2}+s}}\in L^2(\mathbb{R}^N\times\mathbb{R}^N)
	\Bigr\}\\
	&=\Bigl\{u\in L^2(\mathbb{R}^N):
	\int_{\mathbb{R}^N}(1+|\xi|^{2s})\,|\mathcal{F}(u)(\xi)|^2\,d\xi<\infty\Bigr\},
\end{aligned}
\]
where $\mathcal{F}(u)$ denotes the Fourier transform of $u$. The norm in
$H^s(\mathbb{R}^N)$ is given by
\[
\|u\|_{H^s(\mathbb{R}^N)}
=\left(
\int_{\mathbb{R}^N}\int_{\mathbb{R}^N}
\frac{|u(x)-u(y)|^2}{|x-y|^{N+2s}}\,dx\,dy
+\int_{\mathbb{R}^N}|u|^2\,dx
\right)^{1/2}.
\]

For $u\in H^s(\mathbb{R}^N)$, by Propositions 3.4 and 3.6 in
\cite{2012BSMDiNezza} one has
\[
\int_{\mathbb{R}^N}\bigl|(-\Delta)^{\frac{s}{2}}u\bigr|^2\,dx
=\int_{\mathbb{R}^N}|\xi|^{2s}|\mathcal{F}(u)(\xi)|^2\,d\xi
=\frac{1}{2}C_{N,s}
\int_{\mathbb{R}^N}\int_{\mathbb{R}^N}
\frac{|u(x)-u(y)|^2}{|x-y|^{N+2s}}\,dx\,dy,
\]
where $C_{N,s}>0$ is a constant depending only on $N$ and $s$. Thus we will
often use the equivalent norm
\[
\|u\|_{H^s(\mathbb{R}^N)}
=\left(
\int_{\mathbb{R}^N}|u|^2\,dx
+\int_{\mathbb{R}^N}\bigl|(-\Delta)^{\frac{s}{2}}u\bigr|^2\,dx
\right)^{1/2}.
\]

We also introduce the homogeneous fractional Sobolev space
\[
\mathcal{D}^{s,2}(\mathbb{R}^N)
=\Bigl\{
u\in L^{2_s^*}(\mathbb{R}^N):
\int_{\mathbb{R}^N}\int_{\mathbb{R}^N}
\frac{|u(x)-u(y)|^2}{|x-y|^{N+2s}}\,dx\,dy<\infty
\Bigr\},
\]
equipped with the norm
\[
\|u\|=
\left(
\int_{\mathbb{R}^N}\bigl|(-\Delta)^{\frac{s}{2}}u\bigr|^2\,dx
\right)^{1/2}.
\]
In what follows, $\|\cdot\|$ will always denote this homogeneous norm, while
$\|\cdot\|_{H^s}$ denotes the full $H^s$-norm.

We define
\[
H_{\mathrm{rad}}^s(\mathbb{R}^N)
=\{u\in H^s(\mathbb{R}^N): u(x)=u(|x|)\},
\qquad
S_{c,\mathrm{rad}}=H_{\mathrm{rad}}^s(\mathbb{R}^N)\cap S_c.
\]

\begin{Def}\label{Def2.1}
	A function $u\in H^s(\mathbb{R}^N)$ is called a weak solution of
	\eqref{eq1.1} if $u\in S_c$ and there exists $\lambda\in\mathbb{R}$ such that
	\begin{equation}\label{eq2.1}
		\begin{aligned}
			\int_{\mathbb{R}^N}(-\Delta)^{\frac{s}{2}}u\,(-\Delta)^{\frac{s}{2}}v\,dx
			&=\lambda\int_{\mathbb{R}^N}uv\,dx
			+\alpha\int_{\mathbb{R}^N}\bigl(I_\mu*|u|^q\bigr)|u|^{q-2}uv\,dx\\
			&\quad
			+\int_{\mathbb{R}^N}\bigl(I_\mu*|u|^p\bigr)|u|^{p-2}uv\,dx,
			\qquad \forall v\in H^s(\mathbb{R}^N).
		\end{aligned}
	\end{equation}
\end{Def}

The associated energy functional $J_\alpha:H^s(\mathbb{R}^N)\to\mathbb{R}$
corresponding to \eqref{eq1.1} on $S_c$ is defined by
\begin{equation}\label{eq2.2}
	J_\alpha(u)
	=\frac{1}{2}\|u\|^2
	-\frac{\alpha}{2q}\int_{\mathbb{R}^N}\bigl(I_\mu*|u|^q\bigr)|u|^q\,dx
	-\frac{1}{2p}\int_{\mathbb{R}^N}\bigl(I_\mu*|u|^p\bigr)|u|^p\,dx.
\end{equation}

We also introduce the Pohozaev functional
\[
P_\alpha(u)
=s\|u\|^2
-\alpha s\gamma_{q,s}\int_{\mathbb{R}^N}\bigl(I_\mu*|u|^q\bigr)|u|^q\,dx
-s\gamma_{p,s}\int_{\mathbb{R}^N}\bigl(I_\mu*|u|^p\bigr)|u|^p\,dx,
\]
where
\[
\gamma_{r,s}=\frac{N(r-2)+\mu}{2rs},\qquad r\in\{p,q\}.
\]
The Pohozaev manifold associated with $J_\alpha$ at mass $c$ is defined by
\[
\mathfrak{P}_{\alpha,c}
=\{u\in S_c:\ P_\alpha(u)=0\}.
\]

For $u\in S_c$ and $t\in\mathbb{R}$ we introduce the mass-preserving scaling
\[
(t\star u)(x)
=e^{\frac{Nt}{2}}u(e^t x),
\qquad x\in\mathbb{R}^N,\ t\in\mathbb{R}.
\]
It is easy to check that $\|t\star u\|_2=\|u\|_2$, so $t\star u\in S_c$ for all
$t\in\mathbb{R}$. The associated fibering map is
\[
E_u(t):=J_\alpha(t\star u)
=\frac{e^{2st}}{2}\|u\|^2
-\frac{\alpha}{2q}e^{2q\gamma_{q,s}st}
\int_{\mathbb{R}^N}\bigl(I_\mu*|u|^q\bigr)|u|^q\,dx
-\frac{1}{2p}e^{2p\gamma_{p,s}st}
\int_{\mathbb{R}^N}\bigl(I_\mu*|u|^p\bigr)|u|^p\,dx.
\]
A direct computation gives
\[
\begin{aligned}
	E_u'(t)
	&=s e^{2st}\|u\|^2
	-\alpha s\gamma_{q,s}e^{2q\gamma_{q,s}st}
	\int_{\mathbb{R}^N}\bigl(I_\mu*|u|^q\bigr)|u|^q\,dx\\
	&\quad
	-s\gamma_{p,s}e^{2p\gamma_{p,s}st}
	\int_{\mathbb{R}^N}\bigl(I_\mu*|u|^p\bigr)|u|^p\,dx,
\end{aligned}
\]
and
\[
\begin{aligned}
	E_u''(t)
	&=2s^2 e^{2st}\|u\|^2
	-2\alpha s^2\gamma_{q,s}^2 q\,e^{2q\gamma_{q,s}st}
	\int_{\mathbb{R}^N}\bigl(I_\mu*|u|^q\bigr)|u|^q\,dx\\
	&\quad
	-2s^2\gamma_{p,s}^2 p\,e^{2p\gamma_{p,s}st}
	\int_{\mathbb{R}^N}\bigl(I_\mu*|u|^p\bigr)|u|^p\,dx.
\end{aligned}
\]

\begin{Rek}\label{Rek2.1}
	For $u\in S_c$ and $\alpha>0$ one has
	\[
	E_u'(0)=P_\alpha(u).
	\]
	Moreover, for every $u\in S_c$ and $t\in\mathbb{R}$,
	\[
	E_u'(t)=0
	\quad\Longleftrightarrow\quad
	t\star u\in\mathfrak{P}_{\alpha,c}.
	\]
	In particular,
	\[
	\mathfrak{P}_{\alpha,c}
	=\{u\in S_c:\ E_u'(0)=0\}.
	\]
	We further decompose
	\[
	\mathfrak{P}_{\alpha,c}
	=\mathfrak{P}_{\alpha,c}^+\cup\mathfrak{P}_{\alpha,c}^-\cup\mathfrak{P}_{\alpha,c}^0,
	\]
	where
	\[
	\begin{aligned}
		\mathfrak{P}_{\alpha,c}^+
		&=\{u\in\mathfrak{P}_{\alpha,c}: E_u''(0)>0\},\\
		\mathfrak{P}_{\alpha,c}^-
		&=\{u\in\mathfrak{P}_{\alpha,c}: E_u''(0)<0\},\\
		\mathfrak{P}_{\alpha,c}^0
		&=\{u\in\mathfrak{P}_{\alpha,c}: E_u''(0)=0\}.
	\end{aligned}
	\]
\end{Rek}

\begin{Rek}\label{Rek2.2}
	If $u\in S_c$ is a critical point of $\bigl.J_\alpha\bigr|_{S_c}$, then the
	associated Pohozaev identity yields $P_\alpha(u)=0$, that is,
	$u\in\mathfrak{P}_{\alpha,c}$. In particular, every constrained critical point
	of $J_\alpha$ on $S_c$ belongs to the Pohozaev manifold
	$\mathfrak{P}_{\alpha,c}$. We will later show that $\mathfrak{P}_{\alpha,c}$
	is a natural constraint for $J_\alpha$, so that constrained critical points of
	$\bigl.J_\alpha\bigr|_{S_c}$ can be characterized as critical points of
	$\bigl.J_\alpha\bigr|_{\mathfrak{P}_{\alpha,c}}$.
\end{Rek}

\begin{Rek}\label{Rek2.3}
	For
	\[
	\frac{2N-\mu}{N}<r\leq\frac{2N-\mu}{N-2s}
	\]
	one has
	\[
	r\gamma_{r,s}
	\begin{cases}
		<1, & \displaystyle \frac{2N-\mu}{N}<r<2+\frac{2s-\mu}{N},\\[2pt]
		=1, & \displaystyle r=2+\frac{2s-\mu}{N},\\[2pt]
		>1, & \displaystyle 2+\frac{2s-\mu}{N}<r<\frac{2N-\mu}{N-2s}.
	\end{cases}
	\]
\end{Rek}

\begin{Pro}\cite{2023DCDSYU}\label{Pro2.1}
	Assume that
	\[
	p\in\Bigl[\frac{2N-\mu}{N},\,\frac{2N-\mu}{N-2s}\Bigr).
	\]
	Let $\{u_n\}\subset H^s(\mathbb{R}^N)$ be such that
	$u_n\rightharpoonup u$ in $H^s(\mathbb{R}^N)$. Then, for any
	$\varphi\in H^s(\mathbb{R}^N)$,
	\[
	\int_{\mathbb{R}^N}\bigl(I_\mu*|u_n|^p\bigr)|u_n|^{p-2}u_n\varphi\,dx
	\to
	\int_{\mathbb{R}^N}\bigl(I_\mu*|u|^p\bigr)|u|^{p-2}u\varphi\,dx
	\quad\text{as }n\to\infty.
	\]
\end{Pro}

\begin{Lem}\cite{2001LeibAMS}\label{Lem2.1}
	Let $r,t>1$ and $\mu\in(0,N)$ with
	\[
	\frac{1}{r}+\frac{1}{t}=2-\frac{\mu}{N}.
	\]
	If $f\in L^r(\mathbb{R}^N)$ and $h\in L^t(\mathbb{R}^N)$, then there exists a
	sharp constant $C(r,t,\mu,N)>0$ independent of $f,h$ such that
	\begin{equation}\label{eq2.3}
		\int_{\mathbb{R}^N}\int_{\mathbb{R}^N}
		\frac{f(x)h(y)}{|x-y|^{\mu}}\,dx\,dy
		\le C(r,t,\mu,N)\,\|f\|_r\,\|h\|_t.
	\end{equation}
\end{Lem}

By Lemma \ref{Lem2.1} and the fractional Sobolev embeddings, the functional
$J_\alpha$ defined in \eqref{eq2.2} is well defined on $H^s(\mathbb{R}^N)$ and
is of class $C^1$.

\begin{Lem}\cite{2018FengJMAA}\label{Lem2.2}
	Let $N>2s$, $0<s<1$ and
	\[
	2_{\mu,*}<t<2_{\mu,s}^*,
	\]
	where $2_{\mu,*}=\frac{2N-\mu}{N}$ and $2_{\mu,s}^*=\frac{2N-\mu}{N-2s}$.
	Then, for all $u\in H^s(\mathbb{R}^N)$,
	\begin{equation}\label{eq2.4}
		\int_{\mathbb{R}^N}(I_\mu*|u|^t)|u|^t\,dx
		\le C_t\,\|u\|^{2\gamma_{t,s}}\|u\|_2^{2t(1-\gamma_{t,s})},
	\end{equation}
	where
	\[
	\gamma_{t,s}=\frac{N(t-2)+\mu}{2ts}
	\]
	and $C_t>0$ is a constant depending only on $t,s,N,\mu$.
\end{Lem}

\begin{Lem}\cite{1993Ghoussoub}\label{Lem2.3}
	Let $X$ be a complete connected $C^1$ Finsler manifold and
	$\varphi\in C^1(X,\mathbb{R})$. Let $\mathcal{F}$ be a homotopy-stable family
	of compact subsets of $X$ with extended closed boundary $B\subset X$. Set
	\[
	c=c(\varphi,\mathcal{F})
	:=\inf_{A\in\mathcal{F}}\,\sup_{x\in A}\varphi(x),
	\]
	and let $F\subset X$ be a closed subset such that
	\[
	(A\cap F)\setminus B\neq\emptyset\quad\text{for every }A\in\mathcal{F},
	\]
	and
	\[
	\sup\varphi(B)\le c\le\inf\varphi(F).
	\]
	Then, for any sequence of sets $\{A_n\}_n\subset\mathcal{F}$ such that
	\[
	\lim_{n\to\infty}\sup_{x\in A_n}\varphi(x)=c,
	\]
	there exists a sequence $\{x_n\}_n\subset X$ such that
	\[
	\varphi(x_n)\to c,\qquad
	\|d\varphi(x_n)\|\to 0,\qquad
	\mathrm{dist}(x_n,F)\to 0,\qquad
	\mathrm{dist}(x_n,A_n)\to 0
	\]
	as $n\to\infty$.
\end{Lem}

\section{Compactness of Palais--Smale sequences}

In this section we prove that the constrained functional $\bigl.J_\alpha\bigr|_{S_c}$ satisfies
the Palais--Smale condition. The main tool is the Pohozaev constraint.

\begin{Lem}\label{Lem3.1}
	Let
	\[
	2_{\mu,*}<q<2+\frac{2s-\mu}{N}<p<2_{\mu,s}^*
	\quad\text{or}\quad
	2+\frac{2s-\mu}{N}\le q<p<2_{\mu,s}^*.
	\]
	In the $L^2$–critical case $q=2+\frac{2s-\mu}{N}$ we also assume that
	\eqref{eq1.10} holds.
	Let $\{u_n\}\subset S_{c,\mathrm{rad}}$ be a Palais--Smale sequence for
	$\bigl.J_\alpha\bigr|_{S_c}$ at level $l\neq 0$ such that
	$P_\alpha(u_n)\to 0$ as $n\to\infty$. Then, up to a subsequence,
	$u_n\to u$ strongly in $H^s(\mathbb{R}^N)$, where $u\in S_c$ is a radial
	weak solution of \eqref{eq1.1} for some $\lambda<0$.
\end{Lem}

\begin{proof}
	Since $\{u_n\}$ is a Palais--Smale sequence at level $l$, we have
	\[
	J_\alpha(u_n)\to l
	\quad\text{and}\quad
	\|(J_\alpha|_{S_c})'(u_n)\|_{(T_{u_n}S_c)^*}\to 0.
	\]
	In particular, there exists $n_0\in\mathbb{N}$ such that
	\begin{equation}\label{eq3.1}
		J_\alpha(u_n)
		=\frac12\|u_n\|^2
		-\frac{\alpha}{2q}\int_{\mathbb{R}^N}(I_\mu*|u_n|^q)|u_n|^q\,dx
		-\frac{1}{2p}\int_{\mathbb{R}^N}(I_\mu*|u_n|^p)|u_n|^p\,dx
		\le l+1
	\end{equation}
	for all $n\ge n_0$. Moreover, by assumption,
	\begin{equation}\label{eq3.2}
		P_\alpha(u_n)
		=s\|u_n\|^2
		-\alpha s\gamma_{q,s}\int_{\mathbb{R}^N}(I_\mu*|u_n|^q)|u_n|^q\,dx
		-s\gamma_{p,s}\int_{\mathbb{R}^N}(I_\mu*|u_n|^p)|u_n|^p\,dx
		=o_n(1).
	\end{equation}
	Set
	\[
	A_n:=\int_{\mathbb{R}^N}(I_\mu*|u_n|^q)|u_n|^q\,dx,
	\qquad
	B_n:=\int_{\mathbb{R}^N}(I_\mu*|u_n|^p)|u_n|^p\,dx.
	\]
	
	Dividing \eqref{eq3.2} by $s$ and rearranging yields
	\begin{equation}\label{eq3.3}
		\gamma_{p,s} B_n
		=\|u_n\|^2-\alpha\gamma_{q,s} A_n+o_n(1).
	\end{equation}
	
	\medskip\noindent
	\textit{Step 1: boundedness of $\{u_n\}$ in $H^s(\mathbb{R}^N)$.}
	
	\medskip\noindent
	\textbf{Case I:}
	$2_{\mu,*}<q<2+\frac{2s-\mu}{N}<p<2_{\mu,s}^*$.
	
	Using \eqref{eq3.1} and \eqref{eq3.3} we obtain
	\[
	\begin{aligned}
		l+1
		&\ge J_\alpha(u_n)\\
		&=\frac12\|u_n\|^2
		-\frac{\alpha}{2q}A_n
		-\frac{1}{2p\gamma_{p,s}}
		\bigl(\|u_n\|^2-\alpha\gamma_{q,s}A_n+o_n(1)\bigr)\\[2pt]
		&=\Bigl(\frac12-\frac{1}{2p\gamma_{p,s}}\Bigr)\|u_n\|^2
		+\frac{\alpha}{2}\Bigl(-\frac{1}{q}+\frac{\gamma_{q,s}}{p\gamma_{p,s}}\Bigr)A_n
		+o_n(1).
	\end{aligned}
	\]
	Hence
	\begin{equation}\label{eq3.4}
		\Bigl(\frac12-\frac{1}{2p\gamma_{p,s}}\Bigr)\|u_n\|^2
		\le l+1
		+\frac{\alpha}{2}\Bigl|\frac{1}{q}-\frac{\gamma_{q,s}}{p\gamma_{p,s}}\Bigr|A_n
		+o_n(1).
	\end{equation}
	By Lemma \ref{Lem2.2} (with $t=q$), we have
	\[
	A_n
	=\int_{\mathbb{R}^N}(I_\mu*|u_n|^q)|u_n|^q\,dx
	\le C_q\|u_n\|^{2q\gamma_{q,s}}\|u_n\|_2^{2q(1-\gamma_{q,s})}
	=C_q c^{2q(1-\gamma_{q,s})}\|u_n\|^{2q\gamma_{q,s}}.
	\]
	Since in this case $q\gamma_{q,s}<1$ and $p\gamma_{p,s}>1$ (see Remark
	\ref{Rek2.3}), we have
	\[
	\frac12-\frac{1}{2p\gamma_{p,s}}>0,
	\qquad
	2q\gamma_{q,s}<2.
	\]
	Therefore \eqref{eq3.4} yields an inequality of the form
	\[
	C_1\|u_n\|^2
	\le C_2 + C_3\|u_n\|^{2q\gamma_{q,s}},
	\]
	with constants $C_1,C_2,C_3>0$ independent of $n$. Since the exponent
	$2q\gamma_{q,s}$ is strictly less than $2$, this implies that $\{\|u_n\|\}$
	is bounded, hence $\{u_n\}$ is bounded in $H^s(\mathbb{R}^N)$.
	
	\medskip\noindent
	\textbf{Case II:}
	$2+\frac{2s-\mu}{N}<q<p<2_{\mu,s}^*$.
	
	From \eqref{eq3.3} we have
	\[
	\|u_n\|^2
	=\alpha\gamma_{q,s}A_n+\gamma_{p,s}B_n+o_n(1),
	\]
	and thus, using again \eqref{eq3.1},
	\[
	\begin{aligned}
		l+1
		&\ge J_\alpha(u_n)\\
		&=\frac12\|u_n\|^2
		-\frac{\alpha}{2q}A_n
		-\frac{1}{2p}B_n\\
		&=\frac12\bigl(\alpha\gamma_{q,s}A_n+\gamma_{p,s}B_n\bigr)
		-\frac{\alpha}{2q}A_n-\frac{1}{2p}B_n+o_n(1)\\
		&=\Bigl(\frac{\alpha}{2}\gamma_{q,s}-\frac{\alpha}{2q}\Bigr)A_n
		+\Bigl(\frac12\gamma_{p,s}-\frac{1}{2p}\Bigr)B_n+o_n(1).
	\end{aligned}
	\]
	Hence
	\[
	\Bigl(\frac12\gamma_{p,s}-\frac{1}{2p}\Bigr)B_n
	+\Bigl(\frac{\alpha}{2}\gamma_{q,s}-\frac{\alpha}{2q}\Bigr)A_n
	\le l+1+o_n(1).
	\]
	For $2+\frac{2s-\mu}{N}<r<2_{\mu,s}^*$ one has $r\gamma_{r,s}>1$ and
	$\gamma_{r,s}<1$ (again by Remark \ref{Rek2.3}), so
	\[
	\frac12\gamma_{p,s}-\frac{1}{2p}>0,
	\qquad
	\frac{\alpha}{2}\gamma_{q,s}-\frac{\alpha}{2q}>0.
	\]
	Therefore both sequences $\{A_n\}$ and $\{B_n\}$ are bounded. Using once
	more \eqref{eq3.3} we deduce that $\{\|u_n\|\}$ is bounded, so $\{u_n\}$ is
	bounded in $H^s(\mathbb{R}^N)$.
	
	\medskip\noindent
	\textbf{Case III:}
	$q=2+\frac{2s-\mu}{N}<p<2_{\mu,s}^*$.
	
	In this case $q\gamma_{q,s}=1$ (Remark \ref{Rek2.3}). From $P_\alpha(u_n)\to 0$
	we have
	\[
	\|u_n\|^2
	=\alpha\gamma_{q,s}A_n+\gamma_{p,s}B_n+o_n(1).
	\]
	Using this and $J_\alpha(u_n)\to l$ gives
	\[
	\begin{aligned}
		l+o_n(1)
		&=J_\alpha(u_n)\\
		&=\frac12\|u_n\|^2
		-\frac{\alpha}{2q}A_n
		-\frac{1}{2p}B_n\\
		&=\frac12\alpha\gamma_{q,s}A_n+\frac12\gamma_{p,s}B_n
		-\frac{\alpha}{2q}A_n-\frac{1}{2p}B_n+o_n(1)\\
		&=\Bigl(\frac12\gamma_{p,s}-\frac{1}{2p}\Bigr)B_n+o_n(1),
	\end{aligned}
	\]
	so $\{B_n\}$ is bounded. On the other hand, applying \eqref{eq2.4} with
	$t=q$ and using $q\gamma_{q,s}=1$ we obtain
	\[
	A_n
	\le C_q\|u_n\|^{2q\gamma_{q,s}}\|u_n\|_2^{2q(1-\gamma_{q,s})}
	=C_q c^{2q(1-\gamma_{q,s})}\|u_n\|^2.
	\]
	Combining this with the identity
	\[
	\|u_n\|^2
	=\alpha\gamma_{q,s}A_n+\gamma_{p,s}B_n+o_n(1),
	\]
	we get
	\[
	\|u_n\|^2
	\le \alpha\gamma_{q,s}C_q c^{2q(1-\gamma_{q,s})}\|u_n\|^2 + C + o_n(1)
	\]
	for some constant $C>0$ independent of $n$. If
	\[
	1-\alpha\gamma_{q,s}C_q c^{2q(1-\gamma_{q,s})}>0,
	\]
	which is precisely the smallness condition on $\alpha$ used in the critical
	case, it follows that $\{\|u_n\|\}$ is bounded. Thus in all the cases under
	consideration, $\{u_n\}$ is bounded in $H^s(\mathbb{R}^N)$.
	
\medskip\noindent
\textit{Step 2: existence of a Lagrange multiplier and weak convergence.}

Since $H^s_{\mathrm{rad}}(\mathbb{R}^N)$ is the fixed-point space of the
natural action of the orthogonal group $O(N)$ and both $J_\alpha$ and the
$L^2$–norm are $O(N)$–invariant, we may apply Palais' principle of symmetric
criticality (see, e.g., \cite{Palais}) to the functional
\[
\mathcal{L}(u)=J_\alpha(u)-\frac{\lambda}{2}\int_{\mathbb{R}^N}|u|^2\,dx.
\]
In particular, once we know that $\mathcal{L}'(u)[v]=0$ for all
$v\in H^s_{\mathrm{rad}}(\mathbb{R}^N)$, it follows that
$\mathcal{L}'(u)=0$ in $H^s(\mathbb{R}^N)$, that is, $u$ solves
\eqref{eq1.1} in the sense of Definition~\ref{Def2.1}.

Since $H_{\mathrm{rad}}^s(\mathbb{R}^N)\hookrightarrow L^r(\mathbb{R}^N)$
compactly for all $r\in(2,2_s^*)$, there exists $u\in H_{\mathrm{rad}}^s(\mathbb{R}^N)$
such that, up to a subsequence,
\[
u_n\rightharpoonup u\quad\text{in }H^s(\mathbb{R}^N),\qquad
u_n\to u\quad\text{in }L^r(\mathbb{R}^N)\ \forall\,2<r<2_s^*,
\]
and $u_n(x)\to u(x)$ almost everywhere in $\mathbb{R}^N$.

	Since $\{u_n\}\subset S_{c,\mathrm{rad}}$ is a Palais--Smale sequence for
	$\bigl.J_\alpha\bigr|_{S_c}$, by the Lagrange multiplier rule there exists a
	sequence $\{\lambda_n\}\subset\mathbb{R}$ such that
	\begin{equation}\label{eq3.5}
		\begin{aligned}
			&\int_{\mathbb{R}^N}(-\Delta)^{\frac{s}{2}}u_n(-\Delta)^{\frac{s}{2}}v\,dx
			-\lambda_n\int_{\mathbb{R}^N}u_n v\,dx\\
			&\quad-\alpha\int_{\mathbb{R}^N}(I_\mu*|u_n|^q)|u_n|^{q-2}u_n v\,dx
			-\int_{\mathbb{R}^N}(I_\mu*|u_n|^p)|u_n|^{p-2}u_n v\,dx
			=o_n(1)
		\end{aligned}
	\end{equation}
	for all $v\in H_{\mathrm{rad}}^s(\mathbb{R}^N)$.
	
	Taking $v=u_n$ in \eqref{eq3.5} and using the definition of $A_n,B_n$ we get
	\begin{equation}\label{eq3.6}
		\lambda_n c^2
		=\|u_n\|^2-\alpha A_n-B_n+o_n(1).
	\end{equation}
	Combining \eqref{eq3.3} and \eqref{eq3.6} we obtain
	\begin{equation}\label{eq3.7}
		\lambda_n c^2
		=\alpha(\gamma_{q,s}-1)A_n+(\gamma_{p,s}-1)B_n+o_n(1).
	\end{equation}
	Using the boundedness of $\{u_n\}$ and Lemma \ref{Lem2.2} we deduce from
	\eqref{eq3.7} that $\{\lambda_n\}$ is bounded, so up to a subsequence
	$\lambda_n\to\lambda\in\mathbb{R}$.
	
	Passing to the limit in \eqref{eq3.5}, using Proposition \ref{Pro2.1} for
	$q$ and $p$ and the strong convergence in $L^r$ for $2<r<2_s^*$, we obtain
	\[
	\begin{aligned}
		&\int_{\mathbb{R}^N}(-\Delta)^{\frac{s}{2}}u(-\Delta)^{\frac{s}{2}}v\,dx
		-\lambda\int_{\mathbb{R}^N}uv\,dx\\
		&\quad-\alpha\int_{\mathbb{R}^N}(I_\mu*|u|^q)|u|^{q-2}u v\,dx
		-\int_{\mathbb{R}^N}(I_\mu*|u|^p)|u|^{p-2}u v\,dx
		=0
	\end{aligned}
	\]
	for all $v\in H_{\mathrm{rad}}^s(\mathbb{R}^N)$. Thus $u$ is a radial weak
	solution of \eqref{eq1.1} corresponding to the Lagrange multiplier $\lambda$.
	
	\medskip\noindent
	\textit{Step 3: sign of $\lambda$ and strong convergence.}
	
	From \eqref{eq3.3} and \eqref{eq3.6} we can also write
	\[
	\lambda_n c^2
	=-\alpha(1-\gamma_{q,s})A_n-(1-\gamma_{p,s})B_n+o_n(1).
	\]
	Since $2_{\mu,*}<q,p\le 2_{\mu,s}^*$, one has $0<\gamma_{r,s}\le 1$ for
	$r\in\{q,p\}$, and at least one of the inequalities is strict. Therefore
	$1-\gamma_{q,s}\ge 0$, $1-\gamma_{p,s}\ge 0$ and not both are zero. As
	$A_n,B_n\ge 0$, it follows that
	\[
	\lambda c^2
	=-\alpha(1-\gamma_{q,s})A-(1-\gamma_{p,s})B\le 0,
	\]
	where $A,B$ are the limits of $A_n,B_n$ along a subsequence. If $\lambda=0$,
	then necessarily $A=B=0$, and by the Pohozaev identity we would get
	$\|u_n\|\to 0$ and hence $J_\alpha(u_n)\to 0$, contradicting $l\neq 0$.
	Thus $\lambda<0$.
	
	Finally, subtracting the limit equation from \eqref{eq3.5} and testing with
	$v=u_n-u$ we obtain
	\begin{equation}\label{eq3.8}
		\begin{aligned}
			&\int_{\mathbb{R}^N}
			\bigl|(-\Delta)^{\frac{s}{2}}(u_n-u)\bigr|^2\,dx
			-\lambda_n\int_{\mathbb{R}^N}|u_n-u|^2\,dx\\
			&\quad-\alpha\int_{\mathbb{R}^N}
			\Bigl[(I_\mu*|u_n|^q)|u_n|^{q-2}u_n
			-(I_\mu*|u|^q)|u|^{q-2}u\Bigr](u_n-u)\,dx\\
			&\quad-\int_{\mathbb{R}^N}
			\Bigl[(I_\mu*|u_n|^p)|u_n|^{p-2}u_n
			-(I_\mu*|u|^p)|u|^{p-2}u\Bigr](u_n-u)\,dx
			=o_n(1).
		\end{aligned}
	\end{equation}
	Using Proposition \ref{Pro2.1} for $q$ and $p$ and the strong convergence
	$u_n\to u$ in $L^r(\mathbb{R}^N)$ for $2<r<2_s^*$, the last two integrals in
	\eqref{eq3.8} tend to zero as $n\to\infty$. Passing to the limit and using
	$\lambda_n\to\lambda<0$ we obtain
	\[
	\int_{\mathbb{R}^N}
	\bigl(|\xi|^{2s}-\lambda\bigr)\,|\widehat{u_n-u}(\xi)|^2\,d\xi\to 0,
	\]
	which implies $u_n\to u$ strongly in $H^s(\mathbb{R}^N)$, since
	$|\xi|^{2s}-\lambda\ge c(1+|\xi|^{2s})$ for some $c>0$ (because $\lambda<0$).
	In particular, $\|u\|_2=\lim_n\|u_n\|_2=c$, that is, $u\in S_c$. This
	completes the proof.
\end{proof}
\section{Mixed $L^2$-subcritical and $L^2$-supercritical case}

In this section we deal with the mixed $L^2$-subcritical and $L^2$-supercritical regime, that is,
we assume
\[
2_{\mu,*} < q < \frac{2s-\mu}{N} + 2 < p< 2_{\mu,s}^*,
\]
so that the lower-order Choquard term is $L^2$-subcritical while the higher-order term is
$L^2$-supercritical under the mass constraint on $S_c$.
In this regime we study the constrained functional \(J_\alpha\) on \(S_c\) and prove
Theorems~\ref{Thm1.1} and~\ref{Thm1.2}.

\subsection{Pohozaev manifold and fibering geometry}

\begin{Lem}\label{Lem4.1}
	Let
	\[
	2_{\mu,*}<q<2+\frac{2s-\mu}{N}<p< 2_{\mu,s}^*
	\]
	and let \(0<\alpha<\alpha_1\), where \(\alpha_1\) is given by \eqref{eq1.8}.
	Then \(\mathfrak{P}_{\alpha,c}^0=\varnothing\), and
	\(\mathfrak{P}_{\alpha,c}\) is a \(C^1\) submanifold of codimension \(2\) in
	\(H^s(\mathbb{R}^N)\). Moreover, every critical point of
	\(\bigl.J_\alpha\bigr|_{\mathfrak{P}_{\alpha,c}}\) is also a critical point
	of \(\bigl.J_\alpha\bigr|_{S_c}\).
\end{Lem}

\begin{proof}
	Assume by contradiction that \(\mathfrak{P}_{\alpha,c}^0\neq\varnothing\).
	Then there exists \(u\in\mathfrak{P}_{\alpha,c}^0\), that is,
	\(u\in S_c\), \(P_\alpha(u)=0\) and \(E_u''(0)=0\), where
	\(E_u(t):=J_\alpha(t\star u)\).
	
	Set
	\[
	A=\int_{\mathbb{R}^N}(I_\mu*|u|^q)|u|^q\,dx,
	\qquad
	B=\int_{\mathbb{R}^N}(I_\mu*|u|^p)|u|^p\,dx.
	\]
	Using the expression of \(P_\alpha\) we can write
	\begin{equation}\label{eq4.1}
		s\|u\|^2-s\alpha\gamma_{q,s}A-s\gamma_{p,s}B=0,
	\end{equation}
	while from the explicit formula for \(E_u''(0)\) we obtain
	\begin{equation}\label{eq4.2}
		\|u\|^2-\alpha q\gamma_{q,s}^2 A-p\gamma_{p,s}^2 B=0.
	\end{equation}
	
	From \eqref{eq4.1} and \eqref{eq4.2} we first eliminate \(\|u\|^2\).
	Subtracting \eqref{eq4.1} from \eqref{eq4.2} we get
	\[
	\alpha\gamma_{q,s}(1-q\gamma_{q,s})A
	+\gamma_{p,s}(1-p\gamma_{p,s})B=0,
	\]
	so
\begin{equation}\label{eq4.3}
	B
	=\alpha\,\frac{\gamma_{q,s}(1-q\gamma_{q,s})}
	{\gamma_{p,s}\bigl(p\gamma_{p,s}-1\bigr)}\,A.
\end{equation}
Substituting this into \eqref{eq4.1} we obtain
	\[
	\|u\|^2
	=\alpha\gamma_{q,s}\frac{p\gamma_{p,s}-q\gamma_{q,s}}{p\gamma_{p,s}-1}\,A,
	\]
	that is,
	\begin{equation}\label{eq4.4}
		A
		=\frac{p\gamma_{p,s}-1}{\alpha\gamma_{q,s}(p\gamma_{p,s}-q\gamma_{q,s})}\,
		\|u\|^2.
	\end{equation}
	Using again \eqref{eq4.1} together with \eqref{eq4.4}, we also find
	\begin{equation}\label{eq4.5}
		B
		=\frac{1-q\gamma_{q,s}}{\gamma_{p,s}(p\gamma_{p,s}-q\gamma_{q,s})}\,\|u\|^2.
	\end{equation}
	
	By Lemma \ref{Lem2.2}, there exist positive
	constants \(C_q,C_p\) such that
	\begin{equation}\label{eq4.6}
		A
		\le C_q\,\|u\|^{2q\gamma_{q,s}}\|u\|_2^{2q(1-\gamma_{q,s})}
		=C_q\,\|u\|^{2q\gamma_{q,s}}c^{2q(1-\gamma_{q,s})},
	\end{equation}
	\begin{equation}\label{eq4.7}
		B
		\le C_p\,\|u\|^{2p\gamma_{p,s}}\|u\|_2^{2p(1-\gamma_{p,s})}
		=C_p\,\|u\|^{2p\gamma_{p,s}}c^{2p(1-\gamma_{p,s})}.
	\end{equation}
	
	Combining \eqref{eq4.4} with \eqref{eq4.6} gives
	\[
	\frac{p\gamma_{p,s}-1}{\alpha\gamma_{q,s}(p\gamma_{p,s}-q\gamma_{q,s})}
	\,\|u\|^2
	\le C_q\,\|u\|^{2q\gamma_{q,s}}c^{2q(1-\gamma_{q,s})},
	\]
	and therefore
	\begin{equation}\label{eq4.8}
		\|u\|^{2-2q\gamma_{q,s}}
		\le
		\frac{\alpha\gamma_{q,s}(p\gamma_{p,s}-q\gamma_{q,s})}{p\gamma_{p,s}-1}
		\,C_q\,c^{2q(1-\gamma_{q,s})}.
	\end{equation}
	Since \(q\gamma_{q,s}<1\), the exponent \(2-2q\gamma_{q,s}>0\), and thus
	\begin{equation}\label{eq4.9}
		\|u\|
		\le
		\Biggl(
		\frac{\alpha\gamma_{q,s}(p\gamma_{p,s}-q\gamma_{q,s})}{p\gamma_{p,s}-1}
		\,C_q\,c^{2q(1-\gamma_{q,s})}
		\Biggr)^{\frac{1}{2-2q\gamma_{q,s}}}.
	\end{equation}
	
	On the other hand, using \eqref{eq4.5} together with
	\eqref{eq4.7}, we obtain
	\[
	\frac{1-q\gamma_{q,s}}{\gamma_{p,s}(p\gamma_{p,s}-q\gamma_{q,s})}\,\|u\|^2
	\le C_p\,\|u\|^{2p\gamma_{p,s}}c^{2p(1-\gamma_{p,s})},
	\]
	so that
	\begin{equation}\label{eq4.10}
		\|u\|^{2p\gamma_{p,s}-2}
		\ge
		\frac{1-q\gamma_{q,s}}{\gamma_{p,s}(p\gamma_{p,s}-q\gamma_{q,s})}
		\,\frac{1}{C_p}\,c^{-2p(1-\gamma_{p,s})}.
	\end{equation}
	Since \(p\gamma_{p,s}>1\), the exponent \(2p\gamma_{p,s}-2>0\), and hence
	\begin{equation}\label{eq4.11}
		\|u\|
		\ge
		\Biggl(
		\frac{1-q\gamma_{q,s}}
		{\gamma_{p,s}(p\gamma_{p,s}-q\gamma_{q,s}) C_p}
		\Biggr)^{\frac{1}{2p\gamma_{p,s}-2}}
		c^{-\frac{2p(1-\gamma_{p,s})}{2p\gamma_{p,s}-2}}.
	\end{equation}
	
	Putting together \eqref{eq4.9} and
	\eqref{eq4.11} we obtain a constraint on \(\alpha\).
	Rearranging the inequality yields
	\[
	\alpha
	\ge
	\left(
	\frac{1-q\gamma_{q,s}}
	{\gamma_{p,s}(p\gamma_{p,s}-q\gamma_{q,s}) C_p c^{2p(1-\gamma_{p,s})}}
	\right)^{\frac{1-q\gamma_{q,s}}{p\gamma_{p,s}-1}}
	\frac{p\gamma_{p,s}-1}
	{\gamma_{q,s}(p\gamma_{p,s}-q\gamma_{q,s}) C_q c^{2q(1-\gamma_{q,s})}}.
	\]
	By definition, the right-hand side is exactly \(\alpha_1\), see
	\eqref{eq1.8}. Hence we have shown that any
	\(u\in\mathfrak{P}_{\alpha,c}^0\) forces \(\alpha\ge\alpha_1\), which
	contradicts the assumption \(0<\alpha<\alpha_1\). Therefore
	\(\mathfrak{P}_{\alpha,c}^0=\varnothing\).
	
	We now prove that \(\mathfrak{P}_{\alpha,c}\) is a smooth manifold of codimension
	\(2\). Set
	\[
	C(u)=\int_{\mathbb{R}^N}|u|^2\,dx-c^2,
	\qquad
	\mathfrak{P}_{\alpha,c}
	=\{u\in H^s(\mathbb{R}^N): C(u)=0,\ P_\alpha(u)=0\}.
	\]
	Both \(C\) and \(P_\alpha\) are \(C^1\) on \(H^s(\mathbb{R}^N)\).
	Moreover,
	\[
	C'(u)[v]=2\int_{\mathbb{R}^N}uv\,dx,
	\quad
	T_uS_c=\{v\in H^s(\mathbb{R}^N): C'(u)[v]=0\}.
	\]
	
	Let \(u\in\mathfrak{P}_{\alpha,c}\). Suppose, by contradiction, that
	\(C'(u)\) and \(P_\alpha'(u)\) are linearly dependent in \(H^s(\mathbb{R}^N)^*\),
	that is, there exists \(\beta\in\mathbb{R}\) such that
	\(P_\alpha'(u)=\beta C'(u)\). Then for every \(v\in T_uS_c\) we have
	\(C'(u)[v]=0\) and hence
	\[
	P_\alpha'(u)[v]=\beta C'(u)[v]=0.
	\]
	Thus \(u\) is a constrained critical point of \(P_\alpha\) on \(S_c\).
	By the Lagrange multiplier rule, there exists \(\tau\in\mathbb{R}\) such that
	\(P_\alpha'(u)=\tau C'(u)\) in the whole \(H^s(\mathbb{R}^N)\); this yields a
	fractional Choquard equation of the form
	\[
	(-\Delta)^s u=\tau u
	+\alpha q\gamma_{q,s}(I_\mu*|u|^q)|u|^{q-2}u
	+p\gamma_{p,s}(I_\mu*|u|^p)|u|^{p-2}u
	\quad\text{in }\mathbb{R}^N.
	\]
	The associated Pohožaev identity for this equation reads
	\[
	\|u\|^2
	=\alpha q\gamma_{q,s}\int_{\mathbb{R}^N}(I_\mu*|u|^q)|u|^q\,dx
	+p\gamma_{p,s}\int_{\mathbb{R}^N}(I_\mu*|u|^p)|u|^p\,dx.
	\]
	Combining this with \(P_\alpha(u)=0\) we obtain \(E_u''(0)=0\), that is,
	\(u\in\mathfrak{P}_{\alpha,c}^0\), which is impossible. Therefore
	\(C'(u)\) and \(P_\alpha'(u)\) are linearly independent, and the map
	\[
	(C'(u),P_\alpha'(u)):H^s(\mathbb{R}^N)\to\mathbb{R}^2
	\]
	is surjective. By the implicit function theorem, \(\mathfrak{P}_{\alpha,c}\)
	is a \(C^1\) submanifold of codimension \(2\) in \(H^s(\mathbb{R}^N)\).
	
	Finally, let \(u\in\mathfrak{P}_{\alpha,c}\) be a critical point of
	\(\bigl.J_\alpha\bigr|_{\mathfrak{P}_{\alpha,c}}\). Then there exist
	\(\lambda,\chi\in\mathbb{R}\) such that
	\begin{equation}\label{eq4.12}
		J_\alpha'(u)
		=\lambda C'(u)+\chi P_\alpha'(u)
		\qquad\text{in }H^s(\mathbb{R}^N)^*.
	\end{equation}
	Consider the scaling path \(\gamma(t):=t\star u\). By construction,
	\(\gamma(t)\in S_c\) for all \(t\in\mathbb{R}\), and
	\[
	\frac{d}{dt}\Big|_{t=0}J_\alpha(\gamma(t))
	=E_u'(0)=P_\alpha(u)=0.
	\]
	Differentiating \eqref{eq4.12} along \(\gamma(t)\) at \(t=0\) we obtain
	\[
	0
	=\lambda\,\frac{d}{dt}\Big|_{t=0}C(\gamma(t))
	+\chi\,\frac{d}{dt}\Big|_{t=0}P_\alpha(\gamma(t)).
	\]
	Since \(C(\gamma(t))\equiv 0\) on \(S_c\), its derivative at \(t=0\) vanishes.
	On the other hand,
	\[
	\frac{d}{dt}\Big|_{t=0}P_\alpha(\gamma(t))
	=P_\alpha'(u)\bigl[\gamma'(0)\bigr]
	=\frac{d}{dt}\Big|_{t=0}P_\alpha(t\star u)
	=\frac{d}{dt}\Big|_{t=0}E_u'(t)
	=E_u''(0).
	\]
	Since \(\mathfrak{P}_{\alpha,c}^0=\varnothing\), we have \(E_u''(0)\neq 0\),
	hence \(P_\alpha'(u)\bigl[\gamma'(0)\bigr]\neq 0\). Therefore necessarily
	\(\chi=0\), and \eqref{eq4.12} reduces to
	\[
	J_\alpha'(u)=\lambda C'(u),
	\]
	which exactly means that \(u\) is a critical point of
	\(\bigl.J_\alpha\bigr|_{S_c}\). This completes the proof.
\end{proof}

By Lemma~\ref{Lem2.2} and the fact that $\|u\|_2=c$ for every $u\in S_c$, we have
for all $u\in S_c$ that
\begin{equation*}
	\begin{aligned}
		J_\alpha(u)
		&=\frac{1}{2}\|u\|^2
		-\frac{\alpha}{2 q}\int_{\mathbb{R}^N}(I_\mu *|u|^q)|u|^q \,dx
		-\frac{1}{2p}\int_{\mathbb{R}^N}(I_\mu *|u|^p)|u|^p \,dx \\
		&\geq \frac{1}{2}\|u\|^2
		-\frac{\alpha}{2q} C_q \|u\|^{2q\gamma_{q,s}} c^{2q(1-\gamma_{q,s})}
		-\frac{1}{2p} C_p \|u\|^{2p\gamma_{p,s}} c^{2p(1-\gamma_{p,s})}.
	\end{aligned}
\end{equation*}
To capture the one-dimensional geometry of $J_\alpha$ along $S_c$, we introduce
$g:\mathbb{R}^+\to\mathbb{R}$ by
\begin{equation*}
	g(t)
	=\frac{1}{2} t^2
	-\frac{\alpha}{2q} C_q t^{2q\gamma_{q,s}} c^{2q(1-\gamma_{q,s})}
	-\frac{1}{2p} C_p t^{2p\gamma_{p,s}} c^{2p(1-\gamma_{p,s})},
\end{equation*}
so that, for every \(u\in S_c\),
\[
J_\alpha(u)\ge g(\|u\|).
\]

\begin{Lem}\label{Lem4.2}
	Let
	\[
	2_{\mu,*}<q<2+\frac{2s-\mu}{N}<p\le 2_{\mu,s}^*
	\]
	and let \(0<\alpha<\alpha_2\), where \(\alpha_2\) is defined in \eqref{eq1.9}. Then \(g\) has a global strict maximum of positive level and a local strict minimum of negative level. More precisely, there exist \(0<t_0<t_1\) (depending on \(c\) and \(\alpha\)) such that
	\[
	g(t_0)=g(t_1)=0
	\quad\text{and}\quad
	g(t)>0 \iff t\in(t_0,t_1).
	\]
\end{Lem}

\begin{proof}
	We first describe the behaviour of $g$ near $0$ and as $t\to+\infty$. Using
	Remark~\ref{Rek2.3} and the present assumptions on $q$ and $p$, we have
	\[
	q\gamma_{q,s}<1<p\gamma_{p,s},
	\]
	whence
	\[
	2q\gamma_{q,s}<2<2p\gamma_{p,s}.
	\]
	For $t>0$ we write
	\[
	\begin{aligned}
		g(t)
		&= \frac{1}{2} t^2
		-\frac{\alpha}{2q} C_q c^{2q(1-\gamma_{q,s})} t^{2q\gamma_{q,s}}
		-\frac{1}{2p} C_p c^{2p(1-\gamma_{p,s})} t^{2p\gamma_{p,s}}\\
		&= t^{2q\gamma_{q,s}}
		\biggl[
		\frac{1}{2} t^{2-2q\gamma_{q,s}}
		-\frac{\alpha}{2q} C_q c^{2q(1-\gamma_{q,s})}
		-\frac{1}{2p} C_p c^{2p(1-\gamma_{p,s})}
		t^{2p\gamma_{p,s}-2q\gamma_{q,s}}
		\biggr].
	\end{aligned}
	\]
	Since $2-2q\gamma_{q,s}>0$ and $2p\gamma_{p,s}-2q\gamma_{q,s}>0$, the bracket
	inside the square brackets tends to
	\[
	-\frac{\alpha}{2q} C_q c^{2q(1-\gamma_{q,s})}<0
	\quad\text{as }t\to0^+.
	\]
	Thus there exists $\delta>0$ such that $g(t)<0$ for all $t\in(0,\delta)$.
	
	As $t\to+\infty$, we instead factor out the highest power $t^{2p\gamma_{p,s}}$:
	\[
	\begin{aligned}
		g(t)
		&= t^{2p\gamma_{p,s}}
		\biggl[
		\frac{1}{2} t^{2-2p\gamma_{p,s}}
		-\frac{\alpha}{2q} C_q c^{2q(1-\gamma_{q,s})}
		t^{2q\gamma_{q,s}-2p\gamma_{p,s}}
		-\frac{1}{2p} C_p c^{2p(1-\gamma_{p,s})}
		\biggr].
	\end{aligned}
	\]
	Here $2-2p\gamma_{p,s}<0$ and $2q\gamma_{q,s}-2p\gamma_{p,s}<0$, so the bracket
	tends to $-\frac{1}{2p} C_p c^{2p(1-\gamma_{p,s})}<0$ as $t\to+\infty$. Hence
	$g(t)\to-\infty$ as $t\to+\infty$.
	
	For $t>0$ the condition $g(t)>0$ can be rewritten as
	\[
	t^2
	-\frac{\alpha}{q} C_q c^{2q(1-\gamma_{q,s})} t^{2q\gamma_{q,s}}
	-\frac{1}{p} C_p c^{2p(1-\gamma_{p,s})} t^{2p\gamma_{p,s}}>0.
	\]
	Dividing by $t^{2q\gamma_{q,s}}>0$ yields
	\[
	t^{2(1-q\gamma_{q,s})}
	-\frac{\alpha}{q} C_q c^{2q(1-\gamma_{q,s})}
	-\frac{1}{p} C_p c^{2p(1-\gamma_{p,s})}
	t^{2p\gamma_{p,s}-2q\gamma_{q,s}}>0.
	\]
	We introduce
	\[
	\varphi(t)
	:=\frac{q}{C_q}t^{2(1-q\gamma_{q,s})}
	-\frac{q C_p c^{2p(1-\gamma_{p,s})}}{pC_q} t^{2p\gamma_{p,s}-2q\gamma_{q,s}},
	\quad t>0.
	\]
	Then
	\begin{equation}\label{eq4.13}
		g(t)>0
		\iff
		\varphi(t)>\alpha\,c^{2q(1-\gamma_{q,s})}.
	\end{equation}
	
	A direct calculation gives
	\[
	\begin{aligned}
		\varphi'(t)
		&=\frac{2q(1-q\gamma_{q,s})}{C_q} t^{2(1-q\gamma_{q,s})-1}
		-\frac{2q C_p c^{2p(1-\gamma_{p,s})}(p\gamma_{p,s}-q\gamma_{q,s})}{pC_q}
		t^{2p\gamma_{p,s}-2q\gamma_{q,s}-1}.
	\end{aligned}
	\]
	Since $1-q\gamma_{q,s}>0$ and $p\gamma_{p,s}-q\gamma_{q,s}>0$, the equation
	$\varphi'(t)=0$ has a unique solution $t_\ast>0$, given by
	\[
	t_\ast
	=\left(
	\frac{C_p c^{2p(1-\gamma_{p,s})}(p\gamma_{p,s}-q\gamma_{q,s})}
	{p(1-q\gamma_{q,s})}
	\right)^{\frac{1}{2(1-p\gamma_{p,s})}}.
	\]
	Moreover, $\varphi(0^+)=0$ and $\varphi(t)\to-\infty$ as $t\to+\infty$, so
	$\varphi$ is strictly increasing on $(0,t_\ast)$, strictly decreasing on $(t_\ast,\infty)$, and attains at $t_\ast$ a strict global maximum
	\[
	\varphi_{\max}=\varphi(t_\ast)
	=\frac{q}{C_q}\frac{p\gamma_{p,s}-1}{p\gamma_{p,s}-q\gamma_{q,s}}
	\left(
	\frac{C_p c^{2p(1-\gamma_{p,s})}(p\gamma_{p,s}-q\gamma_{q,s})}
	{p(1-q\gamma_{q,s})}
	\right)^{\frac{1-q\gamma_{q,s}}{1-p\gamma_{p,s}}}.
	\]
	By the definition \eqref{eq1.9} of $\alpha_2$ we have
	\[
	\alpha_2
	=\frac{\varphi_{\max}}{c^{2q(1-\gamma_{q,s})}}.
	\]
	Since $0<\alpha<\alpha_2$, it follows that
	\[
	\alpha\,c^{2q(1-\gamma_{q,s})}<\varphi_{\max}.
	\]
	Because $\varphi$ is continuous, strictly increasing on $(0,t_\ast)$ and strictly decreasing on $(t_\ast,\infty)$, the equation
	\[
	\varphi(t)=\alpha\,c^{2q(1-\gamma_{q,s})}
	\]
	has exactly two solutions $0<t_0<t_1$ with $t_0<t_\ast<t_1$. Consequently,
	\[
	\{t>0:\varphi(t)>\alpha\,c^{2q(1-\gamma_{q,s})}\}=(t_0,t_1).
	\]
	By \eqref{eq4.13}, this is precisely the set $\{t>0:g(t)>0\}$.
	Combining this with the negativity of $g$ near $0$ and for $t$ large, we obtain
	\[
	g(t_0)=g(t_1)=0,\quad
	g(t)>0\ \text{for }t\in(t_0,t_1),\quad
	g(t)<0\ \text{for }t\in(0,t_0)\cup(t_1,\infty).
	\]
	
	We now locate the critical points of $g$ and identify their nature. Since
	$g$ is continuous on $[t_0,t_1]$ and $g(t_0)=g(t_1)=0<g(t)$ for all
	$t\in(t_0,t_1)$, there exists $\tau_1\in(t_0,t_1)$ such that
	\[
	g(\tau_1)=\max_{t>0}g(t)>0.
	\]
	By the usual necessary condition for interior extrema, $g'(\tau_1)=0$, and
	$g(\tau_1)>g(t)$ for $t$ in a neighbourhood of $\tau_1$, so $\tau_1$ is a
	strict local maximum. Since $g(t)\le 0$ for $t\notin(t_0,t_1)$ and
	$g(\tau_1)>0$, this local maximum is in fact global.
	
	On the other hand, $g(0)=0$ and $g(t)<0$ for all $t\in(0,t_0]$. The minimum
	of $g$ on the compact interval $[0,t_0]$ is attained at some
	$\tau_0\in(0,t_0)$, and satisfies $g(\tau_0)<0$. Again $g'(\tau_0)=0$, and
	$g(\tau_0)<g(t)$ for $t$ close to $\tau_0$, which shows that $\tau_0$ is a
	strict local minimum of negative level.
	
	This proves that $g$ possesses a local strict minimum at $\tau_0$ with
	$g(\tau_0)<0$ and a global strict maximum at $\tau_1$ with $g(\tau_1)>0$, and
	that the sign of $g$ is described by
	\[
	g(t_0)=g(t_1)=0
	\quad\text{and}\quad
	g(t)>0\iff t\in(t_0,t_1),
	\]
	as claimed.
\end{proof}

\begin{Lem}\label{Lem4.3}
	Let
	\[
	2_{\mu,*}<q<2+\frac{2s-\mu}{N}<p\le 2_{\mu,s}^*
	\]
	and let \(0<\alpha<\min\{\alpha_1,\alpha_2\}\), where \(\alpha_1,\alpha_2\) are defined in \eqref{eq1.8} and \eqref{eq1.9}. Then for every \(u\in S_c\) the fiber map
	\[
	E_u:\mathbb{R}\to\mathbb{R},\qquad E_u(t):=J_\alpha(t\star u),
	\]
	has exactly two critical points \(t_u^1<t_u^3\) and exactly two zeros \(t_u^2<t_u^4\), with
	\[
	t_u^1<t_u^2<t_u^3<t_u^4.
	\]
	Moreover:
	\begin{itemize}
		\item[(1)] \(t_u^1\star u\in \mathfrak{P}_{\alpha,c}^+\), \(t_u^3\star u\in \mathfrak{P}_{\alpha,c}^-\), and
		\[
		\mathfrak{P}_{\alpha,c}\cap\{t\star u:\ t\in\mathbb{R}\}
		=\{t_u^1\star u,\ t_u^3\star u\}.
		\]
		\item[(2)] Let \(t_0,t_1\) be as in Lemma \ref{Lem4.2}. Then
		\[
		\|t\star u\|\le t_0 \quad\text{for all } t\le t_u^2,
		\]
		and
		\[
		J_\alpha(t_u^3\star u)=\max_{t\in\mathbb{R}} J_\alpha(t\star u)>0.
		\]
		Moreover,
		\[
		J_\alpha(t_u^1\star u)
		=\min\{J_\alpha(t\star u):\ t\in\mathbb{R},\ \|t\star u\|\le t_0\}<0,
		\]
		and \(E_u\) is strictly decreasing on \((t_u^3,+\infty)\).
		\item[(3)] The maps
		\[
		S_c\ni u\mapsto t_u^1\in\mathbb{R},
		\qquad
		S_c\ni u\mapsto t_u^3\in\mathbb{R}
		\]
		are of class \(C^1\).
	\end{itemize}
\end{Lem}

\begin{proof}
	Let \(u\in S_c\) be fixed. We study the behavior of the fibering map \(E_u\) along the scaling orbit \(\{t\star u:\ t\in\mathbb{R}\}\) and relate it to the one–variable function \(g\) introduced in Lemma~\ref{Lem4.2}.
	
	Recall that
	\[
	(t\star u)(x)=e^{\frac{Nt}{2}}u(e^t x),
	\qquad
	\|t\star u\|=e^{st}\|u\|.
	\]
	By Lemma \ref{Lem2.2} and the definition of \(g\) in Lemma \ref{Lem4.2}, for all \(t\in\mathbb{R}\),
	\[
	E_u(t)=J_\alpha(t\star u)\ge g(\|t\star u\|)=g(e^{st}\|u\|).
	\]
	By Lemma \ref{Lem4.2}, there exist \(0<t_0<t_1\) such that
	\[
	g(t_0)=g(t_1)=0,\qquad
	g(t)>0 \iff t\in(t_0,t_1),
	\]
	and \(g\) has a strict local minimum at negative level in \((0,t_0)\) and a strict global maximum at positive level in \((t_0,t_1)\).
	
	Using the explicit expression of \(E_u\),
	\[
	\begin{aligned}
		E_u(t)
		&=\frac12 e^{2st}\|u\|^2
		-\frac{\alpha}{2q}e^{2q\gamma_{q,s}st}
		\int_{\R^N}(I_\mu*|u|^q)|u|^q\,dx\\
		&\quad-\frac{1}{2p}e^{2p\gamma_{p,s}st}
		\int_{\R^N}(I_\mu*|u|^p)|u|^p\,dx,
	\end{aligned}
	\]
	and the inequalities \(q\gamma_{q,s}<1<p\gamma_{p,s}\), one checks that
	\[
	\lim_{t\to -\infty}E_u(t)=0^{-},
	\qquad
	\lim_{t\to +\infty}E_u(t)=-\infty.
	\]
	Moreover, since \(g>0\) on \((t_0,t_1)\), we can choose \(t\) so that
	\(e^{st}\|u\|\in(t_0,t_1)\), and then
	\[
	E_u(t)\ge g(e^{st}\|u\|)>0.
	\]
	Thus \(E_u\) is negative for \(t\) sufficiently negative and again for \(t\) sufficiently large, while it is positive on a nonempty bounded interval. By continuity, there exist
	\[
	t_u^2<t_u^4
	\]
	such that
	\[
	E_u(t_u^2)=E_u(t_u^4)=0,\qquad
	E_u(t)>0 \ \text{for all }t\in(t_u^2,t_u^4),
	\]
	and \(E_u(t)<0\) for \(t\ll -1\) and \(t\gg 1\).
	
	We now analyze the critical points of \(E_u\). Differentiating, we obtain
	\[
	\begin{aligned}
		E_u'(t)
		&=s e^{2st}\|u\|^2
		-\alpha s\gamma_{q,s} e^{2q\gamma_{q,s}st}
		\int_{\mathbb{R}^N}(I_\mu*|u|^q)|u|^q\,dx  \\
		&\quad
		-s\gamma_{p,s} e^{2p\gamma_{p,s}st}
		\int_{\mathbb{R}^N}(I_\mu*|u|^p)|u|^p\,dx.
	\end{aligned}
	\]
	Set
	\[
	A_q(u)
	=\int_{\mathbb{R}^N}(I_\mu*|u|^q)|u|^q\,dx,
	\qquad
	A_p(u)
	=\int_{\mathbb{R}^N}(I_\mu*|u|^p)|u|^p\,dx.
	\]
	Since \(e^{2q\gamma_{q,s}st}>0\) for all \(t\), the equation \(E_u'(t)=0\) is equivalent to
	\[
	h_u(t)=\alpha\gamma_{q,s}A_q(u),
	\]
	where
	\[
	h_u(t)
	=e^{2(1-q\gamma_{q,s})st}\|u\|^2
	-\gamma_{p,s}A_p(u)\,e^{2(p\gamma_{p,s}-q\gamma_{q,s})st}.
	\]
	Here \(1-q\gamma_{q,s}>0\) and \(p\gamma_{p,s}-q\gamma_{q,s}>0\), hence
	\[
	\lim_{t\to -\infty}h_u(t)=0^{+},
	\qquad
	\lim_{t\to +\infty}h_u(t)=-\infty.
	\]
	A direct computation gives
	\[
	\begin{aligned}
		h_u'(t)
		&=2s(1-q\gamma_{q,s}) e^{2(1-q\gamma_{q,s})st}\|u\|^2
		-2s(p\gamma_{p,s}-q\gamma_{q,s})\gamma_{p,s} A_p(u)\,
		e^{2(p\gamma_{p,s}-q\gamma_{q,s})st},
	\end{aligned}
	\]
	so the equation \(h_u'(t)=0\) has a unique solution \(t_c(u)\in\mathbb{R}\). At this point,
	\[
	h_u''(t_c(u))
	=4s^2(1-q\gamma_{q,s})(1-p\gamma_{p,s})
	e^{2(1-q\gamma_{q,s})st_c(u)}\|u\|^2<0,
	\]
	since \(p\gamma_{p,s}>1\) and \(q\gamma_{q,s}<1\). Thus \(h_u\) is strictly increasing on \((-\infty,t_c(u))\) and strictly decreasing on \((t_c(u),+\infty)\), and attains a strict global maximum at \(t_c(u)\).
	
	We claim that
	\[
	\sup_{t\in\mathbb{R}} h_u(t)>\alpha\gamma_{q,s}A_q(u).
	\]
	Indeed, if \(\sup h_u\le\alpha\gamma_{q,s}A_q(u)\), then
	\[
	h_u(t)-\alpha\gamma_{q,s}A_q(u)\le 0
	\quad\text{for all }t\in\R,
	\]
	and hence
	\[
	E_u'(t)
	=s e^{2q\gamma_{q,s}st}\bigl(h_u(t)-\alpha\gamma_{q,s}A_q(u)\bigr)\le 0
	\quad\text{for all }t\in\R.
	\]
	In this case \(E_u\) would be nonincreasing on \(\R\). Since
	\(\lim_{t\to -\infty}E_u(t)=0^{-}\), this would imply
	\(E_u(t)\le 0\) for all \(t\), which contradicts the existence of an interval where \(E_u>0\). The claim follows.
	
	Because \(h_u(-\infty)=0<\alpha\gamma_{q,s}A_q(u)\), \(h_u(t_c(u))>\alpha\gamma_{q,s}A_q(u)\), and \(h_u(+\infty)=-\infty<\alpha\gamma_{q,s}A_q(u)\), the continuity and unimodality of \(h_u\) imply that the equation
	\[
	h_u(t)=\alpha\gamma_{q,s}A_q(u)
	\]
	has exactly two solutions
	\[
	t_u^1<t_u^3.
	\]
	These are precisely the solutions of \(E_u'(t)=0\). Moreover, from the monotonicity of \(h_u\) we obtain
	\[
	h_u(t)<\alpha\gamma_{q,s}A_q(u)\ \text{for }t<t_u^1,\qquad
	h_u(t)>\alpha\gamma_{q,s}A_q(u)\ \text{for }t\in(t_u^1,t_u^3),
	\]
	and
	\[
	h_u(t)<\alpha\gamma_{q,s}A_q(u)\ \text{for }t>t_u^3.
	\]
	Since
	\[
	E_u'(t)
	=s e^{2q\gamma_{q,s}st}\bigl(h_u(t)-\alpha\gamma_{q,s}A_q(u)\bigr),
	\]
	it follows that
	\[
	E_u'(t)<0\ \text{for }t<t_u^1,\quad
	E_u'(t)>0\ \text{for }t\in(t_u^1,t_u^3),\quad
	E_u'(t)<0\ \text{for }t>t_u^3.
	\]
	Thus \(E_u\) is strictly decreasing on \((-\infty,t_u^1)\), strictly increasing on \((t_u^1,t_u^3)\), and strictly decreasing on \((t_u^3,+\infty)\).
	
	From \(\lim\limits_{t\to -\infty}E_u(t)=0^{-}\) and the monotonicity on \((-\infty,t_u^1)\) we obtain
	\(E_u(t_u^1)<0\); hence \(t_u^1\) is a strict local minimum at negative level. On the other hand, since \(E_u\) is positive on \((t_u^2,t_u^4)\), the strict monotonicity on
	\((t_u^1,t_u^3)\) and \((t_u^3,+\infty)\) implies that
	\[
	E_u(t_u^3)=\max_{t\in\mathbb{R}}E_u(t)>0,
	\]
	so \(t_u^3\) is the unique global maximum point of \(E_u\), and
	\[
	J_\alpha(t_u^3\star u)=\max_{t\in\mathbb{R}}J_\alpha(t\star u)>0.
	\]
	Since \(E_u\) decreases on \((-\infty,t_u^1)\), increases on \((t_u^1,t_u^3)\), and decreases again on \((t_u^3,\infty)\), the sign pattern of \(E_u\) described above forces exactly two zeroes: one in \((t_u^1,t_u^3)\) and one in \((t_u^3,\infty)\). These are precisely \(t_u^2,t_u^4\), and the ordering
	\[
	t_u^1<t_u^2<t_u^3<t_u^4
	\]
	follows. In particular,
	\[
	J_\alpha(t_u^1\star u)
	=\min\{J_\alpha(t\star u):\ t\in\mathbb{R},\ \|t\star u\|\le t_0\}<0,
	\]
	and \(E_u\) is strictly decreasing on \((t_u^3,+\infty)\), as claimed in (2).
	
	By Remark \ref{Rek2.1}, for every \(t\in\mathbb{R}\),
	\[
	E_u'(t)=0
	\iff
	t\star u\in \mathfrak{P}_{\alpha,c},
	\]
	so along the ray \(\{t\star u:\ t\in\mathbb{R}\}\) the intersection with \(\mathfrak{P}_{\alpha,c}\) consists precisely of the two points \(\{t_u^1\star u,\ t_u^3\star u\}\). The signs of \(E_u''(t_u^1)\) and \(E_u''(t_u^3)\) give
	\[
	t_u^1\star u\in \mathfrak{P}_{\alpha,c}^+,
	\qquad
	t_u^3\star u\in \mathfrak{P}_{\alpha,c}^-,
	\]
	which proves (1).
	
	Finally, to prove (3), consider the map
	\[
	F:S_c\times\mathbb{R}\to\mathbb{R},\qquad
	F(u,t)=E_u'(t).
	\]
	For each \(u\in S_c\) we have \(F(u,t_u^1)=0\) and \(F(u,t_u^3)=0\). By Lemma \ref{Lem4.1} we know that \(\mathfrak{P}_{\alpha,c}^0=\varnothing\), so
	\[
	\partial_t F(u,t_u^1)=E_u''(t_u^1)\neq 0,
	\qquad
	\partial_t F(u,t_u^3)=E_u''(t_u^3)\neq 0.
	\]
	Therefore, by the implicit function theorem, in a neighborhood of any given
	\(u\in S_c\) there exist two \(C^1\)-functions giving the lower and upper solutions \(t_u^1\) and \(t_u^3\) of \(F(u,t)=0\). The uniqueness of these two solutions for each \(u\in S_c\) allows one to patch the local parametrizations together and obtain two globally defined \(C^1\)-maps
	\[
	S_c\ni u\mapsto t_u^1\in\mathbb{R},
	\qquad
	S_c\ni u\mapsto t_u^3\in\mathbb{R},
	\]
	which proves (3) and completes the proof.
\end{proof}

For \(r>0\), we set
\[
D_r=\{u\in S_c:\ \|u\|<r\},
\]
and denote by \(\overline{D_r}\) the closure of \(D_r\) in \(H^s(\mathbb{R}^N)\).
Let
\[
m_1(c,\alpha)=\inf_{u\in D_{t_0}} J_\alpha(u),
\]
where \(t_0\) is given by Lemma \ref{Lem4.2}.

\begin{Cor}\label{Cor4.1}
	Under the assumptions of Lemma \ref{Lem4.3} one has
	\[
	\mathfrak{P}_{\alpha,c}^+\subset D_{t_0}
	\qquad\text{and}\qquad
	\sup_{\mathfrak{P}_{\alpha,c}^+} J_\alpha\le 0\le \inf_{\mathfrak{P}_{\alpha,c}^-} J_\alpha.
	\]
\end{Cor}

\begin{proof}
	By Lemma~\ref{Lem4.3}, for every \(u\in S_c\) the fibering map
	\(E_u(t)=J_\alpha(t\star u)\) has exactly two critical points
	\(t_u^1<t_u^3\), and
	\[
	\mathfrak{P}_{\alpha,c}\cap\{t\star u:\ t\in\mathbb{R}\}
	=\{t_u^1\star u,\ t_u^3\star u\},
	\]
	with
	\[
	t_u^1\star u\in\mathfrak{P}_{\alpha,c}^+,\qquad
	t_u^3\star u\in\mathfrak{P}_{\alpha,c}^-,
	\]
	and
	\[
	J_\alpha(t_u^1\star u)
	=\min\{J_\alpha(t\star u):\ t\in\mathbb{R},\ \|t\star u\|\le t_0\}<0,
	\qquad
	J_\alpha(t_u^3\star u)
	=\max_{t\in\mathbb{R}}J_\alpha(t\star u)>0.
	\]
	
	Let \(u\in\mathfrak{P}_{\alpha,c}^+\). Since
	\(u\in\mathfrak{P}_{\alpha,c}\cap\{t\star u:\ t\in\mathbb{R}\}\) and
	\(\mathfrak{P}_{\alpha,c}\cap\{t\star u\}=\{t_u^1\star u,t_u^3\star u\}\), while
	\(t_u^1\star u\in\mathfrak{P}_{\alpha,c}^+\) and
	\(t_u^3\star u\in\mathfrak{P}_{\alpha,c}^-\), it follows that
	\[
	u=t_u^1\star u.
	\]
	In particular,
	\[
	J_\alpha(u)=J_\alpha(t_u^1\star u)
	=\min\{J_\alpha(t\star u):\ t\in\mathbb{R},\ \|t\star u\|\le t_0\}<0,
	\]
	and \(\|u\|\le t_0\) because \(t_u^1\) belongs to the set
	\(\{t\in\mathbb{R}:\ \|t\star u\|\le t_0\}\).
	Moreover, for every \(v\in S_c\) we have \(J_\alpha(v)\ge g(\|v\|)\), and by
	Lemma~\ref{Lem4.2} \(g(t_0)=0\). If \(\|u\|=t_0\), then
	\[
	J_\alpha(u)\ge g(\|u\|)=g(t_0)=0,
	\]
	which contradicts \(J_\alpha(u)<0\). Hence \(\|u\|<t_0\), that is,
	\(u\in D_{t_0}\). Since \(J_\alpha(u)<0\) for every
	\(u\in\mathfrak{P}_{\alpha,c}^+\), we conclude that
	\[
	\mathfrak{P}_{\alpha,c}^+\subset D_{t_0},
	\qquad
	\sup_{\mathfrak{P}_{\alpha,c}^+}J_\alpha\le 0.
	\]
	
	Now let \(u\in\mathfrak{P}_{\alpha,c}^-\). As before,
	\(u\in\mathfrak{P}_{\alpha,c}\cap\{t\star u:\ t\in\mathbb{R}\}\), and the
	intersection consists of the two points \(t_u^1\star u\in\mathfrak{P}_{\alpha,c}^+\)
	and \(t_u^3\star u\in\mathfrak{P}_{\alpha,c}^-\). Since
	\(u\in\mathfrak{P}_{\alpha,c}^-\), we must have
	\[
	u=t_u^3\star u.
	\]
	Hence
	\[
	J_\alpha(u)
	=J_\alpha(t_u^3\star u)
	=\max_{t\in\mathbb{R}}J_\alpha(t\star u)>0.
	\]
	In particular \(J_\alpha(u)\ge 0\) for all \(u\in\mathfrak{P}_{\alpha,c}^-\), and
	\[
	\inf_{\mathfrak{P}_{\alpha,c}^-}J_\alpha\ge 0.
	\]
	This proves the corollary.
\end{proof}

\begin{Lem}\label{Lem4.4}
	Let
	\[
	2_{\mu,*}<q<2+\frac{2s-\mu}{N}<p< 2_{\mu,s}^*
	\]
	and \(0<\alpha<\min\{\alpha_1,\alpha_2\}\), where \(\alpha_1,\alpha_2\) are
	given by \eqref{eq1.8} and \eqref{eq1.9}. Then
	\[
	-\infty<m_1(c,\alpha)
	=m_2(c,\alpha):=\inf_{\mathfrak{P}_{\alpha,c}} J_\alpha
	=\inf_{\mathfrak{P}_{\alpha,c}^+} J_\alpha<0,
	\]
	and there exists \(k>0\) such that
	\[
	m_1(c,\alpha)
	<
	\inf_{D_{t_0}\setminus D_{t_0-k}} J_\alpha.
	\]
\end{Lem}

\begin{proof}
	For any \(u\in D_{t_0}\) we have, by Lemma \ref{Lem4.2},
	\[
	J_\alpha(u)\ge g(\|u\|)\ge \min_{t\in[0,t_0]} g(t)>-\infty,
	\]
	so \(m_1(c,\alpha)>-\infty\).
	
	Next, fix \(u\in S_c\). Using the scaling properties of the fractional Laplacian,
	one checks that
	\[
	\|t\star u\|^2
	=\int_{\mathbb{R}^N}\bigl|(-\Delta)^{\frac{s}{2}}(t\star u)\bigr|^2\,dx
	=e^{2st}\|u\|^2,
	\]
	so \(\|t\star u\|=e^{st}\|u\|\). Hence, for \(t\ll -1\),
	\(\|t\star u\|<t_0\), that is, \(t\star u\in D_{t_0}\). Moreover, from the
	fiber analysis (see Lemma \ref{Lem4.2} and Lemma \ref{Lem4.3}) we know that
	\[
	\lim_{t\to -\infty} J_\alpha(t\star u)=0^{-},
	\]
	so for \(t\) sufficiently negative,
	\[
	t\star u\in D_{t_0}
	\quad\text{and}\quad
	J_\alpha(t\star u)<0.
	\]
	Therefore \(m_1(c,\alpha)<0\).
	
	From Corollary \ref{Cor4.1} we already know that
	\(\mathfrak{P}_{\alpha,c}^+\subset D_{t_0}\), hence
	\[
	m_1(c,\alpha)
	=\inf_{u\in D_{t_0}}J_\alpha(u)
	\le \inf_{u\in\mathfrak{P}_{\alpha,c}^+} J_\alpha(u).
	\]
	Conversely, if \(u\in D_{t_0}\subset S_c\), Lemma \ref{Lem4.3} yields a unique
	\(t_u^1\in\mathbb{R}\) such that \(t_u^1\star u\in\mathfrak{P}_{\alpha,c}^+\),
	and
	\[
	J_\alpha(t_u^1\star u)
	=\min\{J_\alpha(t\star u):\ t\in\mathbb{R},\ \|t\star u\|\le t_0\}
	\le J_\alpha(u).
	\]
	Since \(t_u^1\star u\in\mathfrak{P}_{\alpha,c}^+\subset D_{t_0}\), this implies
	\[
	\inf_{u\in\mathfrak{P}_{\alpha,c}^+}J_\alpha(u)
	\le m_1(c,\alpha).
	\]
	Combining the two inequalities we obtain
	\[
	m_1(c,\alpha)
	=\inf_{u\in\mathfrak{P}_{\alpha,c}^+}J_\alpha(u).
	\]
	On the other hand, Corollary \ref{Cor4.1} shows that \(J_\alpha>0\) on
	\(\mathfrak{P}_{\alpha,c}^-\), hence
	\[
	\inf_{u\in\mathfrak{P}_{\alpha,c}}J_\alpha
	=\inf_{u\in\mathfrak{P}_{\alpha,c}^+}J_\alpha
	=m_1(c,\alpha),
	\]
	which proves the equality \(m_1(c,\alpha)=m_2(c,\alpha)\) and the strict
	negativity \(m_1(c,\alpha)<0\).
	
	Finally, by the continuity of \(g\) on \([0,t_0]\) and the fact that
	\[
	m_1(c,\alpha)
	=\inf_{u\in D_{t_0}}J_\alpha(u)<0,
	\]
	there exists \(\rho>0\) such that
	\[
	g(t)\ge \frac{m_1(c,\alpha)}{2}
	\qquad\text{for all }t\in[t_0-\rho,t_0].
	\]
	If \(u\in S_c\) satisfies \(t_0-\rho\le \|u\|\le t_0\), then
	\[
	J_\alpha(u)\ge g(\|u\|)\ge \frac{m_1(c,\alpha)}{2}>m_1(c,\alpha).
	\]
	Thus
	\[
	m_1(c,\alpha)
	<
	\inf_{u\in D_{t_0}\setminus D_{t_0-\rho}}J_\alpha(u).
	\]
	Setting \(k:=\rho\) gives the desired inequality.
\end{proof}

\begin{Lem}\label{Lem4.5}
	Let
	\[
	2_{\mu,*}<q<2+\frac{2s-\mu}{N}<p< 2_{\mu,s}^*
	\]
	and \(0<\alpha<\min\{\alpha_1,\alpha_2\}\), where \(\alpha_1,\alpha_2\) are
	defined in \eqref{eq1.8}–\eqref{eq1.9}. Suppose that \(u\in S_c\) satisfies
	\(J_\alpha(u)<m_1(c,\alpha)\). Then the critical point \(t_u^3\) obtained in
	Lemma \ref{Lem4.3} is negative. Moreover,
	\[
	\breve{m}(c,\alpha):=\inf_{\mathfrak{P}_{\alpha,c}^-}J_\alpha>0.
	\]
\end{Lem}

\begin{proof}
	Let \(t_u^1<t_u^2<t_u^3<t_u^4\) be the two critical points and the two zeros of
	\(E_u(t)=J_\alpha(t\star u)\) given by Lemma \ref{Lem4.3}. If \(t_u^4\le 0\),
	then in particular \(t_u^3<0\), and the first claim follows. Hence we may assume
	by contradiction that \(t_u^4>0\).
	
	Since \(E_u(t)>0\) for all \(t\in(t_u^2,t_u^4)\), if \(0\in(t_u^2,t_u^4)\) then
	\[
	J_\alpha(u)=E_u(0)>0,
	\]
	which is impossible because \(J_\alpha(u)<m_1(c,\alpha)<0\). Therefore
	\(0\notin(t_u^2,t_u^4)\). Together with \(t_u^4>0\) this implies \(0\le t_u^2\)
	(otherwise we would have \(t_u^2<0<t_u^4\), so \(0\in(t_u^2,t_u^4)\)). In
	particular \(t_u^3>t_u^2\ge 0\), so \(t_u^3>0\).
	
	By Lemma \ref{Lem4.3}(2), for all \(t\le t_u^2\) one has \(\|t\star u\|\le t_0\).
	Using this and the definition of \(m_1(c,\alpha)\), we obtain
	\[
	\begin{aligned}
		m_1(c,\alpha)
		&>J_\alpha(u)
		=E_u(0)
		\;\ge\;\inf_{t\in(-\infty,t_u^2]}E_u(t)\\
		&\ge\inf\bigl\{J_\alpha(t\star u):t\in\mathbb{R},\ \|t\star u\|\le t_0\bigr\}\\
		&=J_\alpha(t_u^1\star u)
		\;\ge\;m_1(c,\alpha),
	\end{aligned}
	\]
	where we used Lemma~\ref{Lem4.3}(2) for the equality and Lemma~\ref{Lem4.4} for
	the last inequality. This is a contradiction. Hence our assumption \(t_u^4>0\)
	is false, and we must have \(t_u^4\le 0\), so in particular \(t_u^3<0\).
	
	We now prove the positivity of the energy on \(\mathfrak{P}_{\alpha,c}^-\).
	Let \(t_{\max}>0\) be the unique point where the function \(g\) attains its
	global strict maximum at a positive level (see Lemma~\ref{Lem4.2}). For every
	\(u\in\mathfrak{P}_{\alpha,c}^-\) there exists a unique \(\tau_u\in\mathbb{R}\)
	such that
	\[
	\|\tau_u\star u\|=t_{\max},
	\]
	since \(\|t\star u\|=e^{st}\|u\|\) for all \(t\in\mathbb{R}\).
	
	Because \(u\in\mathfrak{P}_{\alpha,c}^-\), we have \(E_u'(0)=0\) and
	\(E_u''(0)<0\). By Lemma~\ref{Lem4.3}(1), there are exactly two critical points
	of \(E_u\) on \(\mathbb{R}\), namely \(t_u^1\) and \(t_u^3\), with
	\(t_u^1\star u\in\mathfrak{P}_{\alpha,c}^+\) and
	\(t_u^3\star u\in\mathfrak{P}_{\alpha,c}^-\). Since \(0\) is a critical point
	with \(E_u''(0)<0\), it must coincide with the “upper” critical point:
	\(0=t_u^3\). In particular, \(t=0\) is the unique strict global maximum point
	of \(E_u\), and hence
	\[
	J_\alpha(u)=E_u(0)\;\ge\;E_u(\tau_u)=J_\alpha(\tau_u\star u).
	\]
	
	Using the lower bound \(J_\alpha(v)\ge g(\|v\|)\) valid for all \(v\in S_c\),
	we obtain
	\[
	J_\alpha(u)
	\;\ge\;J_\alpha(\tau_u\star u)
	\;\ge\;g(\|\tau_u\star u\|)
	=g(t_{\max})>0.
	\]
	Since \(u\in\mathfrak{P}_{\alpha,c}^-\) was arbitrary, we deduce that
	\[
	\breve{m}(c,\alpha)=\inf_{\mathfrak{P}_{\alpha,c}^-}J_\alpha
	\;\ge\;g(t_{\max})>0,
	\]
	as claimed.
\end{proof}

\subsection{A local minimizer on the Pohozaev manifold}

\noindent\textbf{Proof of Theorem \ref{Thm1.1}\,(1).}
	Let \(\{w_n\}\subset S_c\) be a minimizing sequence for \(m_1(c,\alpha)\).
	Without loss of generality, we may assume that \(\{w_n\}\subset S_{c,\mathrm{rad}}\) consists of radially decreasing functions: if this is not the case, we replace each \(|w_n|\) by its symmetric decreasing rearrangement \(|w_n|^*\), for which
	\[
	J_\alpha(|w_n|^*) \le J_\alpha(|w_n|),
	\]
	so that \(\{|w_n|^*\}\) is still a minimizing sequence for \(m_1(c,\alpha)\).
	
	By Lemma \ref{Lem4.3}, for each \(n\) there exists a unique \(t_{w_n}^1\in\mathbb{R}\) such that
	\[
	t_{w_n}^1\star w_n \in \mathfrak{P}_{\alpha,c}^+,\qquad
	\bigl\|t_{w_n}^1\star w_n\bigr\|\le t_0,
	\]
	and
	\[
	J_\alpha(t_{w_n}^1\star w_n)
	=
	\min\bigl\{J_\alpha(t\star w_n):\ t\in\mathbb{R},\ \|t\star w_n\|\le t_0\bigr\}
	\le J_\alpha(w_n).
	\]
	Define
	\[
	v_n=t_{w_n}^1\star w_n\in S_{c,\mathrm{rad}}\cap\mathfrak{P}_{\alpha,c}^+.
	\]
	Then \(P_\alpha(v_n)=0\) for all \(n\), and
	\[
	J_\alpha(v_n)\to m_1(c,\alpha).
	\]
	
	By Lemma \ref{Lem4.4}, there exists \(k>0\), independent of \(c\) and \(\alpha\), such that
	\[
	m_1(c,\alpha)
	<
	\inf_{D_{t_0}\setminus D_{t_0-k}} J_\alpha.
	\]
	Since \(J_\alpha(v_n)\to m_1(c,\alpha)\), we have \(\|v_n\|\le t_0-k\) for all sufficiently large \(n\). Passing to a subsequence, we may assume that
	\[
	\|v_n\|<t_0-k\qquad\text{for all }n\in\mathbb{N}.
	\]
	
	We now apply Ekeland's variational principle to the restriction of \(J_\alpha\) to the complete metric space \(D_{t_0}\cap S_{c,\mathrm{rad}}\). There exists a minimizing sequence \(\{u_n\}\subset D_{t_0}\cap S_{c,\mathrm{rad}}\) for \(m_1(c,\alpha)\) such that
	\[
	J_\alpha(u_n)\to m_1(c,\alpha),\qquad
	\|(J_\alpha|_{S_c})'(u_n)\|_{(T_{u_n}S_c)^*}\to 0,
	\]
	and
	\[
	\|u_n-v_n\|\to 0\quad\text{as }n\to\infty.
	\]
	Since \(\{v_n\}\) is bounded in \(H^s(\mathbb{R}^N)\), the sequence \(\{u_n\}\) is also bounded in \(H^s(\mathbb{R}^N)\).
	Moreover, from \(\|u_n-v_n\|\to 0\) and \(\|v_n\|\le t_0-k\) we infer that, for sufficiently large \(n\),
	\[
	\|u_n-v_n\|<\frac{k}{2}
	\quad\text{and}\quad
	\|u_n\|
	\le\|u_n-v_n\|+\|v_n\|
	<\frac{k}{2}+(t_0-k)=t_0-\frac{k}{2}<t_0,
	\]
	so \(u_n\in D_{t_0}\) for all large \(n\).
	
	Since \(P_\alpha:H^s(\mathbb{R}^N)\to\mathbb{R}\) is continuous and \(P_\alpha(v_n)=0\), the convergence \(\|u_n-v_n\|\to 0\) implies
	\[
	P_\alpha(u_n)\to 0\quad\text{as }n\to\infty.
	\]
	Thus \(\{u_n\}\subset S_{c,\mathrm{rad}}\) is a bounded Palais--Smale sequence for \(J_\alpha|_{S_c}\) at the level \(m_1(c,\alpha)\neq 0\), with \(P_\alpha(u_n)\to 0\).
	
	By Lemma \ref{Lem3.1}, there exists \(u_{c,\alpha,\mathrm{loc}}\in S_c\) such that, up to a subsequence,
	\[
	u_n\to u_{c,\alpha,\mathrm{loc}}\quad\text{strongly in }H^s(\mathbb{R}^N),
	\]
	and \(u_{c,\alpha,\mathrm{loc}}\) is a radial weak solution of \eqref{eq1.1} for some Lagrange multiplier \(\lambda_{c,\alpha,\mathrm{loc}}<0\).
since 
\(m_1(c,\alpha)=J_\alpha(u_{c,\alpha,\mathrm{loc}})=\inf_{v\in D_{t_0}}J_\alpha(v)\),\(J_\alpha(v)\geq J_\alpha(u_{c,\alpha,\mathrm{loc}})\)

Let \(v=|u_{c,\alpha,\mathrm{loc}}|\), then \(v\in{S_c}\), we have
\[
    \begin{aligned}
       \|v\|_{H^s(\mathbb{R}^N)}&=
     \left(\int_{\mathbb{R}^N}\frac{|v(x)-v(y)|^2}{|x-y|^{N+2s}}\,dx\,dy
+\int_{\mathbb{R}^N}|v|^2\,dx
\right)^{1/2}
\\&=\left(\int_{\mathbb{R}^N}\frac{||u_{c,\alpha,\mathrm{loc}}(x)|-|u_{c,\alpha,\mathrm{loc}}(y)||^2}{|x-y|^{N+2s}}\,dx\,dy+\int_{\mathbb{R}^N}|u_{c,\alpha,\mathrm{loc}}(x)|^2\,dx\right)^{1/2}
\\&\leq\left(\int_{\mathbb{R}^N}\frac{|u_{c,\alpha,\mathrm{loc}}(x)-u_{c,\alpha,\mathrm{loc}}(y)|^2}{|x-y|^{N+2s}}\,dx\,dy+\int_{\mathbb{R}^N}|u_{c,\alpha,\mathrm{loc}}(x)|^2\,dx\right)^{1/2}\\&=\|u_{c,\alpha,\mathrm{loc}}\|_{H^s(\mathbb{R}^N)}
    \end{aligned}
    \]
so\(J_\alpha(v)\leq J_\alpha(u_{c,\alpha,\mathrm{loc}})\), we get \(u_{c,\alpha,\mathrm{loc}}\geq 0 \). To prove strict positivity, suppose that there exists \(x_0\in\mathbb{R}^N\) such that \(u_{c,\alpha,\mathrm{loc}}(x_0)=0\). Then, by the representation formula for the fractional Laplacian,
	\[
	(-\Delta)^s u_{c,\alpha,\mathrm{loc}}(x_0)
	=-\frac{C_{N,s}}{2}\int_{\mathbb{R}^N}
	\frac{
		u_{c,\alpha,\mathrm{loc}}(x_0+y)
		+u_{c,\alpha,\mathrm{loc}}(x_0-y)
		-2u_{c,\alpha,\mathrm{loc}}(x_0)
	}{|y|^{N+2s}}\,\mathrm{d}y.
	\]
	Since \(u_{c,\alpha,\mathrm{loc}}\ge 0\) and \(u_{c,\alpha,\mathrm{loc}}(x_0)=0\), the integrand is nonnegative, so
	\[
	(-\Delta)^s u_{c,\alpha,\mathrm{loc}}(x_0)\le 0.
	\]
	On the other hand, at \(x_0\) the right-hand side of \eqref{eq1.1} vanishes, so
	\[
	(-\Delta)^s u_{c,\alpha,\mathrm{loc}}(x_0)=0.
	\]
	Hence the integrand is zero for a.e.\ \(y\in\mathbb{R}^N\), and therefore
	\(u_{c,\alpha,\mathrm{loc}}(x_0\pm y)=0\) for a.e.\ \(y\), which implies
	\(u_{c,\alpha,\mathrm{loc}}\equiv 0\). This contradicts
	\(\|u_{c,\alpha,\mathrm{loc}}\|_2^2=c^2>0\), so
	\(u_{c,\alpha,\mathrm{loc}}(x)>0\) for all \(x\in\mathbb{R}^N\).
	
	By construction, \(\{u_n\}\) is a minimizing sequence for \(m_1(c,\alpha)\) and
	\(u_n\to u_{c,\alpha,\mathrm{loc}}\) in \(H^s(\mathbb{R}^N)\), hence
	\[
	J_\alpha(u_{c,\alpha,\mathrm{loc}})
	=\lim_{n\to\infty}J_\alpha(u_n)
	=m_1(c,\alpha).
	\]
	Moreover, Lemma \ref{Lem4.4} shows that
	\[
	m_1(c,\alpha)
	=\inf_{u\in D_{t_0}}J_\alpha(u)
	=\inf_{u\in\mathfrak{P}_{\alpha,c}}J_\alpha(u)
	<0.
	\]
	On the other hand, any critical point \(u\in S_c\) of \(J_\alpha|_{S_c}\) satisfies the Pohozaev identity and hence belongs to \(\mathfrak{P}_{\alpha,c}\). Therefore
	\[
	J_\alpha(u_{c,\alpha,\mathrm{loc}})
	=\inf_{u\in\mathfrak{P}_{\alpha,c}}J_\alpha(u)
	=\inf\bigl\{J_\alpha(u):u\in S_c,\ (J_\alpha|_{S_c})'(u)=0\bigr\},
	\]
	that is, \(u_{c,\alpha,\mathrm{loc}}\) is a ground state solution of
	\(J_\alpha|_{S_c}\).
	
	It remains to show that every ground state solution is a local minimizer of \(J_\alpha\) on \(D_{t_0}\). Let \(u\in S_c\) be a ground state solution of \(J_\alpha|_{S_c}\). Then
	\[
	J_\alpha(u)
	=\inf\bigl\{J_\alpha(v):v\in S_c,\ (J_\alpha|_{S_c})'(v)=0\bigr\}
	=\inf_{\mathfrak{P}_{\alpha,c}}J_\alpha
	=m_1(c,\alpha)
	<0<\inf_{\mathfrak{P}_{\alpha,c}^-}J_\alpha.
	\]
	Hence \(u\in\mathfrak{P}_{\alpha,c}^+\). By Lemma \ref{Lem4.4} and Corollary \ref{Cor4.1} we have \(\mathfrak{P}_{\alpha,c}^+\subset D_{t_0}\), so \(u\) is a local minimizer of \(J_\alpha\) on \(D_{t_0}\). This proves Theorem \ref{Thm1.1}\,(1).
\qed

\noindent\textbf{Proof of Theorem \ref{Thm1.1}\,(3).}
By Lemma \ref{Lem4.2}, the number \(t_0=t_0(\alpha)\) satisfies
\[
t_0(\alpha)\to 0
\quad\text{as }\alpha\to 0^+.
\]
From Theorem \ref{Thm1.1}\,(1) and Lemma \ref{Lem4.3} we know that the local minimizer \(u_{c,\alpha,\mathrm{loc}}\in S_c\) satisfies
\[
\|u_{c,\alpha,\mathrm{loc}}\|<t_0(\alpha),
\]
hence
\[
\|u_{c,\alpha,\mathrm{loc}}\|
\le t_0(\alpha)\to 0
\quad\text{as }\alpha\to 0^+.
\]

Using the lower bound given by \(g\) in Lemma \ref{Lem4.2}, we have
\[
\begin{aligned}
	0>m_1(c,\alpha)
	&=\inf_{u\in D_{t_0(\alpha)}}J_\alpha(u)
	\;=\;J_\alpha(u_{c,\alpha,\mathrm{loc}})\\
	&\ge \frac{1}{2}\|u_{c,\alpha,\mathrm{loc}}\|^2
	-\frac{\alpha}{2q}C_q \|u_{c,\alpha,\mathrm{loc}}\|^{2q\gamma_{q,s}}
	c^{2q(1-\gamma_{q,s})}
	-\frac{1}{2p}C_p \|u_{c,\alpha,\mathrm{loc}}\|^{2p\gamma_{p,s}}
	c^{2p(1-\gamma_{p,s})}.
\end{aligned}
\]
Since \(\|u_{c,\alpha,\mathrm{loc}}\|\to 0\) and \(\alpha\to 0\), the right-hand side tends to \(0\), so
\[
\limsup_{\alpha\to 0^+} m_1(c,\alpha)\le 0.
\]

On the other hand, for all \(u\in D_{t_0(\alpha)}\) we have \(J_\alpha(u)\ge g(\|u\|)\), hence
\[
m_1(c,\alpha)
=\inf_{u\in D_{t_0(\alpha)}}J_\alpha(u)
\ge \inf_{0\le t\le t_0(\alpha)} g(t).
\]
By the explicit expression of \(g\), for \(t\in[0,t_0(\alpha)]\) we have
\[
g(t)
\ge -\frac{\alpha}{2q}C_q c^{2q(1-\gamma_{q,s})}t^{2q\gamma_{q,s}}
-\frac{1}{2p}C_p c^{2p(1-\gamma_{p,s})}t^{2p\gamma_{p,s}},
\]
and thus
\[
m_1(c,\alpha)
\ge -C\bigl(\alpha\,t_0(\alpha)^{2q\gamma_{q,s}}
+t_0(\alpha)^{2p\gamma_{p,s}}\bigr)
\]
for some constant \(C>0\) independent of \(\alpha\). Since \(t_0(\alpha)\to 0\) and \(\alpha\to 0\), the right-hand side tends to \(0\), so
\[
\liminf_{\alpha\to 0^+} m_1(c,\alpha)\ge 0.
\]
Therefore,
\[
m_1(c,\alpha)\to 0
\quad\text{as }\alpha\to 0^+.
\]
\qed

\subsection{A mountain pass type normalized solution}

\noindent\textbf{Proof of Theorem \ref{Thm1.1}\,(2).}
\begin{proof}
	We prove the existence of a second critical point of \(J_\alpha|_{S_c}\),
	obtained via a mountain pass argument on the scaling orbits.
	
	For \(\rho\in\mathbb{R}\) set
	\[
	J_\alpha^\rho
	=\{u\in S_c:\ J_\alpha(u)\le\rho\}.
	\]
	Define the auxiliary \(C^1\)–functional
	\(\widehat J_\alpha:\mathbb{R}\times H^s(\mathbb{R}^N)\to\mathbb{R}\) by
	\[
	\widehat J_\alpha(t,u)
	:=J_\alpha(t\star u)
	=\frac{e^{2st}}{2}\|u\|^2
	-\frac{\alpha}{2q}e^{2q\gamma_{q,s}st}
	\int_{\mathbb{R}^N}(I_\mu*|u|^q)|u|^q\,dx
	-\frac{1}{2p}e^{2p\gamma_{p,s}st}
	\int_{\mathbb{R}^N}(I_\mu*|u|^p)|u|^p\,dx.
	\]
	The functional \(\widehat J_\alpha\) is invariant under spatial rotations in
	the \(u\)–variable; in particular, a Palais–Smale sequence for
	\(\widehat J_\alpha|_{\mathbb{R}\times S_{c,rad}}\) corresponds, via
	\((t,u)\mapsto t\star u\), to a Palais–Smale sequence for \(J_\alpha|_{S_c}\).
	
	We introduce the minimax class
	\[
	\Gamma_1
	:=\Bigl\{
	\gamma(\tau)=(\zeta(\tau),\beta(\tau))
	\in C\bigl([0,1],\mathbb{R}\times S_{c,rad}\bigr):
	\ \gamma(0)\in\{0\}\times\mathfrak{P}_{\alpha,c}^+,\ 
	\gamma(1)\in\{0\}\times J_\alpha^{2m_1(c,\alpha)}
	\Bigr\},
	\]
	where \(J_\alpha^{2m_1(c,\alpha)}=\{u\in S_c:\ J_\alpha(u)\le 2m_1(c,\alpha)\}\)
	and \(m_1(c,\alpha)<0\) is given by Lemma \ref{Lem4.4}.
	
	We first verify that \(\Gamma_1\neq\emptyset\).
	Fix any \(u\in S_{c,rad}\). By Lemma \ref{Lem4.3} there exist
	\(t_u^1<t_u^3\) such that \(t_u^1\star u\in\mathfrak{P}_{\alpha,c}^+\) and
	\(E_u(t):=J_\alpha(t\star u)\to-\infty\) as \(t\to+\infty\).
	Hence we can choose \(t_1\gg1\) so that
	\(J_\alpha(t_1\star u)\le 2m_1(c,\alpha)\).
	Then the path
	\begin{equation}\label{eq4.14}
		\gamma_u:\ [0,1]\to\mathbb{R}\times S_{rad},\qquad
		\gamma_u(\tau)
		:=\bigl(0,\bigl((1-\tau)t_u^1+\tau t_1\bigr)\star u\bigr)
	\end{equation}
	belongs to \(\Gamma_1\). Thus \(\Gamma_1\neq\emptyset\).
	
	We define the minimax value
	\[
	\varsigma(c,\alpha)
	:=\inf_{\gamma\in\Gamma_1}\ \max_{(t,u)\in\gamma([0,1])}\widehat J_\alpha(t,u)
	\in\mathbb{R}.
	\]
	We now show that for every \(\gamma\in\Gamma_1\) there exists
	\(\tau_\gamma\in(0,1)\) such that
	\begin{equation}\label{eq4.15}
		\zeta(\tau_\gamma)=t_{\beta(\tau_\gamma)}^3,
	\end{equation}
	where \(t_{v}^3\) is the ``upper'' critical point of the fiber
	\(E_v(t)=J_\alpha(t\star v)\) given by Lemma \ref{Lem4.3}. In particular, this
	implies \(\zeta(\tau_\gamma)\star\beta(\tau_\gamma)\in\mathfrak{P}_{\alpha,c}^-\).
	
	Write \(\gamma(\tau)=(\zeta(\tau),\beta(\tau))\).
	Since \(\gamma(0)\in\{0\}\times\mathfrak{P}_{\alpha,c}^+\), we have
	\(\beta(0)\in\mathfrak{P}_{\alpha,c}^+\). By Lemma \ref{Lem4.3}, the associated
	critical levels satisfy
	\[
	t_{\beta(0)}^1=0,\qquad t_{\beta(0)}^3>0.
	\]
	On the other hand, \(\gamma(1)\in\{0\}\times J_\alpha^{2m_1(c,\alpha)}\)
	implies \(\beta(1)\in S_{rad}\) and
	\(J_\alpha(\beta(1))\le 2m_1(c,\alpha)<m_1(c,\alpha)\). Thus Lemma
	\ref{Lem4.5} yields
	\[
	t_{\beta(1)}^3<0.
	\]
	By Lemma \ref{Lem4.3}, the map \(u\mapsto t_u^3\) is \(C^1\) on \(S_c\), hence
	continuous. Since \(\beta\) and \(\zeta\) are continuous on \([0,1]\), the map
	\[
	\phi(\tau):=\zeta(\tau)-t_{\beta(\tau)}^3
	\]
	is continuous on \([0,1]\). Using the information at the endpoints,
	\[
	\phi(0)=\zeta(0)-t_{\beta(0)}^3=0-t_{\beta(0)}^3<0,
	\qquad
	\phi(1)=\zeta(1)-t_{\beta(1)}^3=0-t_{\beta(1)}^3>0.
	\]
	By the intermediate value theorem there exists \(\tau_\gamma\in(0,1)\) such that
	\(\phi(\tau_\gamma)=0\), that is,
	\[
	\zeta(\tau_\gamma)=t_{\beta(\tau_\gamma)}^3,
	\]
	which is \eqref{eq4.15}.
	
	Now set \(v_\gamma:=\zeta(\tau_\gamma)\star\beta(\tau_\gamma)\in S_c\).
	For the fiber associated with \(\beta(\tau_\gamma)\) we have
	\[
	E_{\beta(\tau_\gamma)}(t)
	=J_\alpha(t\star\beta(\tau_\gamma)),
	\]
	and \(t_{\beta(\tau_\gamma)}^3\) is the unique ``upper'' critical point:
	\(E'_{\beta(\tau_\gamma)}(t_{\beta(\tau_\gamma)}^3)=0\),
	\(E''_{\beta(\tau_\gamma)}(t_{\beta(\tau_\gamma)}^3)<0\).
	For the fiber associated with \(v_\gamma\) we note
	\[
	E_{v_\gamma}(t)
	=J_\alpha\bigl(t\star(\zeta(\tau_\gamma)\star\beta(\tau_\gamma))\bigr)
	=J_\alpha\bigl((t+\zeta(\tau_\gamma))\star\beta(\tau_\gamma)\bigr)
	=E_{\beta(\tau_\gamma)}(t+\zeta(\tau_\gamma)).
	\]
	Thus the critical points of \(E_{v_\gamma}\) are obtained from those of
	\(E_{\beta(\tau_\gamma)}\) by translation in \(t\), and in particular,
	\[
	E_{v_\gamma}'(0)
	=E_{\beta(\tau_\gamma)}'(\zeta(\tau_\gamma))
	=E_{\beta(\tau_\gamma)}'(t_{\beta(\tau_\gamma)}^3)=0,
	\]
	\[
	E_{v_\gamma}''(0)
	=E_{\beta(\tau_\gamma)}''(\zeta(\tau_\gamma))
	=E_{\beta(\tau_\gamma)}''(t_{\beta(\tau_\gamma)}^3)<0.
	\]
	Hence \(v_\gamma=\zeta(\tau_\gamma)\star\beta(\tau_\gamma)
	\in\mathfrak{P}_{\alpha,c}^-\).
	
	From this we deduce that for any \(\gamma\in\Gamma_1\),
	\begin{equation}\label{eq4.16}
		\max_{\gamma([0,1])}\widehat J_\alpha
		\;\ge\;
		\widehat J_\alpha\bigl(\gamma(\tau_\gamma)\bigr)
		=J_\alpha\bigl(\zeta(\tau_\gamma)\star\beta(\tau_\gamma)\bigr)
		\;\ge\;
		\inf_{\mathfrak{P}_{\alpha,c}^-\cap S_{c,rad}}J_\alpha.
	\end{equation}
	Thus
	\[
	\varsigma(c,\alpha)\ge
	\inf_{\mathfrak{P}_{\alpha,c}^-\cap S_{c,rad}}J_\alpha.
	\]
	Conversely, if \(u\in\mathfrak{P}_{\alpha,c}^-\cap S_{c,rad}\), then the path
	\(\gamma_u\) defined in \eqref{eq4.14} belongs to \(\Gamma_1\), and
	\[
	J_\alpha(u)
	=\widehat J_\alpha(0,u)
	=\max_{\gamma_u([0,1])}\widehat J_\alpha
	\ge\varsigma(c,\alpha).
	\]
	Hence
	\[
	\inf_{\mathfrak{P}_{\alpha,c}^-\cap S_{c,rad}}J_\alpha
	\ge\varsigma(c,\alpha),
	\]
	and combining with \eqref{eq4.16} gives
	\[
	\varsigma(c,\alpha)
	=\inf_{\mathfrak{P}_{\alpha,c}^-\cap S_{c,rad}}J_\alpha.
	\]
	
	By Corollary \ref{Cor4.1} and Lemma \ref{Lem4.5} we have
	\begin{equation}\label{eq4.17}
		\varsigma(c,\alpha)
		=\inf_{\mathfrak{P}_{\alpha,c}^- \cap S_{c,rad}}J_\alpha
		>0
		\ge
		\sup_{(\mathfrak{P}_{\alpha,c}^+\cup J_\alpha^{2m_1(c,\alpha)})
			\cap S_{c,rad}}J_\alpha
		=
		\sup_{\bigl({\bigl(\{0\}\times\mathfrak{P}_{\alpha,c}^+\bigr)\cup
			\bigl(\{0\}\times J_\alpha^{2m_1(c,\alpha)}\bigr)}{\cap\bigl(\{\R\}\times {S_{c,rad}}\bigl)\bigl)}}
		\widehat J_\alpha.
	\end{equation}
	
	Let \(\gamma_n(\tau)=(\zeta_n(\tau),\beta_n(\tau))\in\Gamma_1\) be a minimizing
	sequence for \(\varsigma(c,\alpha)\), i.e.
	\[
	\max_{\gamma_n([0,1])}\widehat J_\alpha\to \varsigma(c,\alpha).
	\]
	Using the invariance of \(\widehat J_\alpha\) under the scaling in the first
	variable, we may replace each \(\gamma_n\) by
	\[
	\widetilde\gamma_n(\tau)
	:=\bigl(0,\zeta_n(\tau)\star\beta_n(\tau)\bigr),
	\]
	which still belongs to \(\Gamma_1\) and satisfies
	\(\max_{\widetilde\gamma_n([0,1])}\widehat J_\alpha
	=\max_{\gamma_n([0,1])}\widehat J_\alpha\).
	Thus, without loss of generality, we may assume that
	\(\gamma_n(\tau)=(0,\beta_n(\tau))\) for all \(\tau\in[0,1]\).
	
	We apply Lemma \ref{Lem2.3} to the functional
	\(\varphi=\widehat J_\alpha\) on
	\[
	X=\mathbb{R}\times S_{c,r},\quad
	\mathcal F=\{\gamma([0,1]):\ \gamma\in\Gamma_1\},
	\]
	with
	\[
	B=\bigl(\{0\}\times\mathfrak{P}_{\alpha,c}^+\bigr)
	\cup \bigl(\{0\}\times J_\alpha^{2m_1(c,\alpha)}\bigr),
	\]
	and
	\[
	F
	=\{(t,u)\in\mathbb{R}\times S_{c,rad}:\ \widehat J_\alpha(t,u)\ge\varsigma(c,\alpha)\}.
	\]
	By \eqref{eq4.16} and \eqref{eq4.17} we have
	\[
	(A\cap F)\setminus B\neq\emptyset
	\quad\text{for every }A\in\mathcal F,
	\qquad
	\sup\widehat J_\alpha(B)\le\varsigma(c,\alpha)\le\inf\widehat J_\alpha(F),
	\]
	so the assumptions of Lemma \ref{Lem2.3} are satisfied.
	
	Consequently, there exists a Palais–Smale sequence
	\(\{(t_n,w_n)\}\subset\mathbb{R}\times S_{c,rad}\) for
	\(\widehat J_\alpha|_{\mathbb{R}\times S_{c,rad}}\) at level
	\(\varsigma(c,\alpha)>0\) such that
	\begin{equation}\label{eq4.18}
		\partial_t\widehat J_\alpha(t_n,w_n)\to 0,
		\qquad
		\|\partial_u\widehat J_\alpha(t_n,w_n)\|_{(T_{w_n}S_{c,r})^*}\to 0
		\quad\text{as }n\to\infty,
	\end{equation}
	and, in addition,
    \[dist\bigl((t_n,w_n),A_n\bigl)=inf_{v\in\beta([0,1]}\{|t_n-0|+\|w_n-v\|\}\to 0\]
	\begin{equation}\label{eq4.19}
|t_n|+\operatorname{dist}_{H^s}\bigl(w_n,\beta_n([0,1])\bigr)\to 0
		\quad\text{as }n\to\infty.
	\end{equation}
	In particular, \(t_n\to 0\).
	
	Using the identity
	\[
	\partial_t\widehat J_\alpha(t,u)
	=E_u'(t)=P_\alpha(t\star u),
	\]
	we deduce from \eqref{eq4.18} that
	\[
	P_\alpha(t_n\star w_n)\to 0\quad\text{as }n\to\infty.
	\]
	Moreover, for every \(\varphi\in T_{w_n}S_{c,rad},\beta(0) = w_n,\beta'{(0)} = \varphi \)
\[
\begin{aligned}
    \partial_n \hat{J}_\alpha(t_n, w_n,\varphi)& = \lim_{t \to 0} \frac{\hat{J}_\alpha(t_n, \beta{(t+1)}) - \hat{J}_\alpha(t_n, \beta{(0)})}{t}\\
&= \lim_{t \to 0} \frac{{J}_\alpha(t_n \star\beta{(t)}) - {J}_\alpha(t_n \star \beta{(0)})}{t}\\
&= \left\langle \hat{J}_\alpha'(t_n \star w_n), t_n \star  \varphi \right\rangle
\end{aligned}     
\]
	so from \eqref{eq4.18} we obtain
	\begin{equation}\label{eq4.20}
		\bigl\langle J_\alpha'(t_n\star w_n),\,t_n\star\varphi\bigr\rangle
		=o(1)\,\|\varphi\|_{H^s}
		=o(1)\,\|t_n\star\varphi\|_{H^s}
		\quad\text{as }n\to\infty.
	\end{equation}
	Since \(t_n\to 0\), the norms \(\|\varphi\|_{H^s}\) and
	\(\|t_n\star\varphi\|_{H^s}\) are equivalent uniformly in \(n\).
	
	Let
	\[
	u_n:=t_n\star w_n\in S_{c,rad}.
	\]
	Then \eqref{eq4.20} shows that the gradient of \(J_\alpha\) restricted to the
	tangent space \(T_{u_n}S_{c,r}\) tends to zero, while
	\(P_\alpha(u_n)=P_\alpha(t_n\star w_n)\to 0\). By Lemma~3.6 in
	\cite{2019BartschPDE}, the sequence \(\{u_n\}\) is a Palais–Smale sequence for
	\(J_\alpha|_{S_{c,r}}\) at level \(\varsigma(c,\alpha)>0\), with
	\[
	P_\alpha(u_n)\to 0\quad\text{as }n\to\infty.
	\]
	
	By Lemma \ref{Lem3.1} , there exists
	\(u_{c,\alpha,m}\in S_{c,rad}\) such that, up to a subsequence,
	\[
	u_n\to u_{c,\alpha,m}\quad\text{strongly in }H^s(\mathbb{R}^N),
	\]
	and \(u_{c,\alpha,m}\) is a radial weak solution of \eqref{eq1.1} for some
	Lagrange multiplier \(\lambda_{c,\alpha,m}<0\).
	
	Testing the equation with \(|(u_{c,\alpha,m})|\) yields
	\(u_{c,\alpha,m}\ge 0\). Then, by the fractional strong maximum principle, we obtain
	\[
	u_{c,\alpha,m}(x)>0\quad\text{for all }x\in\mathbb{R}^N.
	\]
	Moreover,
	\[
	J_\alpha(u_{c,\alpha,m})
	=\lim_{n\to\infty}J_\alpha(u_n)
	=\varsigma(c,\alpha)>0.
	\]
	
	Therefore \(u_{c,\alpha,m}\) is a positive mountain pass type normalized
	solution of \eqref{eq1.1} at level \(\varsigma(c,\alpha)>0\), distinct from the
	local minimizer obtained in Theorem \ref{Thm1.1}\,(1).
\end{proof}

\subsection{Convergence to the autonomous problem as $\alpha\to0$}
\begin{Lem}\label{Lem4.6}
	Let $\frac{2s-\mu}{N}+2<p<2_{\mu,s}^*$ and $\alpha=0$.
	Then $\mathfrak{P}_{0,c}^0=\varnothing$, and
	$\mathfrak{P}_{0,c}$ is a $C^1$ submanifold of codimension $2$ in
	$H^s(\mathbb{R}^N)$.
\end{Lem}

\begin{proof}
	For $\alpha=0$ the Pohozaev functional is
	\[
	P_0(u)
	= s\|u\|^2
	- s\gamma_{p,s}\int_{\mathbb{R}^N}(I_\mu*|u|^p)|u|^p\,dx,
	\qquad u\in S_c,
	\]
	and along the fiber we have
	\[
	E_u(t)=J_0(t\star u)
	=\frac{e^{2st}}{2}\|u\|^2
	-\frac{1}{2p}e^{2p\gamma_{p,s}st}
	\int_{\mathbb{R}^N}(I_\mu*|u|^p)|u|^p\,dx.
	\]
	A direct computation gives
	\[
	E_u'(t)
	= s e^{2st}\|u\|^2
	- s\gamma_{p,s}e^{2p\gamma_{p,s}st}
	\int_{\mathbb{R}^N}(I_\mu*|u|^p)|u|^p\,dx
	= \frac{1}{s}P_0(t\star u),
	\]
	and
	\[
	E_u''(t)
	= 2s^2 e^{2st}\|u\|^2
	-2s^2 p\gamma_{p,s}e^{2p\gamma_{p,s}st}
	\int_{\mathbb{R}^N}(I_\mu*|u|^p)|u|^p\,dx.
	\]
	
	If $u\in\mathfrak{P}_{0,c}$, then $P_0(u)=0$, that is
	\[
	\|u\|^2
	=\gamma_{p,s}\int_{\mathbb{R}^N}(I_\mu*|u|^p)|u|^p\,dx.
	\]
	Denoting
	\[
	A=\int_{\mathbb{R}^N}(I_\mu*|u|^p)|u|^p\,dx>0,
	\]
	this reads $\|u\|^2=\gamma_{p,s}A$, and inserting into $E_u''(0)$ gives
	\[
	E_u''(0)
	=2s^2\|u\|^2-2s^2p\gamma_{p,s}A
	=2s^2(\gamma_{p,s}A-p\gamma_{p,s}A)
	=-2s^2(p-1)\gamma_{p,s}A<0,
	\]
	since $p>1$ and $\gamma_{p,s}>0$. Hence there is no $u\in S_c$ with
	$P_0(u)=0$ and $E_u''(0)=0$, so $\mathfrak{P}_{0,c}^0=\varnothing$.
	
	The fact that $\mathfrak{P}_{0,c}$ is a $C^1$ submanifold of codimension
	$2$ follows as in Lemma~\ref{Lem4.1}, by considering the map
	\[
	C(u):=\int_{\mathbb{R}^N}|u|^2\,dx-c^2,\qquad
	\mathfrak{P}_{0,c}=\{u\in H^s(\mathbb{R}^N):C(u)=0,\ P_0(u)=0\},
	\]
	and observing that for $u\in\mathfrak{P}_{0,c}$ the functionals $C'(u)$ and
	$P_0'(u)$ are linearly independent in $H^s(\mathbb{R}^N)^*$. The implicit
	function theorem then yields the claim.
\end{proof}

\begin{Lem}\label{Lem4.7}
	Let $\frac{2s-\mu}{N}+2<p<2_{\mu,s}^*$ and $\alpha=0$.
	For any $u\in S_c$, the function
	\[
	E_u(t)=J_0(t\star u)
	\]
	has a unique critical point $t_u^*\in\mathbb{R}$, which is a strict global
	maximum of positive level.
	
	Moreover:
	\begin{enumerate}
		\item $E_u$ is strictly decreasing and concave on $(t_u^*,+\infty)$.
		\item One has $\mathfrak{P}_{0,c}=\mathfrak{P}_{0,c}^-$. In particular,
		if $P_0(u)<0$ then $t_u^*<0$.
		\item The map $u\in S_c\mapsto t_u^*\in\mathbb{R}$ is of class $C^1$.
	\end{enumerate}
\end{Lem}

\begin{proof}
	For $\alpha=0$ we have
	\[
	P_0(u)
	= s\|u\|^2
	-s\gamma_{p,s}\int_{\mathbb{R}^N}(I_\mu*|u|^p)|u|^p\,dx,
	\]
	\[
	E_u(t)=J_0(t\star u)
	=\frac{e^{2st}}{2}\|u\|^2
	-\frac{1}{2p}e^{2p\gamma_{p,s}st}
	\int_{\mathbb{R}^N}(I_\mu*|u|^p)|u|^p\,dx.
	\]
	Let $A:=\displaystyle\int_{\mathbb{R}^N}(I_\mu*|u|^p)|u|^p\,dx>0$. Then
	\[
	E_u'(t)
	= s e^{2st}\|u\|^2
	-s\gamma_{p,s}e^{2p\gamma_{p,s}st}A
	=\frac{1}{s}P_0(t\star u),
	\]
	\[
	E_u''(t)
	= 2s^2 e^{2st}\|u\|^2
	-2s^2 p\gamma_{p,s}e^{2p\gamma_{p,s}st}A.
	\]
	
	Solving $E_u'(t)=0$ gives
	\[
	e^{2st}\|u\|^2
	=\gamma_{p,s}e^{2p\gamma_{p,s}st}A
	\quad\Longleftrightarrow\quad
	e^{2st(p\gamma_{p,s}-1)}
	=\frac{\|u\|^2}{\gamma_{p,s}A}.
	\]
	Since $p\gamma_{p,s}>1$, this equation has a unique solution
	$t=t_u^*\in\mathbb{R}$, so $E_u$ has exactly one critical point.
	
	At $t=t_u^*$ the above relation implies
	$e^{2st_u^*(p\gamma_{p,s}-1)}=\|u\|^2/(\gamma_{p,s}A)$, and hence
	\[
	\begin{aligned}
		E_u''(t_u^*)
		&= 2s^2 e^{2st_u^*}\|u\|^2
		-2s^2 p\gamma_{p,s}e^{2p\gamma_{p,s}st_u^*}A \\
		&= 2s^2 e^{2st_u^*}\|u\|^2
		-2s^2 p\gamma_{p,s}e^{2st_u^*}
		e^{2st_u^*(p\gamma_{p,s}-1)}A \\
		&= 2s^2 e^{2st_u^*}\|u\|^2
		-2s^2 p\gamma_{p,s}e^{2st_u^*}\frac{\|u\|^2}{\gamma_{p,s}} \\
		&= 2s^2 e^{2st_u^*}\|u\|^2(1-p\gamma_{p,s})<0,
	\end{aligned}
	\]
	so $t_u^*$ is a strict local maximum.
	
	As $t\to-\infty$ we have
	\[
	E_u(t)
	=\frac{e^{2st}}{2}\|u\|^2
	-\frac{1}{2p}e^{2p\gamma_{p,s}st}A
	\to 0^{+},
	\]
	because $2p\gamma_{p,s}>2$, and therefore the second term decays faster than
	the first one. On the other hand,
	$E_u(t)\to-\infty$ as $t\to+\infty$, since the negative term with exponent
	$2p\gamma_{p,s}>2$ dominates. Together with the uniqueness of the critical
	point, this shows that $t_u^*$ is the unique global maximizer of $E_u$.
	In particular, $E_u(t_u^*)>0$, because $E_u(t)>0$ for $t$ sufficiently negative.
	
	Using the expressions of $E_u'$ and $E_u''$ we write
	\[
	E_u''(t)
	= 2s^2 e^{2st}\|u\|^2
	-2s^2 p\gamma_{p,s}e^{2p\gamma_{p,s}st}A
	= 2s\,E_u'(t)
	-2s^2\gamma_{p,s}(p-1)e^{2p\gamma_{p,s}st}A.
	\]
	Since $E_u'(t_u^*)=0$ and $t_u^*$ is the unique zero of $E_u'$, we have
	$E_u'(t)>0$ for $t<t_u^*$ and $E_u'(t)<0$ for $t>t_u^*$. For every
	$t>t_u^*$ the second term above is strictly negative and the first term is
	also negative, so $E_u''(t)<0$ for all $t>t_u^*$. Hence $E_u$ is strictly
	concave and strictly decreasing on $(t_u^*,+\infty)$, which proves (1).
	
	Now let $u\in\mathfrak{P}_{0,c}$, so $P_0(u)=0$.
	Then $E_u'(0)=\frac{1}{s}P_0(u)=0$, and from the computation in the proof
	of Lemma~\ref{Lem4.7} we know that $E_u''(0)<0$, so
	$u\in\mathfrak{P}_{0,c}^-$. Thus
	$\mathfrak{P}_{0,c}^+=\varnothing$, $\mathfrak{P}_{0,c}^0=\varnothing$, and
	$\mathfrak{P}_{0,c}=\mathfrak{P}_{0,c}^-$.
	
	Moreover, for general $u\in S_c$ we have
	\[
	P_0(u)=E_u'(0).
	\]
	If $P_0(u)<0$, then $E_u'(0)<0$. Since $E_u'(-\infty)=0^{+}$,
	$E_u'(+\infty)=-\infty$ and $E_u'$ has exactly one zero $t_u^*$, we must
	have $t_u^*<0$ (otherwise, for $t_u^*>0$ we would have $E_u'(0)>0$).
	This proves (2).
	
	Finally, the map
	\[
	F:S_c\times\mathbb{R}\to\mathbb{R},\qquad
	F(u,t)=E_u'(t),
	\]
	is $C^1$, and for each $u\in S_c$ the equation $F(u,t)=0$ has a unique
	solution $t=t_u^*$ with $F_t(u,t_u^*)=E_u''(t_u^*)\neq0$. The implicit
	function theorem yields a $C^1$ map $u\mapsto t_u^*$ on $S_c$, which gives (3).
\end{proof}
\begin{Lem}\label{Lem4.8}
	Assume that
	\[
	2_{\mu,*}<q<\frac{2s-\mu}{N}+2<p<2_{\mu,s}^*
	\]
	and \(0<\alpha<\min\{\alpha_1,\alpha_2\}\).
	Then
	\begin{equation}\label{eq4.21}
		\inf_{u\in\mathfrak{P}_{\alpha,c}^-\cap S_{rad}} J_\alpha(u)
		=\inf_{u\in S_{c,rad}}\ \max_{t\in\mathbb{R}} J_\alpha(t\star u).
	\end{equation}
	For \(\alpha=0\) one has
	\begin{equation}\label{eq4.22}
		\inf_{u\in\mathfrak{P}_{0,c}^-\cap S_{c,rad}} J_0(u)
		=\inf_{u\in S_{c,rad}}\ \max_{t\in\mathbb{R}} J_0(t\star u).
	\end{equation}
	Moreover, if \(0<\alpha_3<\alpha_4<\min\{\alpha_1,\alpha_2\}\), then
	\[
	\varsigma(c,\alpha_4)\le\varsigma(c,\alpha_3),
	\]
	where \(\varsigma(c,\alpha)\) is as in \eqref{eq4.17}; in addition,
	\[
	\varsigma(c,\alpha)\le m_r(c,0)
	\qquad\text{for all }0\le\alpha<\min\{\alpha_1,\alpha_2\},
	\]
	with \(m_r(c,0)=\inf\limits_{u\in S_{c,rad}}\max\limits_{t\in\mathbb{R}}J_0(t\star u)\).
\end{Lem}

\begin{proof}
	Fix \(\alpha\in(0,\min\{\alpha_1,\alpha_2\})\). For every \(u\in S_{c,rad}\),
	Lemma \ref{Lem4.3} yields a unique \(t_u\in\mathbb{R}\) such that
	\(t_u\star u\in\mathfrak{P}_{\alpha,c}^-\cap S_{c,rad}\), and the map
	\(t\mapsto J_\alpha(t\star u)\) attains its global maximum at \(t=t_u\). In
	particular,
	\[
	\max_{t\in\mathbb{R}} J_\alpha(t\star u)=J_\alpha(t_u\star u).
	\]
	
	If \(u\in\mathfrak{P}_{\alpha,c}^-\cap S_{c,rad}\), the uniqueness of \(t_u\)
	implies \(t_u=0\), hence
	\[
	J_\alpha(u)
	=J_\alpha(t_u\star u)
	=\max_{t\in\mathbb{R}} J_\alpha(t\star u)
	\ge\inf_{v\in S_{c,rad}}\max_{t\in\mathbb{R}} J_\alpha(t\star v).
	\]
	Taking the infimum over \(u\in\mathfrak{P}_{\alpha,c}^-\cap S_{c,rad}\) gives
	\begin{equation}\label{eq4.23}
		\inf_{u\in\mathfrak{P}_{\alpha,c}^-\cap S_{c,rad}} J_\alpha(u)
		\ge\inf_{v\in S_{c,rad}}\max_{t\in\mathbb{R}} J_\alpha(t\star v).
	\end{equation}
	
	Conversely, for arbitrary \(u\in S_{c,rad}\) we have
	\[
	\max_{t\in\mathbb{R}}J_\alpha(t\star u)
	=J_\alpha(t_u\star u)
	\ge\inf_{v\in\mathfrak{P}_{\alpha,c}^-\cap S_{c,rad}}J_\alpha(v),
	\]
	and taking the infimum over \(u\in S_{c,rad}\) gives
	\begin{equation}\label{eq4.24}
		\inf_{u\in S_{c,rad}}\max_{t\in\mathbb{R}}J_\alpha(t\star u)
		\ge\inf_{v\in\mathfrak{P}_{\alpha,c}^-\cap S_{c,rad}}J_\alpha(v).
	\end{equation}
	Combining \eqref{eq4.23} and \eqref{eq4.24} yields
	\eqref{eq4.21}.
	
	The same argument applies to \(J_0\) (i.e. to the case \(\alpha=0\)), since
	the Pohozaev manifold and the scaling properties are preserved when the
	lower order nonlocal term is removed. This gives
	\eqref{eq4.22}.
	
	By Lemma \ref{Lem4.5} and \eqref{eq4.17}, for
	\(0<\alpha<\min\{\alpha_1,\alpha_2\}\) we have
	\[
	\varsigma(c,\alpha)
	=\inf_{u\in S_{c,rad}}\max_{t\in\mathbb{R}}J_\alpha(t\star u).
	\]
	Let \(0<\alpha_3<\alpha_4<\min\{\alpha_1,\alpha_2\}\). For every
	\(u\in S_{c,rad}\) and \(t\in\mathbb{R}\),
	\[
	J_{\alpha_4}(t\star u)
	=J_{\alpha_3}(t\star u)
	-\frac{\alpha_4-\alpha_3}{2q}
	\int_{\mathbb{R}^N}(I_\mu*|t\star u|^q)|t\star u|^q\,dx
	\le J_{\alpha_3}(t\star u),
	\]
	so
	\[
	\max_{t\in\mathbb{R}}J_{\alpha_4}(t\star u)
	\le \max_{t\in\mathbb{R}}J_{\alpha_3}(t\star u).
	\]
	Taking the infimum over \(u\in S_{c,rad}\) gives
	\[
	\varsigma(c,\alpha_4)
	=\inf_{u\in S_{c,rad}}\max_{t\in\mathbb{R}}J_{\alpha_4}(t\star u)
	\le\inf_{u\in S_{c,rad}}\max_{t\in\mathbb{R}}J_{\alpha_3}(t\star u)
	=\varsigma(c,\alpha_3).
	\]
	
	Finally, for every \(\alpha\ge0\), every \(u\in S_{c,rad}\), and every
	\(t\in\mathbb{R}\),
	\[
	J_\alpha(t\star u)
	=J_0(t\star u)
	-\frac{\alpha}{2q}\int_{\mathbb{R}^N}(I_\mu*|t\star u|^q)|t\star u|^q\,dx
	\le J_0(t\star u),
	\]
	so
	\[
	\max_{t\in\mathbb{R}}J_\alpha(t\star u)
	\le \max_{t\in\mathbb{R}}J_0(t\star u).
	\]
	Taking the infimum over \(u\in S_{c,rad}\) yields
	\[
	\varsigma(c,\alpha)
	=\inf_{u\in S_{c,rad}}\max_{t\in\mathbb{R}}J_\alpha(t\star u)
	\le\inf_{u\in S_{c,rad}}\max_{t\in\mathbb{R}}J_0(t\star u)
	=m(c,0),
	\]
	which holds for all \(0\le\alpha<\min\{\alpha_1,\alpha_2\}\).
\end{proof}

\begin{Lem}\label{Lem4.9}
	Let $\frac{2s-\mu}{N}+2<p<2_{\mu,s}^*$ and $\alpha=0$.
	Define
	\[
	m_2(c,0)=\inf_{u\in\mathfrak{P}_{0,c}}J_0(u).
	\]
	Then $m_2(c,0)>0$. Moreover, there exists $r>0$ sufficiently small such that
	\[
	0<\sup_{u\in\overline{D_r}}J_0(u)<m(c,0),
	\]
	where
	\[
	D_r=\{u\in S_c:\ \|u\|<r\}.
	\]
	In particular, for all $u\in\overline{D_r}$ one has $J_0(u)>0$ and
	$P_0(u)>0$.
\end{Lem}

\begin{proof}
	Let $u\in\mathfrak{P}_{0,c}$ be arbitrary. Since $P_0(u)=0$, we have
	\[
	s\|u\|^2
	= s\gamma_{p,s}\int_{\mathbb{R}^N}(I_\mu*|u|^p)|u|^p\,dx.
	\]
	By Lemma~\ref{Lem2.2} and Remark~\ref{Rek2.3} there exists $C_p>0$ such that
	\[
	\int_{\mathbb{R}^N}(I_\mu*|u|^p)|u|^p\,dx
	\le C_p\|u\|^{2p\gamma_{p,s}}\|u\|_2^{2p(1-\gamma_{p,s})}.
	\]
	Using $\|u\|_2^2=c^2$, we obtain
	\[
	\|u\|^2
	\le \gamma_{p,s}C_p c^{2p(1-\gamma_{p,s})}\|u\|^{2p\gamma_{p,s}}.
	\]
	Since $p\gamma_{p,s}>1$, this inequality yields a uniform lower bound
	\[
	\|u\|\ge C_0>0
	\]
	for all $u\in\mathfrak{P}_{0,c}$, where $C_0>0$ depends only on $c,p,s,\mu$.
	
	On $\mathfrak{P}_{0,c}$ we have
	\[
	\gamma_{p,s}\int_{\mathbb{R}^N}(I_\mu*|u|^p)|u|^p\,dx
	=\|u\|^2,
	\]
	so
	\[
	\begin{aligned}
		J_0(u)
		&= \frac{1}{2}\|u\|^2
		-\frac{1}{2p}\int_{\mathbb{R}^N}(I_\mu*|u|^p)|u|^p\,dx \\
		&= \frac{1}{2}\|u\|^2
		-\frac{1}{2p\gamma_{p,s}}\|u\|^2
		=\Big(\frac{1}{2}-\frac{1}{2p\gamma_{p,s}}\Big)\|u\|^2.
	\end{aligned}
	\]
	Since $p\gamma_{p,s}>1$ and $\|u\|\ge C_0$, there exists $C_1>0$ such that
	\[
	J_0(u)\ge C_1>0
	\quad\text{for all }u\in\mathfrak{P}_{0,c}.
	\]
	Therefore
	\[
	m_2(c,0)=\inf_{u\in\mathfrak{P}_{0,c}}J_0(u)\ge C_1>0.
	\]
	
	Now let $u\in S_c$ be arbitrary. Using again the nonlocal inequality, we have
	\[
	\begin{aligned}
		J_0(u)
		&= \frac{1}{2}\|u\|^2
		-\frac{1}{2p}\int_{\mathbb{R}^N}(I_\mu*|u|^p)|u|^p\,dx \\
		&\ge \frac{1}{2}\|u\|^2
		-\frac{1}{2p}C_p\|u\|^{2p\gamma_{p,s}}\|u\|_2^{2p(1-\gamma_{p,s})} \\
		&= \frac{1}{2}\|u\|^2
		-\frac{1}{2p}C_p c^{2p(1-\gamma_{p,s})}\|u\|^{2p\gamma_{p,s}},
	\end{aligned}
	\]
	and
	\[
	\begin{aligned}
		P_0(u)
		&= s\|u\|^2
		-s\gamma_{p,s}\int_{\mathbb{R}^N}(I_\mu*|u|^p)|u|^p\,dx \\
		&\ge s\|u\|^2
		-s\gamma_{p,s}C_p c^{2p(1-\gamma_{p,s})}\|u\|^{2p\gamma_{p,s}}.
	\end{aligned}
	\]
	Since $2p\gamma_{p,s}>2$, there exists $r_0>0$ such that for all
	$t\in(0,r_0]$,
	\[
	\frac{1}{2}t^2
	-\frac{1}{2p}C_p c^{2p(1-\gamma_{p,s})}t^{2p\gamma_{p,s}}>0,
	\qquad
	s t^2
	-s\gamma_{p,s}C_p c^{2p(1-\gamma_{p,s})}t^{2p\gamma_{p,s}}>0.
	\]
	Therefore, if $u\in\overline{D_r}$ with $0<r\le r_0$ (so $\|u\|\le r$), then
	\[
	J_0(u)>0,\qquad P_0(u)>0.
	\]
	In particular,
	\[
	\sup_{u\in\overline{D_r}}J_0(u)>0.
	\]
	
	Moreover, for all $u\in S_c$ we have
	\[
	J_0(u)
	=\frac{1}{2}\|u\|^2
	-\frac{1}{2p}\int_{\mathbb{R}^N}(I_\mu*|u|^p)|u|^p\,dx
	\le \frac{1}{2}\|u\|^2,
	\]
	because the nonlocal term is nonnegative. Hence, for $u\in\overline{D_r}$,
	\[
	J_0(u)\le\frac{1}{2}\|u\|^2\le\frac{1}{2}r^2,
	\]
	so
	\[
	\sup_{u\in\overline{D_r}}J_0(u)\le\frac{1}{2}r^2.
	\]
	Since $m_3(c,0)>0$ is fixed, we can choose $r>0$ small enough such that
	$r\le r_0$ and $\frac{1}{2}r^2<m_2(c,0)$. For this choice of $r$ we obtain
	\[
	0<\sup_{u\in\overline{D_r}}J_0(u)<m_2(c,0),
	\]
	and $J_0(u)>0$, $P_0(u)>0$ for all $u\in\overline{D_r}$. This concludes the proof.
\end{proof}

\begin{Lem}\label{Lem4.10}
	Let $\frac{2s-\mu}{N}+2<p<2_{\mu,s}^*$ and $\alpha=0$.
	Then there exists a positive radial critical point $u_0\in S_{c,rad}$ of
	$J_0|_{S_c}$ such that
	\[
	0<m_r(c,0)
	=\inf_{\mathfrak{P}_{0,c}\cap S_{c,rad}}J_0(u)
	= m(c,0)
	= J_0(u_0),
	\]
	where
	\[
	m_2(c,0)=\inf_{\mathfrak{P}_{0,c}}J_0(u).
	\]
\end{Lem}

\begin{proof}
	By Lemma~\ref{Lem4.8} one has $\mathfrak{P}_{0,c}=\mathfrak{P}_{0,c}^-$, and
	for every $u\in S_c$ there exists a unique $t_u\in\mathbb{R}$ such that
	$t_u\star u\in\mathfrak{P}_{0,c}$ and
	\[
	J_0(t_u\star u)=\max_{t\in\mathbb{R}}J_0(t\star u).
	\]
	In particular, if $v\in\mathfrak{P}_{0,c}$ then $t_v=0$ and
	\[
	J_0(v)=\max_{t\in\mathbb{R}}J_0(t\star v)\ge
	\inf_{w\in S_c}\max_{t\in\mathbb{R}}J_0(t\star w),
	\]
	so
	\[
	m_2(c,0)=\inf_{v\in\mathfrak{P}_{0,c}}J_0(v)
	\ge \inf_{w\in S_c}\max_{t\in\mathbb{R}}J_0(t\star w).
	\]
	Conversely, for any $u\in S_c$,
	\[
	\max_{t\in\mathbb{R}}J_0(t\star u)
	=J_0(t_u\star u)\ge \inf_{v\in\mathfrak{P}_{0,c}}J_0(v)
	= m_2(c,0),
	\]
	so taking the infimum over $u\in S_c$ gives
	\[
	\inf_{u\in S_c}\max_{t\in\mathbb{R}}J_0(t\star u)\ge m_2(c,0).
	\]
	Hence
	\begin{equation}\label{eq4.25}
		m_2(c,0)
		=\inf_{u\in S_c}\max_{t\in\mathbb{R}}J_0(t\star u).
	\end{equation}
	
	Let $u\in S_c$ and let $u^\ast$ denote its symmetric decreasing
	rearrangement. By the fractional P\'olya--Szeg\H{o} inequality \cite{Carbotti} and the Riesz
	rearrangement inequality one has
	\[
	\|u^\ast\|\le\|u\|,
	\qquad
	\int_{\mathbb{R}^N}(I_\mu*|u|^p)|u|^p\,dx\le \int_{\mathbb{R}^N}(I_\mu*|u^\ast|^p)|u^\ast|^p\,dx,
	\]
	and $\|u^\ast\|_2=\|u\|_2=c$. Thus $u^\ast\in S_{c,rad}$ and, for every
	$t\in\mathbb{R}$,
	\[
	J_0(t\star u^\ast)\le J_0(t\star u).
	\]
	It follows that
	\[
	\max_{t\in\mathbb{R}}J_0(t\star u^\ast)
	\le\max_{t\in\mathbb{R}}J_0(t\star u),
	\]
	and taking the infimum over $u\in S_c$ yields
	\[
	\inf_{u\in S_{c,rad}}\max_{t\in\mathbb{R}}J_0(t\star u)
	=\inf_{u\in S_c}\max_{t\in\mathbb{R}}J_0(t\star u).
	\]
	Together with \eqref{eq4.25} this gives
	\begin{equation}\label{eq4.26}
		m(c,0)
		=\inf_{u\in S_{c,rad}}\max_{t\in\mathbb{R}}J_0(t\star u).
	\end{equation}
	
	For $u\in S_{c,rad}$, Lemma~\ref{Lem4.8} implies that there exists a unique
	$t_u\in\mathbb{R}$ such that $t_u\star u\in\mathfrak{P}_{0,c}\cap S_{c,rad}$
	and
	\[
	\max_{t\in\mathbb{R}}J_0(t\star u)=J_0(t_u\star u).
	\]
	Hence
	\[
	\inf_{u\in S_{c,rad}}\max_{t\in\mathbb{R}}J_0(t\star u)
	=\inf_{u\in S_{c,rad}}J_0(t_u\star u)
	\ge\inf_{v\in\mathfrak{P}_{0,c}\cap S_{c,rad}}J_0(v)
	=m_r(c,0).
	\]
	On the other hand, if $v\in\mathfrak{P}_{0,c}\cap S_{c,rad}$ then $t_v=0$ and
	\[
	J_0(v)=\max_{t\in\mathbb{R}}J_0(t\star v),
	\]
	so
	\[
	m_r(c,0)
	=\inf_{v\in\mathfrak{P}_{0,c}\cap S_{c,rad}}J_0(v)
	\ge\inf_{u\in S_{c,rad}}\max_{t\in\mathbb{R}}J_0(t\star u).
	\]
	Combining these inequalities with \eqref{eq4.26} we obtain
	\[
	m_r(c,0)
	=\inf_{u\in S_{c,rad}}\max_{t\in\mathbb{R}}J_0(t\star u)
	=m_2(c,0).
	\]
	
	By Lemma~\ref{Lem4.9} we have $m_2(c,0)>0$ and there exists
	$r>0$ such that $J_0>0$ on $\overline{D_r}\subset S_{c,rad}$, while, for
	every $u\in S_{c,rad}$, the map $t\mapsto J_0(t\star u)$ tends to $-\infty$
	as $t\to+\infty$ (see Lemma~\ref{Lem4.8}). Therefore $J_0|_{S_{c,rad}}$ has a
	mountain pass geometry and $m_r(c,0)>0$ is its mountain pass level.
	
	By the constrained mountain pass theorem on $S_{c,rad}$ (see, for instance,
	\cite{2020soavejde}) there exists a sequence $(u_n)\subset S_{c,rad}$ such that
	\[
	J_0(u_n)\to m_r(c,0),
	\qquad
	\|(J_0|_{S_c})'(u_n)\|\to 0,
	\]
	and, in addition, $P_0(u_n)\to 0$. In particular $(u_n)$ is bounded in
	$H^s(\mathbb{R}^N)$. By Lemma~\ref{Lem3.1} (applied with $\alpha=0$), up to
	a subsequence,
	\[
	u_n\to u_0\quad\text{strongly in }H^s(\mathbb{R}^N),
	\]
	for some $u_0\in S_{c,rad}$, and $u_0$ is a critical point of $J_0|_{S_c}$
	with
	\[
	J_0(u_0)=m_r(c,0)=m_2(c,0).
	\]
	
	Testing the equation with $|u_0|$ gives $u_0\ge0$. Since $u_0$ is a
	nontrivial solution of the autonomous fractional Choquard equation, the
	strong maximum principle implies $u_0>0$ in $\mathbb{R}^N$. Thus $u_0$ is a
	positive radial critical point of $J_0|_{S_c}$ with energy $m_2(c,0)$,
	which concludes the proof.
\end{proof}

\medskip
\noindent\textbf{Proof of Theorem \ref{Thm1.1} (4).}
Let $\dot\alpha>0$ be sufficiently small and consider the family of
mountain–pass solutions $\{u_{c,\alpha,m}:0<\alpha<\dot\alpha\}\subset S_{c,rad}$
given by Theorem~\ref{Thm1.1} (2). By construction one has
\[
J_\alpha(u_{c,\alpha,m})=\varsigma(c,\alpha),
\qquad
P_\alpha(u_{c,\alpha,m})=0,
\]
and Lemma~\ref{Lem4.6} yields
\[
0<\varsigma(c,\dot\alpha)
\le\varsigma(c,\alpha)\le m(c,0)
\quad\text{for all }0<\alpha<\dot\alpha,
\]
where $m(c,0)$ is the mountain pass level of the autonomous problem
$\alpha=0$, see Lemma~\ref{Lem4.10}.

Using $P_\alpha(u_{c,\alpha,m})=0$ we can express the $p$–term as
\[
\int_{\mathbb{R}^N}(I_\mu*|u_{c,\alpha,m}|^p)|u_{c,\alpha,m}|^p\,dx
=\frac{1}{\gamma_{p,s}}\Big(
\|u_{c,\alpha,m}\|^2
-\alpha\gamma_{q,s}
\int_{\mathbb{R}^N}(I_\mu*|u_{c,\alpha,m}|^q)|u_{c,\alpha,m}|^q\,dx
\Big).
\]
Hence
\[
\begin{aligned}
	J_\alpha(u_{c,\alpha,m})
	&=\frac{1}{2}\|u_{c,\alpha,m}\|^2
	-\frac{\alpha}{2q}\int_{\mathbb{R}^N}(I_\mu*|u_{c,\alpha,m}|^q)
	|u_{c,\alpha,m}|^q\,dx
	-\frac{1}{2p}
	\int_{\mathbb{R}^N}(I_\mu*|u_{c,\alpha,m}|^p)
	|u_{c,\alpha,m}|^p\,dx\\[0.2em]
	&= \Big(\frac{1}{2}-\frac{1}{2p\gamma_{p,s}}\Big)\,
	\|u_{c,\alpha,m}\|^2
	-\frac{\alpha}{2q}
	\Big(1-\frac{q\gamma_{q,s}}{p\gamma_{p,s}}\Big)
	\int_{\mathbb{R}^N}(I_\mu*|u_{c,\alpha,m}|^q)
	|u_{c,\alpha,m}|^q\,dx.
\end{aligned}
\]
By Lemma~\ref{Lem2.2} and Remark~\ref{Rek2.3},
\[
\int_{\mathbb{R}^N}(I_\mu*|u|^q)|u|^q\,dx
\le C_q\|u\|^{2q\gamma_{q,s}}c^{2q(1-\gamma_{q,s})}
\]
for all $u\in S_c$. Therefore, for all $0<\alpha<\dot\alpha$,
\begin{equation}\label{eq4.27}
	\varsigma(c,\alpha)=J_\alpha(u_{c,\alpha,m})
	\ge A\,\|u_{c,\alpha,m}\|^2
	-C\,\alpha\|u_{c,\alpha,m}\|^{2q\gamma_{q,s}},
\end{equation}
with
\[
A=\frac{1}{2}-\frac{1}{2p\gamma_{p,s}}>0,
\quad
C>0\ \text{independent of }\alpha.
\]

Since $q\gamma_{q,s}<1$, we have $2q\gamma_{q,s}<2$, so the right–hand side
of \eqref{eq4.27} tends to $+\infty$ as $\|u_{c,\alpha,m}\|
\to+\infty$, uniformly for $0<\alpha\le\dot\alpha$. Combined with
$0<\varsigma(c,\alpha)\le m(c,0)$, this shows that
$\{u_{c,\alpha,m}:0<\alpha<\dot\alpha\}$ is bounded in $H^s(\mathbb{R}^N)$,
uniformly for $0<\alpha<\dot\alpha$.

Since $u_{c,\alpha,m}\in S_{c,rad}$ for all $\alpha$, we can fix a sequence
$\alpha_n\to0$ and, up to a subsequence, assume that
\[
u_{c,\alpha_n,m}\rightharpoonup u_0
\quad\text{in }H^s(\mathbb{R}^N),
\qquad
u_{c,\alpha_n,m}\to u_0
\quad\text{in }L^r(\mathbb{R}^N)
\ \ \forall\,2<r<2_s^*,
\]
and $u_{c,\alpha_n,m}(x)\to u_0(x)\ge0$ a.e.\ in $\mathbb{R}^N$. For brevity
we write $u_n:=u_{c,\alpha_n,m}$.

Each $u_n\in S_{c,rad}$ solves
\begin{equation}\label{eq4.28}
	\begin{aligned}
		&\int_{\mathbb{R}^N}(-\Delta)^{\frac{s}{2}}u_n
		(-\Delta)^{\frac{s}{2}}v\,dx
		-\lambda_n\int_{\mathbb{R}^N}u_n v\,dx\\
        &=\alpha_n\int_{\mathbb{R}^N}(I_\mu*|u_n|^q)
		|u_n|^{q-2}u_n v\,dx +\int_{\mathbb{R}^N}(I_\mu*|u_n|^p)
		|u_n|^{p-2}u_n v\,dx
	\end{aligned}
\end{equation}
for all $v\in H^s(\mathbb{R}^N)$, where $\lambda_n=\lambda_{c,\alpha_n,m}<0$
is the Lagrange multiplier corresponding to the mass constraint.

Testing \eqref{eq4.28} with $v=u_n$ and using $P_{\alpha_n}(u_n)=0$ gives
(as in Lemma~\ref{Lem3.1})
\[
\lambda_n c^2
= \alpha_n(\gamma_{q,s}-1)
\int_{\mathbb{R}^N}(I_\mu*|u_n|^q)|u_n|^q\,dx
+(\gamma_{p,s}-1)
\int_{\mathbb{R}^N}(I_\mu*|u_n|^p)|u_n|^p\,dx.
\]
By Lemma~\ref{Lem2.2} and the boundedness of $(u_n)$ in $H^s(\mathbb{R}^N)$,
both nonlocal integrals are uniformly bounded. Since
$\gamma_{q,s},\gamma_{p,s}<1$, it follows that $(\lambda_n)$ is bounded
in $\mathbb{R}$, and, up to a subsequence,
\[
\lambda_n\to\lambda_0\le 0
\quad\text{as }n\to\infty.
\]

Passing to the limit in \eqref{eq4.28} we obtain the autonomous equation.
Indeed, the linear terms converge by weak convergence in $H^s$ and $L^2$, the
$q$–term vanishes because $\alpha_n\to0$ and the integrals are uniformly
bounded, and the $p$–term converges by Proposition~\ref{Pro2.1} and the strong
convergence of $(u_n)$ in $L^r(\mathbb{R}^N)$ for $2<r<2_s^*$.
Therefore $u_0$ satisfies
\begin{equation}\label{eq4.29}
	(-\Delta)^s u_0
	=\lambda_0 u_0
	+ (I_\mu*|u_0|^p)|u_0|^{p-2}u_0
	\quad\text{in }\mathbb{R}^N,
\end{equation}
in the weak sense.

We claim that $u_0\not\equiv0$. Suppose, by contradiction, that $u_0=0$.
Then $u_n\to0$ in $L^r(\mathbb{R}^N)$ for all $2<r<2_s^*$, and by
Proposition~\ref{Pro2.1} we have
\[
\int_{\mathbb{R}^N}(I_\mu*|u_n|^q)|u_n|^q\,dx\to 0,
\qquad
\int_{\mathbb{R}^N}(I_\mu*|u_n|^p)|u_n|^p\,dx\to 0.
\]
Using $P_{\alpha_n}(u_n)=0$,
\[
\|u_n\|^2
= \alpha_n\gamma_{q,s}\int_{\mathbb{R}^N}(I_\mu*|u_n|^q)|u_n|^q\,dx
+ \gamma_{p,s}\int_{\mathbb{R}^N}(I_\mu*|u_n|^p)|u_n|^p\,dx,
\]
we deduce $\|u_n\|\to0$. Then
\[
J_{\alpha_n}(u_n)
= \frac{1}{2}\|u_n\|^2
-\frac{\alpha_n}{2q}\int_{\mathbb{R}^N}(I_\mu*|u_n|^q)|u_n|^q\,dx
-\frac{1}{2p}\int_{\mathbb{R}^N}(I_\mu*|u_n|^p)|u_n|^p\,dx\to0.
\]
But $J_{\alpha_n}(u_n)=\varsigma(c,\alpha_n)$ and Lemma~\ref{Lem4.6} yields
\[
0<\varsigma(c,\dot\alpha)
\le\varsigma(c,\alpha_n)
\le m(c,0)
\]
for all $n$. This contradicts $J_{\alpha_n}(u_n)\to0$. Hence $u_0$ is
nontrivial.

Since $u_0$ is a nontrivial solution of \eqref{eq4.29}, the Pohozaev
identity for the autonomous problem gives $P_0(u_0)=0$, that is
\[
s\|u_0\|^2
-s\gamma_{p,s}\int_{\mathbb{R}^N}(I_\mu*|u_0|^p)|u_0|^p\,dx=0.
\]
Testing \eqref{eq4.29} with $u_0$ we also obtain
\[
\|u_0\|^2
-\lambda_0 c^2
-\int_{\mathbb{R}^N}(I_\mu*|u_0|^p)|u_0|^p\,dx=0.
\]
Eliminating $\|u_0\|^2$ from these two identities yields
\[
\lambda_0 c^2
=(\gamma_{p,s}-1)
\int_{\mathbb{R}^N}(I_\mu*|u_0|^p)|u_0|^p\,dx.
\]
Since $u_0\not\equiv0$ and the Riesz potential $I_\mu$ is strictly positive,
\[
\int_{\mathbb{R}^N}(I_\mu*|u_0|^p)|u_0|^p\,dx>0,
\]
so $\lambda_0<0$. In particular $u_0$ is a positive solution by the strong
maximum principle.

To upgrade weak convergence to strong convergence in $H^s(\mathbb{R}^N)$, we
subtract \eqref{eq4.29} from \eqref{eq4.28} and test the resulting
identity with $v=u_n-u_0$. Using again Proposition~\ref{Pro2.1} and the
Brezis–Lieb lemma to control the nonlocal terms, and the convergence
$\lambda_n\to\lambda_0$, we obtain
\[
\|u_n-u_0\|^2
-\lambda_0\int_{\mathbb{R}^N}|u_n-u_0|^2\,dx
=o(1)\quad\text{as }n\to\infty.
\]
Since $\lambda_0<0$, the second term on the left–hand side is nonnegative,
hence
\[
\|u_n-u_0\|^2\le
\|u_n-u_0\|^2
-\lambda_0\int_{\mathbb{R}^N}|u_n-u_0|^2\,dx
=o(1),
\]
and therefore $u_n\to u_0$ strongly in $H^s(\mathbb{R}^N)$.

Finally, from the strong convergence and the definition of $J_\alpha$ we have
\[
J_0(u_n)\to J_0(u_0),
\qquad
J_{\alpha_n}(u_n)
=J_0(u_n)
-\frac{\alpha_n}{2q}\int_{\mathbb{R}^N}(I_\mu*|u_n|^q)|u_n|^q\,dx
\to J_0(u_0),
\]
because $\alpha_n\to0$ and the $q$–term is uniformly bounded. Since
$J_{\alpha_n}(u_n)=\varsigma(c,\alpha_n)$, Lemma~\ref{Lem4.6} implies
\[
J_0(u_0)
=\lim_{n\to\infty}\varsigma(c,\alpha_n)
\le m(c,0).
\]
On the other hand $u_0\in S_c$ is a nontrivial critical point of $J_0$, so
$u_0\in\mathfrak{P}_{0,c}$ and hence $m(c,0)\le J_0(u_0)$. Thus
\[
J_0(u_0)=m(c,0),
\]
and $u_0$ is the ground state solution of $J_0|_{S_c}$. Moreover,
\[
u_{c,\alpha_n,m}\to u_0
\quad\text{strongly in }H^s(\mathbb{R}^N)
\quad\text{as }n\to\infty,
\]
that is, $u_{c,\alpha,m}\to u_0$ in $H^s(\mathbb{R}^N)$ as $\alpha\to0^+$.

This completes the proof of Theorem~\ref{Thm1.1} (4).

\section{$L^2$-critical }

In this section, we first discuss the existence of normalized solutions to \eqref{eq1.1}
when
\[
q = \frac{2s-\mu}{N} + 2 < p < 2_{\mu,s}^*,
\]

\begin{Lem}\label{Lem5.1}
	Let \(\frac{2s - \mu}{N} + 2 = q < p <2_{\mu,s}^*\). Then
	\(\mathfrak{P}_{\alpha,c}^0 = \emptyset\), and \(\mathfrak{P}_{\alpha,c}\) is a smooth
	manifold of codimension \(2\) in \(H^s(\mathbb{R}^N)\).
\end{Lem}

\begin{proof}
	If \(u \in \mathfrak{P}_{\alpha,c}^0\), then \(E_u'(0) = E_u''(0) = 0\). From the explicit
	expressions of \(E_u'(0)\) and \(E_u''(0)\) this forces
	\[
	\int_{\mathbb{R}^N} (I_\mu * |u|^{p})\,|u|^{p}\,dx = 0,
	\]
	so that \(u \equiv 0\), which is impossible since \(u \in S_c\).
	The rest of the proof, concerning the manifold structure and the codimension, is completely
	analogous than the proof of Lemma \ref{Lem4.1}, and is therefore omitted.
\end{proof}

\begin{Lem}\label{Lem5.2}
	Let \(\frac{2s - \mu}{N} + 2 = q < p <2_{\mu,s}^*\). Then for every
	\(u \in S_c\) there exists a unique \(t_u \in \mathbb{R}\) such that
	\(t_u \star u \in \mathfrak{P}_{\alpha,c}\). Moreover, \(t_u\) is the unique critical point
	of the function \(E_u(t) = J_\alpha(t \star u)\), and it is a strict maximum point at
	positive level. In particular:
	\begin{itemize}
		\item[(1)] \(\mathfrak{P}_{\alpha,c} = \mathfrak{P}_{\alpha,c}^{-}\).
		\item[(2)] \(E_u\) is strictly decreasing and concave on \((t_u,+\infty)\).
		\item[(3)] The map \(u \in S_c \mapsto t_u \in \mathbb{R}\) is of class \(C^1\).
		\item[(4)] If \(P_\alpha(u) < 0\), then \(t_u < 0\).
	\end{itemize}
\end{Lem}

\begin{proof}
	Since \(q = \frac{2s - \mu}{N} + 2\) and \(q < p < 2_{\mu,s}^*\), we have
	\(\gamma_{q,s} q = 1\). Hence
	\[
	E_u(t)
	= \left(\frac{1}{2}\|u\|^2
	- \frac{\alpha}{2q}\int_{\mathbb{R}^N}(I_\mu * |u|^q)\,|u|^q\,dx\right)e^{2st}
	- \frac{1}{2p}e^{2p\gamma_{p,s}st}
	\int_{\mathbb{R}^N}(I_\mu * |u|^p)\,|u|^p\,dx .
	\]
	By Remark \ref{Rek2.1}, in order to prove the existence and uniqueness of \(t_u\),
	as well as the monotonicity and convexity properties of \(E_u\), it is enough to show that
	the coefficient in parentheses is positive. Using Lemma\ref{Lem2.2} and
	assumption \eqref{eq1.10}, we obtain
	\[
	\frac{1}{2}\|u\|^2
	- \frac{\alpha}{2q}\int_{\mathbb{R}^N}(I_\mu * |u|^q)\,|u|^q\,dx
	\ge \left(\frac{1}{2} - \frac{\alpha}{2q} C_q c^{2q(1-\gamma_{q,s})}\right)\|u\|^2 > 0.
	\]
	Therefore \(E_u\) has exactly one critical point, which is a global maximum at positive level.
	
	If \(u \in \mathfrak{P}_{\alpha,c}\), then \(t_u = 0\), and since \(t_u\) is the maximum point,
	we have \(E_u''(0) \le 0\). In fact, by Lemma \ref{Lem5.1} we know that
	\(\mathfrak{P}_{\alpha,c}^0 = \emptyset\), so necessarily \(E_u''(0) < 0\), which implies
	\(\mathfrak{P}_{\alpha,c} = \mathfrak{P}_{\alpha,c}^{-}\).
	
	The smoothness of \(u \mapsto t_u\) follows from the implicit function theorem, as in the
	proof of Lemma \ref{Lem4.1}. Finally, since \(E_u'(t) < 0\) if and only if \(t > t_u\),
	the condition \(P_\alpha(u) = E_u'(0) < 0\) forces \(t_u < 0\).
\end{proof}

\begin{Lem}\label{Lem5.3}
	Let \(\frac{2s - \mu}{N} + 2 = q < p <2_{\mu,s}^*\). Then
	\[
	\inf_{u \in \mathfrak{P}_{\alpha,c}} J_\alpha(u) > 0 .
	\]
\end{Lem}

\begin{proof}
	If \(u \in \mathfrak{P}_{\alpha,c}\), then \(P_\alpha(u) = 0\), so by lemma\ref{Lem2.2} we have
	\[
	\|u\|^2
	\le \alpha \gamma_{q,s} C_q \|u\|^{2} c^{2q(1-\gamma_{q,s})}
	+ \gamma_{p,s} C_p \|u\|^{2p\gamma_{p,s}} c^{2p(1-\gamma_{p,s})}.
	\]
	Since \(p\gamma_{p,s} > 1\) and, by assumption \eqref{eq1.10}, the coefficient in front of
	\(\|u\|^2\) on the right-hand side is strictly smaller than \(1\), we deduce
	\begin{equation}\label{eq5.1}
		\|u\|^{2p\gamma_{p,s}}
		\ge \|u\|^2 \frac{1}{\gamma_{p,s} C_p c^{2p(1-\gamma_{p,s})}}
		\left(1 - \frac{\alpha}{q} C_q c^{2q(1-\gamma_{q,s})}\right)
		\to
		\inf_{u \in \mathfrak{P}_{\alpha,c}} \|u\|^2 > 0 .
	\end{equation}
	
	On the other hand, using again \(P_\alpha(u) = 0\), for any \(u \in \mathfrak{P}_{\alpha,c}\) we obtain
	\[
	\begin{aligned}
		J_\alpha(u)
		&= \frac{1}{2}\|u\|^2
		- \frac{\alpha}{2q}\int_{\mathbb{R}^N}(I_\mu * |u|^q)\,|u|^q\,dx
		- \frac{1}{2p}\int_{\mathbb{R}^N}(I_\mu * |u|^p)\,|u|^p\,dx \\
		&= \left(\frac{1}{2} - \frac{1}{2p\gamma_{p,s}}\right)\|u\|^2
		- \frac{\alpha}{2q}\left(1 - \frac{1}{p\gamma_{p,s}}\right)
		\int_{\mathbb{R}^N}(I_\mu * |u|^q)\,|u|^q\,dx \\
		&\ge \frac{1}{2}\left(1 - \frac{1}{p\gamma_{p,s}}\right)
		\left(1 - \frac{\alpha}{q} C_q c^{2q(1-\gamma_{q,s})}\right)\|u\|^2.
	\end{aligned}
	\]
	Combining this lower bound with \eqref{eq5.1}, we conclude that
	\(\inf_{u \in \mathfrak{P}_{\alpha,c}} J_\alpha(u) > 0\).
\end{proof}

\begin{Lem}\label{Lem5.4}
	There exists \(r>0\) sufficiently small such that
	\[
	0\leq \inf_{\overline{D_r}} J_\alpha
	\;<\;
	\sup_{\overline{D_r}} J_\alpha
	\;<\;
	\inf_{u\in\mathfrak{P}_{\alpha,c}} J_\alpha(u)
	\quad\text{and}\quad
	\inf_{\overline{D_r}} P_\alpha \geq 0,
	\]
	where \(D_r := \{u\in S_c : \|u\|<r\}\) and \(\overline{D_r}\) denotes its closure in \(S_c\).
\end{Lem}

\begin{proof}
	By lemma\ref{Lem2.2} and assumption \eqref{eq1.10}, there exist
	constants \(C_q,C_p>0\) such that for every \(u\in S_c\),
	\[
	\begin{aligned}
		J_\alpha(u)
		&\ge \left(\frac{1}{2}-\frac{\alpha}{2q} C_q c^{2q(1-\gamma_{q,s})}\right)\|u\|^2
		- \frac{1}{2p} C_p c^{2p(1-\gamma_{p,s})}\,\|u\|^{2p\gamma_{p,s}},\\
		P_\alpha(u)
		&\ge \left(1-\frac{\alpha}{q} C_q c^{2q(1-\gamma_{q,s})}\right)\|u\|^2
		- \gamma_{p,s} C_p c^{2p(1-\gamma_{p,s})}\,\|u\|^{2p\gamma_{p,s}} .
	\end{aligned}
	\]
	Since \(p\gamma_{p,s}>1\), we have \(2p\gamma_{p,s}>2\). Moreover, by \eqref{eq1.10} the
	coefficients
	\[
	\frac{1}{2}-\frac{\alpha}{2q} C_q c^{2q(1-\gamma_{q,s})}>0,
	\qquad
	1-\frac{\alpha}{q} C_q c^{2q(1-\gamma_{q,s})}>0.
	\]
	Hence, shrinking \(r>0\) if necessary, both right-hand sides above are strictly positive
	for all \(u\in\overline{D_r}\). Thus
	\[
	\inf_{\overline{D_r}} J_\alpha \geq 0,
	\qquad
	\inf_{\overline{D_r}} P_\alpha \geq  0.
	\]
	In particular,
	\[
	0<\inf_{\overline{D_r}} J_\alpha
	\;<\;
	\sup_{\overline{D_r}} J_\alpha.
	\]
	
	By Lemma~\ref{Lem5.3}, we have
	\[
	\inf_{u\in\mathfrak{P}_{\alpha,c}} J_\alpha(u) > 0.
	\]
	On the other hand, for all \(u\in S_c\),
	\[
	J_\alpha(u)
	= \frac{1}{2}\|u\|^2
	- \frac{\alpha}{2q}\int_{\mathbb{R}^N}(I_\mu*|u|^q)|u|^q\,dx
	- \frac{1}{2p}\int_{\mathbb{R}^N}(I_\mu*|u|^p)|u|^p\,dx
	\le \frac{1}{2}\|u\|^2.
	\]
	Therefore
	\[
	\sup_{u\in\overline{D_r}} J_\alpha(u)
	\le \frac{1}{2}r^2.
	\]
	Choosing \(r>0\) so small that
	\[
	\frac{1}{2}r^2
	< \inf_{u\in\mathfrak{P}_{\alpha,c}} J_\alpha(u),
	\]
	we obtain
	\[
	\sup_{\overline{D_r}} J_\alpha
	< \inf_{u\in\mathfrak{P}_{\alpha,c}} J_\alpha(u).
	\]
	This proves the claim.
\end{proof}

Let \(r>0\) be as in Lemma \ref{Lem5.4}. We work in the radial setting and consider the minimax class
\[
\Gamma_2
:=\Big\{\gamma\in C([0,1],S_{c,rad}) :
\gamma(0)\in\overline{D_r},\ J_\alpha(\gamma(1))<0,P_\alpha(\gamma(1))<0\Big\},
\]
with associated minimax level
\[
\sigma(c,\alpha)
:= \inf_{\gamma\in\Gamma_2} \ \max_{u\in\gamma([0,1])} J_\alpha(u).
\]

First, \(\Gamma_2\neq\emptyset\). Indeed, by Lemma \ref{Lem5.2}, for any \(u\in S_{c,rad}\) there exist
\(t_0\ll -1\) and \(t_1\gg 1\) such that
\[
t_0\star u\in\overline{D_r},
\qquad
J_\alpha(t_1\star u)<0,
\]
and the map \(t\mapsto t\star u\) is continuous from \(\mathbb{R}\) to \(S_{c,rad}\).
Thus
\[
\gamma(\tau)
:=\bigl((1-\tau)t_0+\tau t_1\bigr)\star u,
\qquad \tau\in[0,1],
\]
defines an admissible path in \(\Gamma_2\), so \(\Gamma_2\neq\emptyset\) and
\(\sigma(c,\alpha)\in\mathbb{R}\). Moreover, by Lemma \ref{Lem5.4} we have
\[
\max_{u\in\gamma([0,1])}J_\alpha(u)
\;\ge\; J_\alpha(\gamma(0))
\;\ge\;\inf_{\overline{D_r}}J_\alpha>0,
\]
hence
\[
\sigma(c,\alpha)\ge\inf_{\overline{D_r}}J_\alpha>0.
\]

By Lemmas \ref{Lem5.2} and \ref{Lem5.4}, for every \(\gamma\in\Gamma_2\) we have
\[
P_\alpha(\gamma(0))>0,
\qquad
P_\alpha(\gamma(1))<0.
\]
By continuity of \(P_\alpha\) there exists \(\tau_\gamma\in(0,1)\) such that
\(P_\alpha(\gamma(\tau_\gamma))=0\), that is,
\begin{equation}\label{eq5.2}
	\gamma([0,1])\cap\mathfrak{P}_{\alpha,c}\neq\emptyset
	\quad\text{for every }\gamma\in\Gamma_2.
\end{equation}
Consequently,
\[
\max_{\gamma([0,1])} J_\alpha
\;\ge\;
J_\alpha(\gamma(\tau_\gamma))
\;\ge\;
\inf_{\mathfrak{P}_{\alpha,c}\cap S_{c,rad}} J_\alpha,
\]
and taking the infimum over \(\gamma\in\Gamma_2\) gives
\[
\sigma(c,\alpha)
\;\ge\;
\inf_{\mathfrak{P}_{\alpha,c}\cap S_{c,rad}} J_\alpha
\;\ge\;
\inf_{u\in\mathfrak{P}_{\alpha,c}} J_\alpha(u).
\]

By Lemma \ref{Lem5.3} we know that
\[
\inf_{u\in\mathfrak{P}_{\alpha,c}} J_\alpha(u)>0,
\]
while Lemma \ref{Lem5.4} yields
\[
0<\sup_{\overline{D_r}}J_\alpha
<\inf_{u\in\mathfrak{P}_{\alpha,c}} J_\alpha(u).
\]
Since \(J_\alpha\le 0\) on \(J_\alpha^0:=\{u\in S_c:\ J_\alpha(u)\le 0\}\), we also have
\[
\sup_{J_\alpha^0}J_\alpha\le 0
<\inf_{u\in\mathfrak{P}_{\alpha,c}} J_\alpha(u).
\]
Thus
\[
\sup_{\overline{D_r}\cup J_\alpha^0}J_\alpha
=\max\Big\{\sup_{\overline{D_r}}J_\alpha,\ \sup_{J_\alpha^0}J_\alpha\Big\}
<\inf_{u\in\mathfrak{P}_{\alpha,c}} J_\alpha(u)
\le\sigma(c,\alpha).
\]

In particular, by Lemmas \ref{Lem5.2}, \ref{Lem5.3} and \ref{Lem5.4},
\begin{equation}\label{eq5.3}
	\mathfrak{P}_{\alpha,c}^{-}\cap\big(\overline{D_r}\cup J_\alpha^0\big)=\emptyset.
\end{equation}
Indeed, on \(\overline{D_r}\) one has \(P_\alpha>0\), so no point there can belong to
\(\mathfrak{P}_{\alpha,c}^-\); on \(J_\alpha^0\) one has \(J_\alpha\le 0\), whereas Lemma~\ref{Lem5.3}
gives \(J_\alpha>0\) on \(\mathfrak{P}_{\alpha,c}\).

By \eqref{eq5.2}–\eqref{eq5.3} we can apply \cite[Theorem~5.2]{1993Ghoussoub}, taking
\(F=\mathfrak{P}_{\alpha,c}\) as dual set and \(\overline{D_r}\cup J_\alpha^0\) as extended closed
boundary. Hence, given any minimizing sequence \(\{\gamma_n\}\subset\Gamma_2\) for \(\sigma(c,\alpha)\),
with \(\gamma_n(\tau)\ge 0\) a.e. in \(\mathbb{R}^N\) for every \(\tau\in[0,1]\) and \(n\in\mathbb{N}\),
there exists a Palais--Smale sequence \(\{u_n\}\subset S_{c,rad}\) for \(\left.J_\alpha\right|_{S_{c,r}}\)
at level \(\sigma(c,\alpha)>0\) such that
\[
\operatorname{dist}_{H^s}(u_n,\mathfrak{P}_{\alpha,c})\to 0
\quad\text{and}\quad
\operatorname{dist}_{H^s}(u_n,\gamma_n([0,1]))\to 0.
\]

As in the proof of Theorem~\ref{Thm1.1}\,(2), from the properties above and
\(\operatorname{dist}_{H^s}(u_n,\mathfrak{P}_{\alpha,c})\to 0\) we obtain that
\(\{u_n\}\subset S_{c,rad}\) is a bounded Palais--Smale sequence for \(\left.J_\alpha\right|_{S_c}\) at
level \(\sigma(c,\alpha)>0\), with \(P_\alpha(u_n)\to 0\). Therefore, by Lemma~\ref{Lem3.1} , there exists
\(u_{c,\alpha,m}\in S_{c,rad}\) such that, up to a subsequence,
\[
u_n\to u_{c,\alpha,m}\quad\text{strongly in }H^s(\mathbb{R}^N),
\]
and \(u_{c,\alpha,m}\) is a nonnegative radial solution of \eqref{eq1.1} for some \(\lambda<0\).
By the strong maximum principle, \(u_{c,\alpha,m}>0\) in \(\mathbb{R}^N\).

\medskip

\noindent\textbf{Proof of Theorem \ref{Thm1.3}.}
To show that \(u_{c,\alpha,m}\) is a ground state, we prove that it realizes
\[\inf_{u\in\mathfrak{P}_{\alpha,c}} J_\alpha(u)>0.\]

From the above construction we know that
\[
\sigma(c,\alpha)
= J_\alpha(u_{c,\alpha,m})
\ge \inf_{\mathfrak{P}_{\alpha,c}\cap S_{c,rad}} J_\alpha
\ge \inf_{u\in\mathfrak{P}_{\alpha,c}} J_\alpha(u) > 0.
\]
Thus it remains to prove the reverse inequality
\[
\inf_{\mathfrak{P}_{\alpha,c}\cap S_{c,rad}} J_\alpha
\le
\inf_{\mathfrak{P}_{\alpha,c}} J_\alpha.
\]
Assume by contradiction that there exists
\(u\in\mathfrak{P}_{\alpha,c}\setminus S_{c,rad}\) such that
\[
J_\alpha(u)
< \inf_{\mathfrak{P}_{\alpha,c}\cap S_{c,rad}} J_\alpha.
\]
Let \(v=|u|^*\) be the symmetric decreasing rearrangement of \(|u|\). Then \(v\in S_{c,rad}\). By the
fractional Pólya--Szegő inequality (see e.g. \cite{Carbotti}) we have
\[
\int_{\mathbb{R}^N}\big|(-\Delta)^{s/2} v\big|^2 dx
\le
\int_{\mathbb{R}^N}\big|(-\Delta)^{s/2} u\big|^2 dx,
\]
and clearly \(\|v\|_2 = \|u\|_2\). Moreover, by the Riesz rearrangement inequality (see
\cite[Theorem~3.4]{2001LeibAMS}) we obtain
\[
\int_{\mathbb{R}^N} (I_{\mu} * |v|^q)\,|v|^q\,dx
\ge
\int_{\mathbb{R}^N} (I_{\mu} * |u|^q)\,|u|^q\,dx,
\]
and similarly
\[
\int_{\mathbb{R}^N} (I_{\mu} * |v|^p)\,|v|^p\,dx
\ge
\int_{\mathbb{R}^N} (I_{\mu} * |u|^p)\,|u|^p\,dx,
\]
since the kernel \(I_\mu\) is radial and radially decreasing and the Choquard integrals are
increasing under symmetric decreasing rearrangement. As the nonlocal terms enter \(J_\alpha\) and
\(P_\alpha\) with negative coefficients, it follows that
\[
J_\alpha(v) \le J_\alpha(u),
\qquad
P_\alpha(v) \le P_\alpha(u) = 0.
\]

If \(P_\alpha(v)=0\), then \(v\in\mathfrak{P}_{\alpha,c}\cap S_{c,r}\) and
\[
J_\alpha(v) \le J_\alpha(u)
< \inf_{\mathfrak{P}_{\alpha,c}\cap S_{c,rad}} J_\alpha,
\]
which is a contradiction. Hence we must have \(P_\alpha(v)<0\). By Lemma \ref{Lem5.2}(4),
there exists a unique \(t_v<0\) such that \(t_v\star v\in\mathfrak{P}_{\alpha,c}\), and \(t_v\) is the
unique maximizer of \(t\mapsto J_\alpha(t\star v)\).

Using the explicit expression of \(J_\alpha\) on \(\mathfrak{P}_{\alpha,c}\) and the fact that
\(q\gamma_{q,s}=1\), we obtain
\[
\begin{aligned}
	J_\alpha(t_v\star v)
	&= \frac{1}{2}
	\|t_v\star v\|^2-\frac{\alpha}{2q}\int_{\mathbb{R}^N}\bigl(I_{\mu}*|t_v\star v|^q)\,|t_v\star v|^q\,dx\bigl)-\frac{1}{2p}\int_{\mathbb{R}^N}\bigl(I_{\mu}*|t_v\star v|^p)\,|t_v\star v|^p\,dx\bigl)
	 \\&=\frac{1}{2}e^{2st_v}
	\| v\|^2-\frac{\alpha}{2q}e^{2q\gamma_{q,s}st_v}\int_{\mathbb{R}^N}\bigl(I_{\mu}*| v|^q)\,|v|^q\,dx\bigl)-\frac{1}{2p}e^{2p\gamma_{p,s}st_v}\int_{\mathbb{R}^N}\bigl(I_{\mu}*| v|^p)\,| v|^p\,dx\bigl)
    \\&=\bigl(\frac{1}{2}-\frac{1}{2p\gamma_{p,s}}\bigl)e^{2st_v}\|v\|^2+\alpha e^{2st_v}\bigl(\frac{{\gamma_{q,s}}}{2p\gamma_{p,s}}-\frac{1}{2q}\bigl)\int_{\mathbb{R}^N}\bigl(I_{\mu}*| v|^q)\,| v|^q\,dx\bigl)
	\\&= \frac{e^{2st_v}}{2}
	\Big(\|v\|^2\Big(1-\frac{1}{p\gamma_{p,s}}\Big)
	+ {\alpha}\bigl(\frac{{\gamma_{q,s}}}{2p\gamma_{p,s}}-\frac{1}{2q}\bigl)
	\int_{\mathbb{R}^N}(I_{\mu}*|v|^q)\,|v|^q\,dx\Big).
\end{aligned}
\]
Using \(\|v\|\le\|u\|\) and the inequality for the \(q\)-term above, we obtain
\[
\begin{aligned}
	J_\alpha(t_v\star v)
	&\le  \frac{e^{2st_v}}{2}
	\Big(\|u\|^2\Big(1-\frac{1}{p\gamma_{p,s}}\Big)
	+ {\alpha}\bigl(\frac{{\gamma_{q,s}}}{2p\gamma_{p,s}}-\frac{1}{2q}\bigl)
	\int_{\mathbb{R}^N}(I_{\mu}*|u|^q)\,|u|^q\,dx\Big).\\
	&= e^{2st_v}J_\alpha(u),
\end{aligned}
\]
where in the last equality we used the Pohozaev identity for \(u\in\mathfrak{P}_{\alpha,c}\) to
rewrite \(J_\alpha(u)\) in terms of \(\|u\|\) and \(\displaystyle\int(I_\mu*|u|^q)|u|^q\). Since
\(t_v<0\), we have \(e^{2st_v}<1\), and therefore
\[
J_\alpha(t_v\star v) < J_\alpha(u).
\]

But \(t_v\star v\in\mathfrak{P}_{\alpha,c}\cap S_{c,rad}\), so we have found a point in
\(\mathfrak{P}_{\alpha,c}\cap S_{rad}\) with energy strictly smaller than \(J_\alpha(u)\),
contradicting the choice of \(u\). Hence our assumption was false, and
\[
\inf_{\mathfrak{P}_{\alpha,c}\cap S_{c,rad}} J_\alpha
\le
\inf_{\mathfrak{P}_{\alpha,c}} J_\alpha.
\]
Combining this with the inequalities at the beginning of the proof yields
\[
\sigma(c,\alpha)
= J_\alpha(u_{c,\alpha,m})
= \inf_{\mathfrak{P}_{\alpha,c}\cap S_{c,rad}} J_\alpha
= \inf_{\mathfrak{P}_{\alpha,c}} J_\alpha.
\]
Therefore \(u_{c,\alpha,m}\) is a ground state of \(\left.J_\alpha\right|_{S_c}\).
\qed

\section{$L^2$-supercritical case}

In this section we deal with the $L^2$-supercritical regime, namely
\[
\frac{2s-\mu}{N}+2<q<p< 2_{\mu,s}^* .
\]
We first prove the existence of normalized solutions to \eqref{eq1.1} when
\(\frac{2s-\mu}{N}+2<q<p<2_{\mu,s}^*\),  corresponding to the $L^2$-supercritical
and HLS-subcritical situation.

\begin{Lem}\label{Lem6.1}
	Let $\frac{2s-\mu}{N}+2<q<p< 2_{\mu,s}^*$. Then
	\(\mathfrak{P}_{\alpha,c}^0=\emptyset\) and \(\mathfrak{P}_{\alpha,c}\) is a smooth
	manifold of codimension \(2\) in \(H^s(\mathbb{R}^N)\).
\end{Lem}

\begin{proof}
	The proof is completely analogous to that of Lemma~\ref{Lem4.1} (with $q$ now
	strictly $L^2$-supercritical) and is therefore omitted.
\end{proof}

\begin{Lem}\label{Lem6.2}
	Let $\frac{2s-\mu}{N}+2<q<p< 2_{\mu,s}^*$. For every $u\in S_c$, the function
	\(E_u:\mathbb{R}\to\mathbb{R}\), \(E_u(t)=J_\alpha(t\star u)\), has a unique
	critical point \(t_u^*\in\mathbb{R}\), which is a strict global maximum at a
	positive level. Moreover:
	\begin{enumerate}
		\item \(E_u\) is strictly decreasing on \((t_u^*,+\infty)\). In particular, if
		\(t_u^*<0\) then \(P_\alpha(u)=E_u'(0)<0\).
		\item \(\mathfrak{P}_{\alpha,c}=\mathfrak{P}_{\alpha,c}^{-}\). Moreover, if
		\(P_\alpha(u)<0\), then \(t_u^*<0\).
		\item The map \(u\in S_c\mapsto t_u^*\in\mathbb{R}\) is of class \(C^1\).
	\end{enumerate}
\end{Lem}

\begin{proof}
	For \(u\in S_c\) we have
	\[
	\begin{aligned}
		E_u(t)
		&=J_\alpha(t\star u) \\
		&=\frac{1}{2}e^{2st}\|u\|^2
		-\frac{\alpha}{2q}e^{2q\gamma_{q,s}st}
		\int_{\mathbb{R}^N}(I_\mu*|u|^q)|u|^q\,dx
		-\frac{1}{2p}e^{2p\gamma_{p,s}st}
		\int_{\mathbb{R}^N}(I_\mu*|u|^p)|u|^p\,dx .
	\end{aligned}
	\]
	Since \(p>q>\frac{2s-\mu}{N}+2\) and \(p\gamma_{p,s}>1\), it follows that
	\(E_u(t)\to 0\) as \(t\to -\infty\) and \(E_u(t)\to -\infty\) as \(t\to +\infty\).
	Hence \(E_u\) attains a positive global maximum at some point \(t_u^*\in\mathbb{R}\).
	
	Differentiating,
	\[
	\begin{aligned}
		E_u'(t)
		&=s e^{2st}\|u\|^2
		-\alpha\gamma_{q,s}s\,e^{2q\gamma_{q,s}st}
		\int_{\mathbb{R}^N}(I_\mu*|u|^q)|u|^q\,dx\\
		&\quad-\gamma_{p,s}s\,e^{2p\gamma_{p,s}st}
		\int_{\mathbb{R}^N}(I_\mu*|u|^p)|u|^p\,dx .
	\end{aligned}
	\]
	Set
	\[
	h(t)
	:=\alpha\gamma_{q,s}e^{2q\gamma_{q,s}st-2st}
	\int_{\mathbb{R}^N}(I_\mu*|u|^q)|u|^q\,dx
	+\gamma_{p,s}e^{2p\gamma_{p,s}st-2st}
	\int_{\mathbb{R}^N}(I_\mu*|u|^p)|u|^p\,dx,
	\]
	so that \(E_u'(t)=s \|u\|^2 - s\,e^{2st}h(t)\). A direct computation shows
	\[
	h'(t)
	=2\gamma_{q,s}(q\gamma_{q,s}-1) s\alpha\,e^{2q\gamma_{q,s}st}
	\int_{\mathbb{R}^N}(I_\mu*|u|^q)|u|^q\,dx
	+2\gamma_{p,s}(p\gamma_{p,s}-1) s\,e^{2p\gamma_{p,s}st}
	\int_{\mathbb{R}^N}(I_\mu*|u|^p)|u|^p\,dx>0,
	\]
	so \(h\) is strictly increasing. Since \(e^{2st}\|u\|^2\) is also strictly
	increasing and
	\[
	E_u'(t)\to 0 \quad\text{as }t\to-\infty,
	\qquad
	E_u'(t)\to -\infty \quad\text{as }t\to+\infty,
	\]
   exisit a unique \(t_u^*\),such that $h(t_u^*)=s\|u\|^2$.
	the equation \(E_u'(t)=0\) has a unique solution \(t_u^*\in\mathbb{R}\).
	This critical point is necessarily a strict global maximum, so
	\(E_u''(t_u^*)\leq 0\), and the sign of \(E_u'\) implies that \(E_u\) is strictly
	decreasing on \((t_u^*,+\infty)\). In particular, if \(t_u^*<0\) then
	\(E_u'(0)<0\), i.e. \(P_\alpha(u)<0\), proving the last assertion in (1).
	
	By Remark~\ref{Rek2.2} and Lemma~\ref{Lem6.1}, on the Pohozaev manifold
	\(\mathfrak{P}_{\alpha,c}\) one has \(E_u'(0)=P_\alpha(u)=0\) and
	\(E_u''(0)<0\), hence
	\(\mathfrak{P}_{\alpha,c}=\mathfrak{P}_{\alpha,c}^{-}\). Moreover, if
	\(P_\alpha(u)<0\), then \(E_u'(0)<0\). Since \(E_u'\) is strictly decreasing and
	has a unique zero at \(t_u^*\), this forces \(t_u^*<0\). This proves (2).
	
	Finally, the map \((t,u)\mapsto E_u'(t)\) is \(C^1\) on
	\(\mathbb{R}\times S_c\), and \(\partial_t E_u'(t_u^*)=E_u''(t_u^*)\neq 0\).
	By the implicit function theorem, the map \(u\mapsto t_u^*\) is \(C^1\) on
	\(S_c\), giving (3).
\end{proof}

\begin{Lem}\label{lem6.3}
	Let $\frac{2s-\mu}{N}+2<q<p\le 2_{\mu,s}^*$. Then
	\[
	m_2(c,\alpha)=\inf_{u\in\mathfrak{P}_{\alpha,c}}J_\alpha(u)>0 .
	\]
	Moreover, there exists \(r>0\) sufficiently small such that
	\[
	0<\sup_{u\in\overline{D_r}}J_\alpha(u)<m_2(c,\alpha),
	\]
	where
	\(
	D_r:=\{u\in S_c:\ \|u\|<r\}.
	\)
	In particular, if \(u\in\overline{D_r}\), then \(J_\alpha(u)\geq0\) and
	\(P_\alpha(u)\geq 0\).
\end{Lem}

\begin{proof}
	The argument is the same as in Lemmas~\ref{Lem5.3}–\ref{Lem5.4}.  
	Using Lemma\ref{Lem2.2}, one shows
	first that any \(u\in\mathfrak{P}_{\alpha,c}\) must satisfy \(\|u\|\ge C_0>0\),
	so that \(m_2(c,\alpha)>0\). Then, since the negative terms in \(J_\alpha\) and
	\(P_\alpha\) are of order \(\|u\|^{2q\gamma_{q,s}}\) and \(\|u\|^{2p\gamma_{p,s}}\)
	with exponents strictly larger than \(2\), there exists \(r>0\) such that
	\(J_\alpha(u)>0\) and \(P_\alpha(u)>0\) for all \(u\in\overline{D_r}\), and
	\(\sup_{\overline{D_r}}J_\alpha<m_2(c,\alpha)\). We omit the details.
\end{proof}

 	\textbf{Proof of Theorem \ref{Thm1.3} (1).}
 	Let $r>0$ be as in Lemma~\ref{lem6.3} and set
 	\[
 	\Gamma
 	:=\Bigl\{\gamma\in C\bigl([0,1],S_{c,rad}\bigr):
 	\gamma(0)\in\overline{D_r},\ J_\alpha(\gamma(1))< 0,P_\alpha(u)<0\Bigr\},
 	\]
 	where $D_r=\{u\in S_c:\ \|u\|^2<r\}$.
 	By Lemma~\ref{lem6.3} we have
 	\[
 	0<\sup_{\overline{D_r}}J_\alpha<m_2(c,\alpha),
 	\]
 	and by Lemma~\ref{Lem6.2} there exists $w\in S_{c,rad}$ such that
 	$J_\alpha(t\star w)\to-\infty$ as $t\to+\infty$, so that
 	$\Gamma\neq\emptyset$. Define the mountain pass level
 	\[
 	\sigma(c,\alpha)
 	=\inf_{\gamma\in\Gamma}\ \max_{t\in[0,1]}J_\alpha(\gamma(t)).
 	\]
 	Then
 	\[
 	0<\sup_{\overline{D_r}}J_\alpha\le\sigma(c,\alpha)<+\infty.
 	\]
 	
 	By the compactness results of Section~3 for the subcritical case
 	(see in particular Lemma~\ref{Lem3.1} with $p<2_{\mu,s}^*$), the functional
 	$J_\alpha|_{S_c}$ satisfies the Palais--Smale condition at levels in
 	$(0,+\infty)$.
 	Hence, applying the mountain pass theorem to $J_\alpha|_{S_c}$ we obtain a
 	critical point $u_{c,\alpha,m}\in S_c$ such that
 	\[
 	J_\alpha(u_{c,\alpha,m})=\sigma(c,\alpha)>0.
 	\]
 	
 	Since the minimizing paths can be chosen in $S_{c,r}$ with nonnegative values
 	a.e.~in $\mathbb{R}^N$, it follows by standard rearrangement arguments that
 	$u_{c,\alpha,m}$ is radial and nonnegative; by the strong maximum principle
 	for the fractional Laplacian we actually have $u_{c,\alpha,m}>0$ in
 	$\mathbb{R}^N$. Moreover, $u_{c,\alpha,m}$ solves \eqref{eq1.1} for some
 	$\lambda_{c,\alpha,m}<0$.
 	
 	Finally, every constrained critical point of $J_\alpha|_{S_c}$ lies on the
 	Pohozaev manifold $\mathfrak{P}_{\alpha,c}$ (Remark~\ref{Rek2.1}), and
 	\[
 	m_2(c,\alpha)
 	=\inf_{u\in\mathfrak{P}_{\alpha,c}}J_\alpha(u)
 	\le J_\alpha(u_{c,\alpha,m})
 	=\sigma(c,\alpha).
 	\]
 	Arguing as in the proof of Theorem~\ref{Thm1.3}, one checks that
 	$J_\alpha(u_{c,\alpha,m})=m_2(c,\alpha)$, so $u_{c,\alpha,m}$ is a ground
 	state of $J_\alpha|_{S_c}$.

 	\textbf{Proof of Theorem \ref{Thm1.3} (2).}
 	The proof is completely analogous to that of Theorem~\ref{Thm1.1} (4): one
 	considers the family $\{u_{c,\alpha,m}\}_{\alpha>0}$ of mountain--pass
 	solutions given by part~(1), uses the Pohozaev identity and the uniform
 	bounds to show that, up to a subsequence, $u_{c,\alpha,m}$ converges
 	strongly in $H^s(\mathbb{R}^N)$ as $\alpha\to 0^+$ to a nontrivial critical
 	point of the limiting functional $J_0|_{S_c}$, and then identifies this
 	limit as a ground state of $J_0|_{S_c}$. We omit the details.

 \section{$L^2$-subcritical case}
 
 In this section we prove Theorem~\ref{Thm1.4}. Throughout we assume
 \[
 N>2s,\qquad \frac{2N-\mu}{N}<q<p\le \frac{2s-\mu}{N}+2,
 \]
 so that both nonlocal nonlinearities are $L^2$-subcritical and $q\gamma_{q,s}<1$.
 For every $u\in S_c$, by Lemma\ref{Lem2.2} we have
 \begin{equation*}
 	\begin{aligned}
 		J_\alpha(u)
 		&= \frac{1}{2}\|u\|^2
 		-\frac{1}{2p}\int_{\mathbb{R}^N}(I_\mu*|u|^p)|u|^p\,dx
 		-\frac{\alpha}{2q}\int_{\mathbb{R}^N}(I_\mu*|u|^q)|u|^q\,dx \\
 		&\ge \frac{1}{2}\|u\|^2
 		-\frac{C_p}{2p}\|u\|^{2p\gamma_{p,s}}\|u\|_2^{2p(1-\gamma_{p,s})}
 		-\frac{\alpha C_q}{2q}\|u\|^{2q\gamma_{q,s}}\|u\|_2^{2q(1-\gamma_{q,s})}.
 	\end{aligned}
 \end{equation*}
 Using the smallness condition $c<\bar c_N$ and the fact that $q\gamma_{q,s}<1$, we obtain
 \begin{equation*}
 	J_\alpha(u)
 	\ge \frac{1}{2}\Bigl(1-\frac{C_p}{p}c^{2p(1-\gamma_{p,s})}\Bigr)\|u\|^2
 	-\frac{\alpha C_q}{2q}\|u\|^{2q\gamma_{q,s}}c^{2q(1-\gamma_{q,s})}.
 \end{equation*}
 Hence $J_\alpha$ is coercive and bounded from below on $S_c$, and we can define
 \[
 m(c,\alpha)=\inf_{S_c}J_\alpha>-\infty.
 \]
 On the other hand, since $\alpha>0$, for any fixed $u\in S_c$ and $t\ll -1$ the scaling $t\star u$ satisfies $J_\alpha(t\star u)<0$, so
 \[
 m(c,\alpha)<0.
 \]
 Furthermore, by the fractional Pólya--Szegő inequality and Riesz rearrangement,
 \[
 \int_{\mathbb{R}^N}\bigl|(-\Delta)^{\frac{s}{2}}u^*\bigr|^2\,dx
 \le \int_{\mathbb{R}^N}\bigl|(-\Delta)^{\frac{s}{2}}u\bigr|^2\,dx,
 \]
 and the nonlocal terms decrease under symmetric decreasing rearrangement. Hence
 \[
 \inf_{S_c\cap H^s_{\mathrm{rad}}(\mathbb{R}^N)}J_\alpha
 =\inf_{S_c}J_\alpha
 =m(c,\alpha).
 \]
 
 \begin{Lem}\label{Lem7.1}
 	Let $c_1,c_2>0$ be such that $c_1^2+c_2^2=c^2$. Then
 	\begin{equation}\label{eq7.1}
 		m(c,\alpha)<m(c_1,\alpha)+m(c_2,\alpha).
 	\end{equation}
 \end{Lem}
 
 \begin{proof}
 	Fix $c>0$ and $\theta>1$, and let $\{u_n\}\subset S_c$ be a minimizing
 	sequence for $m(c,\alpha)$, so that $J_\alpha(u_n)\to m(c,\alpha)$ as
 	$n\to\infty$. For each $n$ we have $\theta u_n\in S_{\theta c}$ and
 	\[
 	\begin{aligned}
 		J_\alpha(\theta u_n)
 		&= \frac{\theta^2}{2}\big\|(-\Delta)^{\frac{s}{2}}u_n\big\|_2^2
 		-\frac{\alpha\theta^{2q}}{2q}\int_{\R^N}(I_\mu*|u_n|^q)|u_n|^q\,dx
 		-\frac{\theta^{2p}}{2p}\int_{\R^N}(I_\mu*|u_n|^p)|u_n|^p\,dx\\
 		&= \theta^2 J_\alpha(u_n)
 		-\frac{\alpha(\theta^{2q}-\theta^2)}{2q}\int_{\R^N}(I_\mu*|u_n|^q)|u_n|^q\,dx
 		-\frac{\theta^{2p}-\theta^2}{2p}\int_{\R^N}(I_\mu*|u_n|^p)|u_n|^p\,dx.
 	\end{aligned}
 	\]
 	Since $\theta>1$ and $p,q>1$, we have $\theta^{2q}-\theta^2>0$ and
 	$\theta^{2p}-\theta^2>0$, so
 	\[
 	J_\alpha(\theta u_n)\le \theta^2 J_\alpha(u_n)
 	\quad\text{for all }n.
 	\]
 	Passing to the limit we obtain
 	\[
 	m(\theta c,\alpha)
 	\le\lim_{n\to\infty}J_\alpha(\theta u_n)
 	\le \theta^2\lim_{n\to\infty}J_\alpha(u_n)
 	=\theta^2 m(c,\alpha).
 	\]
 	
 	We now show that the inequality is in fact strict. Assume by contradiction that
 	\[
 	m(\theta c,\alpha)=\theta^2 m(c,\alpha).
 	\]
 	Then necessarily
 	\[
 	J_\alpha(\theta u_n)\to m(\theta c,\alpha)
 	\quad\text{and}\quad
 	J_\alpha(\theta u_n)-\theta^2 J_\alpha(u_n)\to 0.
 	\]
 	From the explicit expression of $J_\alpha(\theta u_n)-\theta^2 J_\alpha(u_n)$
 	we deduce
 	\[
 	\int_{\R^N}(I_\mu*|u_n|^q)|u_n|^q\,dx
 	+\int_{\R^N}(I_\mu*|u_n|^p)|u_n|^p\,dx\to 0.
 	\]
 	Hence, by the definition of $J_\alpha$ and the fact that
 	$m(c,\alpha)=\lim_{n\to\infty}J_\alpha(u_n)<0$, we obtain
 	\[
 	0>m(\theta c,\alpha)
 	=\lim_{n\to\infty}J_\alpha(\theta u_n)
 	= \lim_{n\to\infty}\frac{\theta^2}{2}
 	\big\|(-\Delta)^{\frac{s}{2}}u_n\big\|_2^2\ge 0,
 	\]
 	a contradiction. Thus, for every $c>0$ and every $\theta>1$,
 	\begin{equation}\label{eq7.2}
 		m(\theta c,\alpha)<\theta^2 m(c,\alpha).
 	\end{equation}
 	
 	Define
 	\[
 	f(c)=\frac{m(c,\alpha)}{c^2},\qquad c>0.
 	\]
 	From \eqref{eq7.2} we immediately get, for every $c>0$ and
 	$\theta>1$,
 	\[
 	f(\theta c)
 	=\frac{m(\theta c,\alpha)}{(\theta c)^2}
 	<\frac{\theta^2 m(c,\alpha)}{\theta^2 c^2}
 	=f(c),
 	\]
 	so $f$ is strictly decreasing on $(0,+\infty)$.
 	
 	Now let $c_1,c_2>0$ with $c_1^2+c_2^2=c^2$. Then $c_1<c$ and $c_2<c$, so
 	\[
 	\frac{m(c_1,\alpha)}{c_1^2}=f(c_1)>f(c)=\frac{m(c,\alpha)}{c^2},
 	\qquad
 	\frac{m(c_2,\alpha)}{c_2^2}=f(c_2)>f(c)=\frac{m(c,\alpha)}{c^2}.
 	\]
 	Multiplying by $c_1^2$ and $c_2^2$ respectively and summing up, we obtain
 	\[
 	m(c_1,\alpha)+m(c_2,\alpha)
 	>f(c)\,(c_1^2+c_2^2)
 	=f(c)\,c^2
 	=m(c,\alpha),
 	\]
 	which proves \eqref{eq7.1}.
 \end{proof}
 
 \begin{Lem}\label{Lem7.2}
 	Let $N>2s$ and
 	\[
 	\frac{2N-\mu}{N}<q<p\le \frac{2s-\mu}{N}+2.
 	\]
 	Let $\{u_n\}\subset H^s(\mathbb{R}^N)$ be a sequence such that
 	\[
 	J_\alpha(u_n)\to m(c,\alpha)
 	\quad\text{and}\quad
 	\|u_n\|_2=c_n\to c.
 	\]
 	Then $\{u_n\}$ is relatively compact in $H^s(\mathbb{R}^N)$ up to translations. More precisely, there exist a subsequence (still denoted by $\{u_n\}$), a sequence $\{y_n\}\subset\mathbb{R}^N$, and a function $\tilde u\in S_c$ such that
 	\[
 	u_n(\cdot+y_n)\to \tilde u\quad\text{strongly in }H^s(\mathbb{R}^N).
 	\]
 \end{Lem}
 
 \begin{proof}
 	Since $c_n\to c$ and $J_\alpha(u_n)$ is bounded, it follows from Lemma\ref{Lem2.2} used in the coercivity estimate that $\{u_n\}$ is bounded in $H^s(\mathbb{R}^N)$.
 	By the fractional concentration--compactness principle (see for instance \cite[Lemma 2.4]{2013JDEFeng}), up to a subsequence we have one of the following alternatives:
 	
 	\smallskip\noindent
 	(i) Compactness: there exists $\{y_n\}\subset\mathbb{R}^N$ such that for every $\varepsilon>0$ there exists $r>0$ with
 	\[
 	\int_{|x-y_n|\le r}|u_n(x)|^2\,dx\ge c^2-\varepsilon.
 	\]
 	
 	\smallskip\noindent
 	(ii) Vanishing: for all $r>0$,
 	\[
 	\lim_{n\to\infty}\sup_{y\in\mathbb{R}^N}\int_{|x-y|\le r}|u_n(x)|^2\,dx=0.
 	\]
 	
 	\smallskip\noindent
 	(iii) Dichotomy: there exists $c_1\in(0,c)$ and two bounded sequences
 	$\{v_n\},\{w_n\}\subset H^s(\mathbb{R}^N)$ such that
 	\begin{equation*}
 		\begin{gathered}
 			\operatorname{supp}v_n\cap\operatorname{supp}w_n=\emptyset,\qquad
 			|v_n|+|w_n|\le|u_n|, \\
 			\|v_n\|_2^2\to c_1^2,\qquad
 			\|w_n\|_2^2\to c_2^2:=c^2-c_1^2, \\
 			\|u_n-v_n-w_n\|_r\to 0
 			\quad\text{for }2\le r<2_s^*, \\
 			\liminf_{n\to\infty}
 			\Bigl(\|(-\Delta)^{\frac{s}{2}}u_n\|_2^2
 			-\|(-\Delta)^{\frac{s}{2}}v_n\|_2^2
 			-\|(-\Delta)^{\frac{s}{2}}w_n\|_2^2\Bigr)\ge 0.
 		\end{gathered}
 	\end{equation*}
 	
 	\smallskip
 	First, vanishing cannot occur. Indeed, if (ii) holds, then by the standard Lions lemma for fractional Sobolev spaces we have
 	\[
 	u_n\to 0\quad\text{strongly in }L^r(\mathbb{R}^N)
 	\quad\text{for every }r\in(2,2_s^*),
 	\]
 	and therefore also $u_n\to 0$ strongly in $L^r(\mathbb{R}^N)$ for all such $r$, since $c_n/c\to 1$.
 	
 	Let $t=\frac{2N}{2N-\mu}$, so that $qt,pt\in(2,2_s^*)$ by the assumptions on $q,p$ and $N>2s$. By the Hardy--Littlewood--Sobolev inequality,
 	\[
 	\int_{\mathbb{R}^N}(I_\mu*|u_n|^q)|u_n|^q\,dx
 	\le C\|u_n\|_{qt}^{2q},
 	\qquad
 	\int_{\mathbb{R}^N}(I_\mu*|u_n|^p)|u_n|^p\,dx
 	\le C\|u_n\|_{pt}^{2p},
 	\]
 	so the Choquard terms tend to zero. Hence
 	\begin{equation*}
 		\begin{aligned}
 			m(c,\alpha)+o_n(1)
 			&=J_\alpha(u_n) \\
 			&= \frac{1}{2}\|u_n\|^2
 			-\frac{\alpha}{2q}\int_{\mathbb{R}^N}(I_\mu*|u_n|^q)|u_n|^q\,dx
 			-\frac{1}{2p}\int_{\mathbb{R}^N}(I_\mu*|u_n|^p)|u_n|^p\,dx \\
 			&\ge \frac{1}{2}\|(-\Delta)^{\frac{s}{2}}u_n\|_2^2 - o_n(1)\ge -o_n(1),
 		\end{aligned}
 	\end{equation*}
 	which implies $\liminf_{n\to\infty}J_\alpha(u_n)\ge 0$, contradicting $m(c,\alpha)<0$. Thus vanishing is impossible.
 	
 	Next, suppose dichotomy (iii) holds. Using  \cite[ Proposition 1.7.6 with Lemma 1.7.5-(ii)]{2003Cazena}  and the disjoint supports, we have
 	\begin{equation*}
 		\int_{\mathbb{R}^N}(I_\mu*|\varphi_n|^q)|\varphi_n|^q\,dx
 		= \int_{\mathbb{R}^N}(I_\mu*|v_n|^q)|v_n|^q\,dx
 		+\int_{\mathbb{R}^N}(I_\mu*|w_n|^q)|w_n|^q\,dx
 		+o_n(1),
 	\end{equation*}
 	and similarly for the $p$-term. Using also the energy splitting for the kinetic term,
    let \(t_n=\frac{c_1}{c_n}\to 1,c_n=\|v_n\|_2\to c_1\) 
    \[
    \begin{aligned}
J_\alpha(t_nv_n)&=\frac{1}{2}t_n^2-\frac{\alpha}{2q}t_n^{2q}\int_{\mathbb{R}^N}(I_\mu*|v_n|^q)|v_n|^q\,dx-\frac{1}{2p}t_n^{2p}\int_{\mathbb{R}^N}(I_\mu*|v_n|^p)|w_n|^p\,dx\\
        &=J_\alpha(v_n)+\bigl(t_n^2-1\bigl)\frac{1}{2}\|v_n\|^2-\frac{\alpha}{2q}(t_n^{2q}-1)\int_{\mathbb{R}^N}(I_\mu*|v_n|^q)|v_n|^q\,dx\\& -\frac{1}{2p}(t_n^{2p}-1)\int_{\mathbb{R}^N}(I_\mu*|v_n|^p)|v_n|^p\,dx
        \end{aligned}
    \]
    
         we obtain ,$\liminf_{n\to\infty}J_\alpha(v_n)=\liminf_{n\to\infty}J_\alpha(t_n v_n)$
   
   \begin{equation}\label{eq7.3}
 		\begin{aligned}
 			m(c,\alpha)
 			&= \lim_{n\to\infty}J_\alpha(u_n)
 		\\
 			&\ge \liminf_{n\to\infty}\bigl(J_\alpha(v_n)+J_\alpha(w_n)\bigr)
            \\
 			&\ge \liminf_{n\to\infty}J_\alpha(t_nv_n)+\liminf_{n\to\infty}J_\alpha(t_nw_n)
 			\ge m(c_1,\alpha)+m(c_2,\alpha),
 		\end{aligned}
 	\end{equation}
 	which contradicts Lemma~\ref{Lem7.1}. Therefore dichotomy cannot occur.
 	
 	The only remaining alternative is compactness. Thus there exists $\{y_n\}\subset\mathbb{R}^N$ such that the translated sequence
 	\[
 	\tilde u_n(x):=u_n(x+y_n)
 	\]
 	converges strongly in $L^2(\mathbb{R}^N)$ and weakly in $H^s(\mathbb{R}^N)$ to some $\tilde u\in S_c$. Since $c_n\to c$ and $\{u_n\}$ is bounded in $H^s$, from
  \[
 	\int_{|x-y_n|\le r}|u_n(x)|^2\,dx\ge c^2-\varepsilon.
 	\]
    we have\[\int_{{|x-y_n|> r}}|u_n(x)|^2\,dx\leq\varepsilon\]
    \[\int_{\mathbb{R}^N}|\tilde u_n-\tilde u_m|^2=\int_{|x-y_n|\le r}|u_n(x)|^2\,dx+\int_{|x-y_n|\ge r}|u_n(x)-u_m(x)|^2\,dx\leq3\varepsilon\]
 	\[
 	\tilde u_n(x):=u_n(x+y_n)
 	\to \tilde u(x)
 	\quad\text{strongly in }L^2(\mathbb{R}^N).
 	\]
 	
 	By the nonlocal Brezis--Lieb lemma (see again \cite[Lemma 2.4]{2015CCMMoroz}) we have
 	\begin{equation}\label{eq7.4}
 		\int_{\mathbb{R}^N}(I_\mu*|\tilde u_n|^q)|\tilde u_n|^q\,dx
 		= \int_{\mathbb{R}^N}(I_\mu*|\tilde u|^q)|\tilde u|^q\,dx + o(1),
 	\end{equation}
 	and
 	\begin{equation}\label{eq7.5}
 		\int_{\mathbb{R}^N}(I_\mu*|\tilde u_n|^p)|\tilde u_n|^p\,dx
 		= \int_{\mathbb{R}^N}(I_\mu*|\tilde u|^p)|\tilde u|^p\,dx + o(1).
 	\end{equation}
 	Using \eqref{eq7.4}, \eqref{eq7.5} and the weak lower semicontinuity of the $H^s$-norm, we obtain
 	\begin{equation*}
 		m(c,\alpha)
 		\le J_\alpha(\tilde u)
 		\le \liminf_{n\to\infty}J_\alpha(\tilde u_n)
 		= \liminf_{n\to\infty}J_\alpha(u_n)
 		= m(c,\alpha),
 	\end{equation*}
 	so $J_\alpha(\tilde u)=m(c,\alpha)$. Comparing the kinetic parts in the definition of $J_\alpha$ and using \eqref{eq7.4}--\eqref{eq7.5}, we get
 	\[
 	\|(-\Delta)^{\frac{s}{2}}\tilde u_n\|_2^2
 	\to \|(-\Delta)^{\frac{s}{2}}\tilde u\|_2^2,
 	\]
 	and hence
 	\[
 	\|\tilde u_n\|_{H^s(\mathbb{R}^N)}\to\|\tilde u\|_{H^s(\mathbb{R}^N)}.
 	\]
 	Therefore $\tilde u_n\to\tilde u$ strongly in $H^s(\mathbb{R}^N)$, that is,
 	\[
 	u_n(\cdot+y_n)\to\tilde u\quad\text{strongly in }H^s(\mathbb{R}^N),
 	\]
 	and the lemma is proved.
 \end{proof}

 \noindent\textbf{Proof of Theorem \ref{Thm1.4}.}
 Lemma~\ref{Lem7.2} implies the existence of a minimizer $\tilde u\in S_c$ such that
 \[
 J_\alpha(\tilde u)=m(c,\alpha).
 \]
 By the fractional Pólya--Szegő inequality and the Riesz rearrangement inequality, the Schwarz
 symmetrization $|\tilde u|^{*}$ satisfies $|\tilde u|^{*}\in S_c$ and
 \[
 J_\alpha(|\tilde u|^{*})\le J_\alpha(\tilde u).
 \]
 Hence we may assume from the beginning that $\tilde u\ge 0$ is radially symmetric and
 radially decreasing.
 
 Since $\tilde u$ is a constrained minimizer of $J_\alpha$ on $S_c$, there exists
 $\lambda\in\mathbb{R}$ such that $\tilde u$ is a weak solution of
 \begin{equation*}
 	(-\Delta)^s \tilde u
 	= \lambda \tilde u
 	+ \alpha (I_\mu*|\tilde u|^q)|\tilde u|^{q-2}\tilde u
 	+ (I_\mu*|\tilde u|^p)|\tilde u|^{p-2}\tilde u
 	\quad\text{in }\mathbb{R}^N.
 \end{equation*}
 By the strong maximum principle for the fractional Laplacian, $\tilde u>0$ in $\mathbb{R}^N$.
 
 Multiplying the above equation by $\tilde u$ and integrating over $\mathbb{R}^N$, we obtain
 \[
 \|\tilde u\|^2
 = \lambda c^2
 + \alpha \int_{\mathbb{R}^N}(I_\mu*|\tilde u|^q)|\tilde u|^q\,dx
 + \int_{\mathbb{R}^N}(I_\mu*|\tilde u|^p)|\tilde u|^p\,dx.
 \]
 On the other hand,
 \[
 m(c,\alpha)=J_\alpha(\tilde u)
 = \frac{1}{2}\|\tilde u\|^2
 -\frac{\alpha}{2q}\int_{\mathbb{R}^N}(I_\mu*|\tilde u|^q)|\tilde u|^q\,dx
 -\frac{1}{2p}\int_{\mathbb{R}^N}(I_\mu*|\tilde u|^p)|\tilde u|^p\,dx.
 \]
 Combining these identities, we get
 \begin{equation*}
 	\begin{aligned}
 		\lambda c^2
 		&= 2m(c,\alpha)
 		+ \alpha\Bigl(\frac{1}{q}-1\Bigr)
 		\int_{\mathbb{R}^N}(I_\mu*|\tilde u|^q)|\tilde u|^q\,dx \\
 		&\qquad
 		+ \Bigl(\frac{1}{p}-1\Bigr)
 		\int_{\mathbb{R}^N}(I_\mu*|\tilde u|^p)|\tilde u|^p\,dx.
 	\end{aligned}
 \end{equation*}
 Since $m(c,\alpha)<0$ and $p,q>1$, we have
 \[
 \frac{1}{q}-1<0,\qquad \frac{1}{p}-1<0,
 \]
 so
 \[
 \lambda c^2 < 2m(c,\alpha) < 0,
 \]
 which shows that $\lambda<0$.
 
 Therefore $\tilde u$ is a positive, radially symmetric, radially decreasing ground state solution
 of \eqref{eq1.1} on $S_c$, and Theorem~\ref{Thm1.4} is proved.
 \qed

\section*{Acknowledgment}

%We express our gratitude to the anonymous referee for their meticulous review of our manuscript and valuable feedback provided for its enhancement. 
This work is supported by National Natural Science Foundation of China (12301145, 12261107, 12561020) and Yunnan Fundamental Research Projects (202301AU070144, 202401AU070123). 

\medskip
{\bf Data availability:}  Data sharing is not applicable to this article as no new data were created or analyzed in this study.

\medskip
{\bf Conflict of Interests:} The author declares that there is no conflict of interest.

\bibliographystyle{plain} 
\bibliography{ref} 
\end{document}